\RequirePackage{fix-cm}

\documentclass[smallcondensed]{svjour3}     
\usepackage{graphicx}
\usepackage{mathptmx}      
\usepackage{amssymb}
\usepackage{amsmath}
\usepackage{subfigure}
\usepackage{color}
\renewcommand{\thefigure}{\thesection.\arabic{figure}}
\renewcommand{\thetable}{\thesection.\arabic{table}}

\begin{document}

\title{A Positivity-preserving High Order Finite Volume
Compact-WENO Scheme for Compressible Euler Equations
}


\author{Yan Guo  \and
           Tao Xiong
           \and
         Yufeng Shi 
}


\institute{Yan Guo \at
              Department of Mathematics, China University of Mining and Technology, Xuzhou, Jiangsu 221116, P.R. China. \\
              \email{yanguo@cumt.edu.cn}           
          \and
          Tao Xiong \at
          Department of Mathematics, University of Houston, Houston, Texas 77004, U.S.A. \\
          \email{txiong@math.uh.edu}
           \and
           Yufeng Shi \at
              School of Electric Power Engineering, China University of Mining and Technology, Xuzhou, Jiangsu 221116, P.R. China.\\
              \email{shiyufeng@cumt.edu.cn}
}

\date{Received: date / Accepted: date}

\maketitle

\begin{abstract}
In this paper, a positivity-preserving fifth-order finite
volume compact-WENO scheme is proposed for solving compressible Euler equations.
As we know, conservative compact finite volume
schemes have high resolution properties while WENO (Weighted
Essentially Non-Oscillatory) schemes are essentially non-oscillatory
near flow discontinuities. We extend the idea of
WENO schemes to some classical finite volume compact schemes \cite{pirozzoli2002conservative},
where lower order compact stencils are combined with WENO nonlinear
weights to get a higher order finite volume compact-WENO scheme.
The newly developed positivity-preserving limiter \cite{zhang2010positivity,zhang2011maximum}
is used to preserve positive density and internal energy for compressible Euler equations of fluid
dynamics. The HLLC (Harten, Lax, and van Leer with Contact)
approximate Riemann solver \cite{toro2009riemann,batten1997choice} is used to get the numerical flux
at the cell interfaces. Numerical tests are presented to
demonstrate the high-order accuracy, positivity-preserving,
high-resolution and robustness of the proposed scheme.

\end{abstract}

\section{Introduction}
\label{intro}
\setcounter{equation}{0}
\setcounter{figure}{0}
\setcounter{table}{0}

Computing  numerical solutions of nonlinear hyperbolic systems of
conservation laws is an interesting and challenging work.  In recent
years, a variety of high resolution schemes which are high order
accurate for smooth solutions and non-oscillatory for discontinuous
solutions without introducing spurious oscillations
have been proposed for these problems. WENO schemes \cite{liu1994weighted,jiang1995efficient,shu1998essentially,shu1989efficient,balsara2000monotonicity} have high order accuracy in smooth region and keep the essentially non-oscillatory properties
for capturing shocks. However, these classical WENO schemes often suffer
from poor spectral resolution and excessive numerical dissipation.

Compact schemes \cite{lele1992compact} have attracted a lot of
attention due to its spectral-like resolution properties by using
global grids. These schemes have the features of high-order accuracy with
smaller stencils. However, linear compact schemes necessarily produce Gibbs-like
oscillations when they are directly applied to
flows with shock discontinuities, and the amplitude would not decrease with mesh
refinement. To address this difficulty, several hybrid compact schemes are proposed to couple the ENO 
or WENO schemes for shock-turbulence interaction problems, e.g., a hybrid compact-ENO scheme by Adams and Shariff \cite{adams1996high} and a hybrid compact-WENO scheme by Pirozzoli \cite{pirozzoli2002conservative}. A new hybrid scheme as a weighted average of the compact 
scheme \cite{pirozzoli2002conservative} and the WENO scheme \cite{jiang1995efficient} was developed by
Ren et. al. \cite{ren2003characteristic}. Another compact scheme by treating the discontinuity as an internal boundary was proposed by Shen et. al. \cite{shen2006high}. These hybrid schemes require indicators to detect discontinuities and switch to a non-compact scheme around discontinuities, spectral-like
resolution properties would be lost. 

A class of nonlinear compact schemes was proposed by Cockburn and Shu \cite{cockburn1994nonlinearly}
for shock calculations. It was based on the cell-centered compact schemes \cite{lele1992compact} and combined with TVD or TVB limiters to control spurious numerical oscillations. Deng and Maekawa \cite{deng1997compact} and Deng and Zhang \cite{deng2000developing} developed a class of nonlinear compact schemes based on the ENO and WENO ideas respectively by adaptively choosing candidate stencils. Zhang et. al. \cite{zhang2008development} proposed increasingly higher order compact schemes based on higher order WENO reconstructions \cite{balsara2000monotonicity}. Instead of interpolating the conservative variables, they directly interpolated the flux by using the Lax-Friedrichs flux splitting and characteristic-wise projections. An improvement of the compact scheme converging to steady-state solutions of Euler equations was studies in \cite{zhang2013improvement}. A new linear central compact scheme was proposed in \cite{liu2013new}, both grid points and half grid points are evolved to get higher order accuracy and
better resolutions. 

Jiang et. al. \cite{jiang2001weighted} developed a class of weighted
compact schemes based on the Pad\'{e} type scheme of Lele \cite{lele1992compact}. 
It is a weighted combination of two biased third order compact stencils and 
a central fourth order compact stencil. A sixth order central compact scheme can be obtained
in smooth regions. Recently Ghosh and Baeder employed the idea in
\cite{jiang2001weighted}, and developed a class of
compact-reconstruction finite difference WENO schemes \cite{ghosh2012compact}. 
Lower order biased compact candidate stencils are identified at the cell interface and
combined with the optimal nonlinear WENO weights. The resulting high order scheme 
is upwind. Their scheme was shown to be superior spectral accurate and non-oscillatory 
at discontinuities. 


In this paper, we consider to design finite volume high order compact
schemes for solving compressible Euler equations.
A conservative formulation of the Euler equations is given by
\begin{equation}\label{1DEuler}
U_t+F(U)_x=0,
\end{equation}
where $U$ and $F(U)$ are vectors of conservative variables and
fluxes respectively, which are given by
\begin{displaymath}
U=\left[\begin{array}{ccc}
 u_1\\u_2\\u_3
  \end{array}\right]
 =\left[\begin{array}{ccc}
 \rho\\\rho u\\E
  \end{array}\right]
  ,\quad F(U)
 =\left[\begin{array}{ccc}
 \rho u\\\rho u^2+p\\u(E+p)
\end{array}\right],
\end{displaymath}
with
\begin{equation}
E=\rho(\frac{1}{2}u^2+e), \quad e=e(\rho,p)=\frac{p}{(\gamma-1)\rho},
\end{equation}
where $\rho$ is the density, $p$ is the pressure, $u$ is the particle velocity ,
$E$ is the total energy per unit volume, $e$ is the specific internal
energy and $\gamma$ is the ratio of specific heat ($\gamma = 1.4$ for ideal gas).
The sound speed $a$ is defined as
\begin{equation}\label{sspeed}
a=\sqrt{\frac{\gamma p}{\rho}}.
\end{equation}
Physically, the density $\rho$ and the pressure $p$ should both be
positive, and failure of preserving positive density or
pressure may cause blow-up of the numerical solutions. Many first order 
schemes were shown to be positivity-preserving, such as Godunov-type 
schemes \cite{einfeldt1991godunov}, flux vector splitting schemes \cite{gressier1999positivity},
Lax-Friedrichs schemes
\cite{perthame1996positivity,zhang2010positivity}, HLLC schemes
\cite{batten1997choice} and gas-kinetic schemes
\cite{perthame1990boltzmann,tao1999gas}. Some second-order schemes
were also developed based on these first order schemes, such as
\cite{einfeldt1991godunov,tao1999gas,perthame1996positivity,estivalezes1996high}. 
Recently Zhang and Shu have developed positivity-preserving methods for high order
discontinuous Galerkin (DG) methods
\cite{zhang2010positivity,zhang2011positivity,zhang2012maximum},
finite volume and finite difference WENO schemes \cite{zhang2011maximum,zhang2012positivity}.
Self-adjusting and positivity preserving high order schemes were developed by Balsara for MHD equations \cite{balsara2012self}. Hu et. al. have developed positivity-preserving
high-order conservative schemes  by using a flux cut-off method for solving compressible Euler
equations \cite{hu2013positivity}. Xiong et. al have developed a parametrized positivity preserving flux limiters for finite difference schemes solving compressible Euler equations \cite{xiong2013parametrized}.

In the present paper, we will develop a conservative positivity-preserving
fifth-order finite volume compact-WENO (FVCW) scheme 
for compressible Euler equations. We employ the main idea in
\cite{ghosh2012compact} where lower order compact stencils are
combined with the optimal WENO weights to yield a fifth-order upwind compact
interpolation. As an alternative to the finite difference compact
interpolation in \cite{ghosh2012compact}, we design a finite volume compact
upwind scheme, which is more nature and can be easily used on unstructured meshes.
We also employ the newly developed positivity-preserving rescaling limiter in \cite{zhang2010positivity,zhang2011maximum} to preserve positive density and internal energy, which
is very important in some extreme cases, such as vacuum or near vacuum solutions. 
The HLLC approximate Riemann solver \cite{toro2009riemann,batten1997choice}
will be used as the numerical flux at the element interfaces
due to its less dissipation and robustness for solving compressible Euler equations.
The first order finite volume scheme with the HLLC flux is proved to preserve
positive density and internal energy.
We will show that the high order finite volume compact scheme with the positivity preserving rescaling limiter, can maintain high order accuracy similarly as the non-compact finite volume schemes. Numerical experiments will be presented to demonstrate the high spectral accuracy, high resolution, positivity-preserving and robustness of our proposed approach.

The rest of the paper is organized as follows. In
Section 2, the positivity-preserving finite volume
compact-WENO scheme for compressible Euler equations is presented.
Numerical tests for some benchmark problems of compressible Euler equations
are studied in Section 3. Conclusions are made in Section 4.

\section{Positivity-preserving finite volume compact-WENO scheme}
\label{sec2}
\setcounter{equation}{0}
\setcounter{figure}{0}
\setcounter{table}{0}

\subsection{Finite volume scheme for compressible Euler equations}
\label{sec2.1}
In this section, we first introduce the finite volume scheme \cite{leveque2004finite} for compressible Euler equations
(\ref{1DEuler}) . The computational domain $[a,b]$ is
divided into $N$ cells as follows
\begin{equation*}
a=x_{\frac{1}{2}}<x_{\frac{3}{2}}<\cdots<x_{N+\frac{1}{2}}=b.
\end{equation*}
The cells are denoted by $I_j=[x_{j-\frac{1}{2}},x_{j+\frac{1}{2}}]$ with the cell center $x_j=\frac12(x_{j-\frac{1}{2}}+x_{j+\frac{1}{2}})$ and
the cell size $\Delta
x_j=x_{j+\frac{1}{2}}-x_{j-\frac{1}{2}}$. If we integrate equation (\ref{1DEuler}) over cell $I_j$,
we obtain
\begin{equation}\label{1DEulerFV}
\frac{\partial}{\partial{t}}\int^{x_{j+\frac{1}{2}}}_{x_{j-\frac{1}{2}}}Udx+
F(U(x_{j+\frac{1}{2}},t))-F(U(x_{j-\frac{1}{2}},t))=0.
\end{equation}
The cell average of $I_j$ is defined as
\begin{equation}
\bar{U}_j=\frac{1}{\Delta
{x_j}}\int^{x_{j+\frac{1}{2}}}_{x_{j-\frac{1}{2}}}U(x,t)dx,
\end{equation}
and the finite volume conservative scheme for (\ref{1DEulerFV}) is
\begin{equation}\label{1DEulerFV-b}
\frac{d\bar{U}_j(t)}{dt}=-\frac{1}{\Delta
{x_j}}(\hat{F}_{j+\frac{1}{2}}-\hat{F}_{j-\frac{1}{2}}),
\end{equation}
where the numerical flux $\hat{F}_{j+\frac{1}{2}}$ is a vector function of
mass, momentum and total energy at the cell boundary and is defined by
\begin{equation}\label{flux}
\hat{F}_{j+\frac{1}{2}}=\hat{F}(U^{-}_{j+\frac{1}{2}},U^{+}_{j+\frac{1}{2}}).
\end{equation}
In this paper, $U^{-}_{j+\frac{1}{2}}$ and $U^{+}_{j+\frac{1}{2}}$ are obtained
from a high order compact-WENO reconstruction, which will be
discussed in the following subsections.

\subsection{Compact-WENO reconstruction for scalars}
\label{sec2.2} 

For simplicity, we consider uniform grids with cell
size $\Delta x_j=h=\frac{b-a}{N}, \forall j$ in this paper. We first
review the finite volume compact reconstruction proposed in
\cite{pirozzoli2002conservative}. For a scalar variable $u(x)$, a compact representation 
around the grid node $x_{j+\frac{1}{2}}$ can be written as
\begin{equation}\label{compactrec}
\sum_{l=-L_1}^{L_2}\alpha_l\tilde{u}_{j+\frac{1}{2}+l}=\sum_{m=-M_1}^{M_2}a_m\bar{u}_{j+m},
\end{equation}
where $\tilde{u}_{j+\frac{1}{2}}$ denotes the reconstruction value of $u(x)$ at the grid node $x_{j+\frac12}$. Assuming that the
function $u$ can be expanded by Taylor series up to $K$-th order around
$x_{j+\frac{1}{2}}$
\begin{equation}\label{Taylor}
u(x)=\sum_{n=0}^{K-1}u_{j+\frac{1}{2}}^{(n)}\frac{(x-x_{j+\frac{1}{2}})^n}{n!}+O(h^K),
\end{equation}
we have
\begin{equation}
\tilde{u}_{j+\frac{1}{2}+l}=\sum_{n=0}^{K-1}u_{j+\frac{1}{2}}^{(n)}
                           \frac{l^n}{n!}h^n+O(h^K),
\end{equation}
\begin{equation}
\bar{u}_{j+m}=\sum_{n=0}^{K-1}u_{j+\frac{1}{2}}^{(n)}\frac{1}{(n+1)!}
             [m^{n+1}-(m-1)^{n+1}]h^n+O(h^K).
\end{equation}
A fifth-order compact upwind scheme in this class is for
$K=5$, which can yield the following scheme by taking $L_1=L_2=M_1=M_2=1$ in (\ref{compactrec}),
\begin{equation}\label{5th-order1}
 \frac{3}{10}\tilde{u}_{j-\frac{1}{2}}+\frac{6}{10}\tilde{u}_{j+\frac{1}{2}}+\frac{1}{10}\tilde{u}_{j+\frac{3}{2}}
=\frac{1}{30}\bar{u}_{j-1}+\frac{19}{30}\bar{u}_{j}+\frac{10}{30}\bar{u}_{j+1}.
\end{equation}
Symmetrically, we also have
\begin{equation}\label{5th-order2}
 \frac{1}{10}\tilde{u}_{j-\frac{1}{2}}+\frac{6}{10}\tilde{u}_{j+\frac{1}{2}}+\frac{3}{10}\tilde{u}_{j+\frac{3}{2}}
=\frac{10}{30}\bar{u}_{j}+\frac{19}{30}\bar{u}_{j+1}+\frac{1}{30}\bar{u}_{j+2}.
\end{equation}

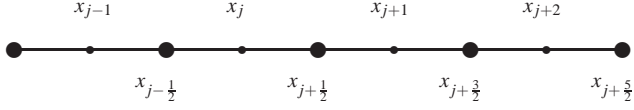
\begin{figure}
\setlength{\unitlength}{1mm}
\begin{picture}(10,20)
  \put(10,10){\line(10,0){80}}
  \multiput(10,10)(20,0){5}{\circle*{2}}
  \multiput(20,10)(20,0){4}{\circle*{1}}
  \put(18,15){$x_{j-1}$}
  \put(38,15){$x_{j}$}
  \put(57,15){$x_{j+1}$}
  \put(77,15){$x_{j+2}$}
  \put(26,5){$x_{j-\frac{1}{2}}$}
  \put(46,5){$x_{j+\frac{1}{2}}$}
  \put(66,5){$x_{j+\frac{3}{2}}$}
  \put(86,5){$x_{j+\frac{5}{2}}$}
\end{picture}\caption{Candidate stencils for interior points}\label{Fig-1}
\end{figure}

These classical fifth order linear finite volume compact schemes
(\ref{5th-order1}) and (\ref{5th-order2}) based on smaller stencils 
are very accurate and keep good resolutions in smooth regions, but
unacceptable non-physical oscillations are generated when they are
directly applied to problems with discontinuities and
the amplitude would not decrease as the grid nodes are refined. 

In the following, we adopt the main idea of \cite{ghosh2012compact} to form
a nonlinear finite volume compact-WENO scheme. 
For a fifth order finite volume compact-WENO scheme, 
three third-order compact stencils will be used as candidates, as shown in Fig.\ref{Fig-1}.
From (\ref{compactrec}), for the three candidate stencils, we have
\begin{equation}
\label{candidates}
\begin{aligned}
&\frac{2}{3}{u}_{j-\frac{1}{2}}^{(0)}+\frac{1}{3}{u}_{j+\frac{1}{2}}^{(0)}=\frac{1}{6}(\bar{u}_{j-1}+5\bar{u}_{j}),&\\
&\frac{1}{3}{u}_{j-\frac{1}{2}}^{(1)}+\frac{2}{3}{u}_{j+\frac{1}{2}}^{(1)}=\frac{1}{6}(5\bar{u}_{j}+\bar{u}_{j+1}),&\\
&\frac{2}{3}{u}_{j+\frac{1}{2}}^{(2)}+\frac{1}{3}{u}_{j+\frac{3}{2}}^{(2)}=\frac{1}{6}(\bar{u}_{j}+5\bar{u}_{j+1}).&\\
\end{aligned}
\end{equation}
Given the cell averages $\{\bar{u}_j\}$, a nonlinear weighted combination of (\ref{candidates}) will result in 
\begin{equation}
\begin{aligned}\label{CWENO-FV}
&\frac{2\omega_0+\omega_1}{3}\tilde{u}_{j-\frac{1}{2}}+
\frac{\omega_0+2(\omega_1+\omega_2)}{3}\tilde{u}_{j+\frac{1}{2}}
+\frac{1}{3}\omega_2\tilde{u}_{j+\frac{3}{2}}&\\
=&\frac{1}{6}\omega_0\bar{u}_{j-1}+\frac{5(\omega_0+\omega_1)+\omega_2}{6}\bar{u}_j
+\frac{\omega_1+5\omega_2}{6}\bar{u}_{j+1},&
\end{aligned}
\end{equation}
where the nonlinear weights $\{\omega_0, \omega_1, \omega_2\}$ will be specified later. Let $u_{j+\frac{1}{2}}^{-}$ denote the fifth order approximation
of the nodal value $u(x_{j+\frac{1}{2}},t^n)$ in cell $I_j$. From (\ref{CWENO-FV}), a fifth order
compact-WENO approximation of $u_{j+\frac{1}{2}}^{-}$ based on the stencil
$\{x_{j-1},x_{j},x_{j+1}\}$ is given by
\begin{equation}
u_{j+\frac{1}{2}}^{-}=\tilde{u}_{j+\frac{1}{2}}.
\end{equation}

In smooth regions, the finite volume compact-WENO scheme
yields a fifth-order upwind compact scheme
\cite{pirozzoli2002conservative}. To construct a nonlinear compact scheme, we choose 
a set of normalized nonlinear weights $\omega_k$ \cite{borges2008improved,castro2011high} by taking
\begin{equation}
\label{weight}
\omega_k=\frac{\alpha_k^z}{\sum_{l=0}^2\alpha_l^z},\quad
\alpha_k^z=c_k\left(1+\left(\frac{\tau_5}{\beta_k+\epsilon}\right)^2
\right), \quad k=0,1,2,
\end{equation}
where $\tau_5=|\beta_{2}-\beta_{0}|$ and
the classical smooth indicators $\beta_k$ $(k=0,1,2)$ \cite{shu1998essentially} are given by
\begin{equation*}
\begin{aligned}
&\beta_{0}=\frac{13}{12}(\bar{u}_{j-2}-2\bar{u}_{j-1}+\bar{u}_{j})^2+\frac{1}{4}(\bar{u}_{j-2}-4\bar{u}_{j-1}+3\bar{u}_{j})^2,&\\
&\beta_{1}=\frac{13}{12}(\bar{u}_{j-1}-2\bar{u}_{j}+\bar{u}_{j+1})^2+\frac{1}{4}(\bar{u}_{j-1}-\bar{u}_{j+1})^2,&\\
&\beta_{2}=\frac{13}{12}(\bar{u}_{j}-2\bar{u}_{j+1}+\bar{u}_{j+2})^2+\frac{1}{4}(3\bar{u}_{j}-4\bar{u}_{j+1}+\bar{u}_{j+2})^2.&
\end{aligned}
\end{equation*}
$\epsilon$ is a small positive number to avoid the denominator
to be $0$, in our numerical tests, we take $\epsilon=10^{-13}$. The optimal
linear weights are $c_0=\frac{2}{10},c_1=\frac{5}{10},c_2=\frac{3}{10}$.
The weights \eqref{weight} are denoted as WENO-Z weights, which can avoid accuracy
lost at critical points \cite{borges2008improved}. 


For the scalar case, a tri-diagonal system (\ref{CWENO-FV}) is solved to get $u_{j+\frac{1}{2}}^{-}$.
Let $u_{j+\frac{1}{2}}^{+}$ denote the fifth order approximation
of the nodal value $u(x_{j+\frac{1}{2}},t^n)$ from cell $I_{j+1}$, following a similar procedure as above, it can be obtained by the stencil $\{x_{j},x_{j+1},x_{j+2}\}$. Similar to classical WENO schemes, near critical points, the corresponding weight approaches to $0$ and the system reduces to a biased bidiagonal system. Across the discontinuities, the fifth-order scheme yields a third-order compact scheme which has higher resolution than a third order non-compact scheme.

\subsection{Compact-WENO reconstruction for systems}
\label{sec2.3} In this subsection, we will describe the finite
volume compact-WENO reconstruction for compressible Euler equations.
The scalar algorithm (\ref{CWENO-FV}) in the previous subsection will be applied along each characteristic
field. As we know, the conservative Euler equations (\ref{1DEuler}) can also be written in
a quasi-linear form \cite{toro2009riemann}
\begin{equation}
U_t+A(U)U_x=0,
\end{equation}
where the coefficient matrix $A(U)$ is the Jacobian matrix of $F(U)$ and can
be written as
\begin{displaymath}
A(U)=\left[ \begin{array}{ccc}
0&1&0\\
-\frac{1}{2}(\gamma-3)(\frac{u_2}{u_1})^2&(3-\gamma)(\frac{u_2}{u_1})&\gamma-1\\

-\frac{\gamma u_2u_3}{u_1^2}+(\gamma-1)(\frac{u_2}{u_1})^3&
\frac{\gamma u_3}{u_1}-\frac{3}{2}(\gamma-1)(\frac{u_2}{u_1})^2
&\gamma (\frac{u_2}{u_1})
\end{array} \right].
\end{displaymath}
The total specific enthalpy $H$ is related to the
specific enthalpy $h$, they are
\begin{equation}
H=\frac{E+p}{\rho}\equiv \frac{1}{2}u^2+h,\quad h=e+\frac{p}{\rho}.
\end{equation}
The eigenvalues of the Jacobian matrix $A(U)$ are
\begin{equation}
\lambda_1=u-a,\quad \lambda_2=u, \quad \lambda_3=u+a,
\label{eigen}
\end{equation}
where $a$ is the speed of sound (\ref{sspeed}). The corresponding
right eigenvectors are
\begin{displaymath}\label{right-ev}
r^{(1)}=\left[\begin{array}{ccc}
 1\\u-a\\H-ua
  \end{array}\right], \quad
r^{(2)}=\left[\begin{array}{ccc}
 1\\u\\\frac{1}{2}u^2
  \end{array}\right], \quad
  r^{(3)}=\left[\begin{array}{ccc}
 1\\u+a\\H+ua
  \end{array}\right].
\end{displaymath}
A matrix $R(U)$ is formed by the right eigenvectors
\begin{equation}
  R(U)=(r^{(1)},r^{(2)},r^{(3)}).
  \label{reigenv}
\end{equation}
Letting $L(U)=R(U)^{-1}$, then
\begin{equation*}
L(U)A(U)R(U)=\Lambda,
\end{equation*}
here $\Lambda$ is the diagonal matrix
$\Lambda=diag(\lambda_1,\lambda_2,\lambda_3)$. Denoting a vector $l^{(k)}$ to
be the $k$-th row in $L(U)$, then
\begin{equation}
\begin{aligned}
&l^{(1)}=\frac{1}{2}(c_2+u/a,-c_1u-1/a,c_1),&\\
&l^{(2)}=(1-c_2,c_1u,-c_1),&\\
&l^{(3)}=\frac{1}{2}(c_2-u/a,-c_1u+1/a,c_1),&
\end{aligned}
\label{leigenv}
\end{equation}
where $c_1=(\gamma-1)/a^2$, $c_2=\frac{1}{2}u^2c_1$.

At the grid node $x_{j+\frac{1}{2}}$,
denoting $U_{j+\frac{1}{2}}^{-}$ as the fifth order approximation
of the nodal values $U(x_{j+\frac{1}{2}},t^n)$ at time $t^n$ within
the cells $I_j$, the scalar finite volume compact-WENO
reconstruction (\ref{CWENO-FV}) is applied to each component of the
characteristic variables $\bar{V}_j=L(U^{Roe}_{j+\frac12})\bar{U}_j$ to obtain
$U_{j+\frac{1}{2}}^{-}$, where $U^{Roe}_{j+\frac12}$ denotes the Roe-average
of the cell-average values $\bar U_j$ and $\bar U_{j+1}$ \cite{toro2009riemann}.

For the systems, a characteristic-wise finite volume compact-WENO scheme
consists of the following steps:
\begin{enumerate}
\item At each grid node $x_{j+\frac12}$, computing the eigenvalues (\ref{eigen}) and
eigenvectors (\ref{reigenv}) and (\ref{leigenv}) by using $U^{Roe}_{j+\frac12}$.
\item Along each characteristic field, computing the weights (\ref{weight}) from characteristic 
variables $\bar{V}_j=L(U^{Roe}_{j+\frac12})\bar{U}_j$.
\item Applying the scalar reconstruction (\ref{CWENO-FV}) at each characteristic field
\begin{equation}\label{cha-1}
\begin{aligned}
a_{j+\frac{1}{2}}^{(k)}l_{j+\frac{1}{2}}^{(k)}\tilde{U}_{j-\frac{1}{2}}+
b_{j+\frac{1}{2}}^{(k)}l_{j+\frac{1}{2}}^{(k)}\tilde{U}_{j+\frac{1}{2}}+
c_{j+\frac{1}{2}}^{(k)}l_{j+\frac{1}{2}}^{(k)}\tilde{U}_{j+\frac{3}{2}}\\=
d_{j+\frac{1}{2}}^{(k)}l_{j+\frac{1}{2}}^{(k)}\bar{U}_{j-1}+
e_{j+\frac{1}{2}}^{(k)}l_{j+\frac{1}{2}}^{(k)}\bar{U}_{j}+
f_{j+\frac{1}{2}}^{(k)}l_{j+\frac{1}{2}}^{(k)}\bar{U}_{j+1}
\end{aligned}
\end{equation}
for $k=1,2,3$. The coefficients $a_{j+\frac{1}{2}}^{(k)},
b_{j+\frac{1}{2}}^{(k)}, c_{j+\frac{1}{2}}^{(k)},
d_{j+\frac{1}{2}}^{(k)}, e_{j+\frac{1}{2}}^{(k)},
f_{j+\frac{1}{2}}^{(k)}$ corresponding to the coeffcients in (\ref{CWENO-FV}),
which can be obtained from Step 2.
\item Rewriting the equation (\ref{cha-1}) to be
\begin{equation}\label{1DEuler-FVC}
 A_{j+\frac{1}{2}}\tilde{U}_{j-\frac{1}{2}}+
 B_{j+\frac{1}{2}}\tilde{U}_{j+\frac{1}{2}}+
 C_{j+\frac{1}{2}}\tilde{U}_{j+\frac{3}{2}}=
 D_{j+\frac{1}{2}}\bar{U}_{j-1}+
 E_{j+\frac{1}{2}}\bar{U}_{j}+
 F_{j+\frac{1}{2}}\bar{U}_{j+1}
\end{equation}
where
\begin{displaymath}
A_{j+\frac{1}{2}}=\left[\begin{array}{ccc}
 a_{j+\frac{1}{2}}^{(1)}l_{j+\frac{1}{2}}^{(1)}\\
 a_{j+\frac{1}{2}}^{(2)}l_{j+\frac{1}{2}}^{(2)}\\
 a_{j+\frac{1}{2}}^{(3)}l_{j+\frac{1}{2}}^{(3)}
  \end{array}\right],
B_{j+\frac{1}{2}}=\left[\begin{array}{ccc}
 b_{j+\frac{1}{2}}^{(1)}l_{j+\frac{1}{2}}^{(1)}\\
 b_{j+\frac{1}{2}}^{(2)}l_{j+\frac{1}{2}}^{(2)}\\
 b_{j+\frac{1}{2}}^{(3)}l_{j+\frac{1}{2}}^{(3)}
  \end{array}\right],
C_{j+\frac{1}{2}}=\left[\begin{array}{ccc}
c_{j+\frac{1}{2}}^{(1)}l_{j+\frac{1}{2}}^{(1)}\\
c_{j+\frac{1}{2}}^{(2)}l_{j+\frac{1}{2}}^{(2)}\\
c_{j+\frac{1}{2}}^{(3)}l_{j+\frac{1}{2}}^{(3)}
  \end{array}\right],
\end{displaymath}
\begin{displaymath}
D_{j+\frac{1}{2}}=\left[\begin{array}{ccc}
 d_{j+\frac{1}{2}}^{(1)}l_{j+\frac{1}{2}}^{(1)}\\
 d_{j+\frac{1}{2}}^{(2)}l_{j+\frac{1}{2}}^{(2)}\\
 d_{j+\frac{1}{2}}^{(3)}l_{j+\frac{1}{2}}^{(3)}
  \end{array}\right],
E_{j+\frac{1}{2}}=\left[\begin{array}{ccc}
 e_{j+\frac{1}{2}}^{(1)}l_{j+\frac{1}{2}}^{(1)}\\
 e_{j+\frac{1}{2}}^{(2)}l_{j+\frac{1}{2}}^{(2)}\\
 e_{j+\frac{1}{2}}^{(3)}l_{j+\frac{1}{2}}^{(3)}
  \end{array}\right],
F_{j+\frac{1}{2}}=\left[\begin{array}{ccc}
f_{j+\frac{1}{2}}^{(1)}l_{j+\frac{1}{2}}^{(1)}\\
f_{j+\frac{1}{2}}^{(2)}l_{j+\frac{1}{2}}^{(2)}\\
f_{j+\frac{1}{2}}^{(3)}l_{j+\frac{1}{2}}^{(3)}
  \end{array}\right].
\end{displaymath}
Noticing that $l_{j+\frac{1}{2}}^{(k)}$ for $k=1,2,3$ are all vectors,
a $3\times 3$ block tri-diagonal system (\ref{1DEuler-FVC}) is solved
by using the chasing method \cite{gong1997lu} to obtain $\tilde{U}_{j+\frac{1}{2}}$.
\end{enumerate}
From (\ref{1DEuler-FVC}), a fifth order
compact-WENO approximation of $U_{j+\frac{1}{2}}^{-}$ based on the stencil
$\{x_{j-1},x_{j},x_{j+1}\}$ is given by
\begin{equation}
U_{j+\frac{1}{2}}^{-}=\tilde{U}_{j+\frac{1}{2}}.
\end{equation}
Letting $U_{j+\frac{1}{2}}^{+}$ denote the fifth order approximation
of the nodal value $U(x_{j+\frac{1}{2}},t^n)$ from cell $I_{j+1}$, 
following a similar procedure as above, it can be obtained by the stencil $\{x_{j},x_{j+1},x_{j+2}\}$.

\subsection{Positivity-preserving and HLLC approximate Riemann solver}
For compressible Euler equations, the Riemann solutions
consist of a contact wave and two acoustic waves, either may be
a shock or a rarefaction wave. In \cite{godunov1959difference},
Godunov presented a first-order upwind scheme which could capture
shock waves without introducing nonphysical spurious oscillations.
The important part of the Godunov-type method is the exact or
approximate solutions of the Riemann problem. Exact solutions to the
Riemann problem is difficult or too expensive to be obtained. 
Approximate Riemann solvers are often used to build Godunov-type
numerical schemes. The HLLC approximate Riemann solver
\cite{toro2009riemann,batten1997choice} has been proved to be very
simple, reliable and robust. In \cite{batten1997choice}, Batten et
al. proposed an appropriate choice of the acoustic wavespeeds
required by HLLC and proved that the resulting numerical method
resolves isolated shock and contact waves exactly, and is positively
conservative which will be reviewed in the following.

For the HLLC flux, two averaged states
$U^{*}_l, U^{*}_r$ between the two acoustic waves $S_L, S_R$ are
considered, which are separated by the contact wave whose speed is
denoted by $S_M$. The approximate Riemann solution with two states $U_l$ and $U_r$ is
defined as
\begin{equation}
U^{HLLC}=\left\{\begin{array}{ll}
U_l, &\quad \textrm{if $S_L>0$},\\
U^{*}_l, &\quad \textrm{if $S_L\leq 0<S_M$},\\
U^{*}_r, &\quad \textrm{if $S_M\leq 0\leq S_R$},\\
U_r, &\quad \textrm{if $S_R<0$}.
\end{array}\right.
\end{equation}
The corresponding flux is
\begin{equation}\label{HLLC}
\hat F^{HLLC}(U_l, U_r)=\left\{\begin{array}{ll}
F_l, &\quad \textrm{if $S_L>0$},\\
F^{*}_l=F_l+S_L(U^{*}_l-U_l), &\quad \textrm{if $S_L\leq 0<S_M$},\\
F^{*}_r=F_r+S_r(U^{*}_r-U_r), &\quad \textrm{if $S_M\leq 0\leq S_R$},\\
F_r, &\quad \textrm{if $S_R<0$}.
\end{array}\right.
\end{equation}
where $F_l=F(U_l)$ and $F_r=F(U_r)$, similarly for the following variables with subscripts $l$ and $r$.

To determine $U^{*}_l$, the following assumption has been made \cite{batten1997choice}
\begin{equation}
S_M=u^{*}_l=u^{*}_r=u^{*}.
\end{equation}
which gives the contact wave velocity
\begin{equation}\label{contact-v}
S_M=\frac{\rho_ru_r(S_R-u_r)-\rho_lu_l(S_L-u_l)+p_l-p_r}
         {\rho_l(S_R-u_r)-\rho_l(S_L-u_l)}.
\end{equation}
and
\begin{equation}\label{star-l}
\left\{\begin{array}{l}
 \rho^{*}_l=\rho_l\frac{S_L-u_l}{S_L-S_M},\\
 p^{*}=\rho_l(u_l-S_L)(u_l-S_M)+p_l,\\
 \rho^{*}_lu^{*}_l=\frac{(S_L-u_l)\rho_lu_l+(p^{*}-p_l)}
                     {S_L-S_M},\\
E^{*}_l=\frac{(S_L-u_l)E_l-p_lu_l+p^{*}S_M}
                     {S_L-S_M}.
\end{array}\right.
\end{equation}
The right star state can be obtained symmetrically.

To make the scheme preserving positivity, the acoustic wavespeeds
are computed from
\begin{equation}\label{acoustic ws}
S_L=\min{[u_l-a_l,\tilde{u}^{*}-\tilde{a}^{*}]}, \quad
S_R=\min{[u_r+a_r,\tilde{u}^{*}+\tilde{a}^{*}]},
\end{equation}
where
\begin{equation}
\left\{\begin{array}{l}
\tilde{u}^{*}=\frac{u_l+u_rR_\rho}{1+R_\rho},\\
\tilde{a}^{*}=\sqrt{(\gamma-1)[\tilde{H}^{*}-\frac{1}{2}\tilde{u}^{*2}]},\\
\tilde{H}^{*}=\frac{(H_l+H_rR_\rho)}{1+R_\rho},\\
R_\rho=\sqrt{\frac{\rho_r}{\rho_l}}.
\end{array}\right.
\end{equation}
Defining the set of physically realistic states as those with positive
densities and internal energies by
\begin{equation}
G=\left\{\begin{array}{l}
U=\left[\begin{array}{ccc}
\rho\\
\rho u\\
E \end{array}\right],\rho>0,e=\frac{E}{\rho}-\frac{u^2}{2}>0
\end{array}\right\},
\end{equation}
then $G$ is a convex set \cite{batten1997choice}.

We now consider a first order finite volume scheme
\begin{equation}\label{1DEulerFV-First}
\bar{U}^{n+1}_{j}=\bar{U}^{n}_{j}-
\lambda[\hat{F}(\bar{U}^{n}_{j},\bar{U}^{n}_{j+1})-
\hat{F}(\bar{U}^{n}_{j-1},\bar{U}^{n}_{j})],
\end{equation}
where $\hat{F}(\cdot,\cdot)$ is a HLLC flux and
$\lambda=\frac{\Delta t}{h}$. For a positively conservative scheme
(\ref{1DEulerFV-First}), if $\bar{U}^{n}_{j}, j=1,\cdots,N$, is
contained in $G$, then $\bar{U}^{n+1}_{j}, j=1,\cdots,N$, will also
lie inside $G$. This would be guaranteed by proving the intermediate
states $U^{*}_l \in G$ if we have $U_l \in G$, and proving $U^{*}_r \in G$ if we have $U_r \in
G$, because $G$ is a convex set (for details see \cite{batten1997choice}). 

In the following,
we will show the left star state $U^{*}_l \in G$, while similar arguments hold for
the right star state $U^{*}_r \in G$. That is, when $U_l \in G$, which is equivalent to
\begin{equation}
\rho_l>0, \quad E_l-\frac{1}{2}\rho_lu_l^2>0,
\end{equation}
we will have
\begin{equation}
\rho^{*}_l>0,
\end{equation}
and
\begin{equation}\label{E-star}
E^{*}_l-\frac{1}{2}\rho_l^{*}{u_l^{*}}^2>0.
\end{equation}
From (\ref{star-l}), we can get
\begin{equation}
\rho^{*}_l=\rho_l\frac{S_L-u_l}{S_L-S_M}.
\end{equation}
$S_M$ in (\ref{contact-v}) is an averaged velocity, from (\ref{acoustic ws}) we have
\begin{equation}\label{inequ-v}
S_L<S_M, \quad S_L<u_l,
\end{equation}
and $\rho^{*}_l>0$ is easily obtained. Using relations (\ref{star-l}) and
(\ref{inequ-v}), (\ref{E-star}) can be rewritten as
\begin{equation}
(u_l-S_L)E_l+p_lu_l-p^{*}S_M+
\frac{((S_L-u_l)\rho_lu_l-p_l+p^{*})^2}{2\rho_l(S_L-u_l)}>0,
\end{equation}
which is equivalent to
\begin{equation}
\frac{1}{2}\rho_l(S_M-u_l)^2-p_l\frac{S_M-u_l}{u_l-S_L}+\frac{p_l}{\gamma-1}>0.
\end{equation}
To guarantee this inequality for any value of $S_M-u_l$, the
discriminant of the above quadratic function of $S_M-u_l$ should be
negative, which gives the following condition
\begin{equation}
\frac{p_l^2}{(u_l-S_L)^2}-2\rho_l^2e_l<0,
\end{equation}
that is
\begin{equation}
S_L<u_l-\frac{p_l}{\rho_l\sqrt{2e_l}}=
u_l-\sqrt{\frac{\gamma-1}{2\gamma}}\sqrt{\frac{\gamma p_l}{\rho_l}}
= u_l-\sqrt{\frac{\gamma-1}{2\gamma}}a_l,
\end{equation}
which is always satisfied with acoustic wavespeeds (\ref{acoustic ws}).

\begin{remark}
As in \cite{zhang2010positivity}, if we consider the first order finite volume scheme (\ref{1DEulerFV-First}) with the HLLC flux (\ref{HLLC}) and averaged intermediate states (\ref{star-l}) for solving the compressible Euler equations (\ref{1DEuler}), this first order scheme is positivity-preserving with the choice of the acoustic wavespeeds (\ref{acoustic ws}) and under the following CFL condition
\begin{equation}
\lambda \left\| |u|+a \right\|_\infty \leq1.
\end{equation}
\end{remark}

Now to design a positivity-preserving fifth-order finite
volume compact-WENO scheme,
we first consider the Euler forward time discretization for equation
(\ref{1DEulerFV-b})
\begin{equation}\label{1DEulerFV-c}
\bar{U}^{n+1}_j=\bar{U}^{n}_j-\lambda(\hat{F}(U_{j+\frac{1}{2}}^{-},U_{j+\frac{1}{2}}^{+})-
\hat{F}(U_{j-\frac{1}{2}}^{-},U_{j-\frac{1}{2}}^{+}),
\end{equation}
where $\hat{F}$ is the HLLC flux,
$U_{j+\frac{1}{2}}^{-}$ and $U_{j+\frac{1}{2}}^{+}$ are obtained by
using the compact-WENO reconstructions in Section \ref{sec2.3}. We
employ the idea in \cite{zhang2011maximum,zhang2010positivity} to construct high
order finite volume compact-WENO schemes to preserve positive
density and pressure or internal energy for the Euler system.

We consider a polynomial vector $Q_j(x)=(\rho_j(x)),(\rho u)_j(x),E_j(x))^{T}$
with degree $K (K\geq 2)$ on $I_j$, such that
\begin{equation}
\label{polyQ}
U_{j-\frac{1}{2}}^{+}=Q_j(x_{j-\frac{1}{2}}), \quad
U_{j+\frac{1}{2}}^{-}=Q_j(x_{j+\frac{1}{2}}), \quad
\bar{U}_{j+q}=\frac{1}{\Delta x}\int_{x_{j+q-\frac{1}{2}}}^{x_{j+q+\frac{1}{2}}}Q_j(x)dx,
\end{equation}
here $q$ is related to $K$, for example, if $K=5$, $q=-1, 0, 1$.
By using the
$M$-point Gauss-Lobatto quadrature rule on $I_j$ and choose
the quadrature points as
$S_j=\{x_{j-\frac{1}{2}}=\hat{x}_j^1,\cdots,\hat{x}_j^M=x_{j+\frac{1}{2}}\}$,
a sufficient condition for $\bar{U}_j^{n+1} \in G$ is
$Q_j(\hat{x}_j^{\alpha})\in G$ for $\alpha=1,2,\cdots,M$, under a
suitable CFL condition. We denote $\hat{\omega}_\alpha$ as the
Legendre Gauss-Lobatto quadrature weights on the interval
$[-\frac{1}{2},\frac{1}{2}]$, and
$\sum_{\alpha=1}^M\hat{\omega}_\alpha=1$ with $2M-3\leq K$.
Following \cite{zhang2010positivity}, we have
\begin{theorem}
Consider the high order ($K\ge2$) finite volume compact-WENO scheme (\ref{1DEulerFV-c}) with the HLLC flux (\ref{HLLC}) for solving the compressible Euler equations (\ref{1DEuler}). The first order scheme (\ref{1DEulerFV-First}) with the HLLC flux would be positivity-preserving under the condition (\ref{acoustic ws}) with the averaged intermediate states (\ref{star-l}). If the reconstructed polynomial vector $Q_j(x)=(\rho_j(x)),(\rho u)_j(x),E_j(x)$ (\ref{polyQ}) satisfies $Q_j(\hat{x}_j^{\alpha})\in G$, $\forall j$, the scheme (\ref{1DEulerFV-c}) is positivity-preserving ($\bar{U}_j^{n+1}\in G$) under the the CFL condition
\begin{equation}\label{CFL-CD}
\lambda \left\| |u|+a \right\|_\infty \leq \omega_1.
\end{equation}
\end{theorem}

\begin{proof}
The proof is similar to that in \ \cite{zhang2010positivity}. By using the $M$-point Gauss-Lobatto rule, the cell average $\bar{U}_j$ can be written as
\begin{equation}
\bar{U}_j=\frac{1}{\Delta x}\int_{x_{j-\frac{1}{2}}}^{x_{j+\frac{1}{2}}}Q_j(x)dx=
\sum_{\alpha=1}^{M}\hat{\omega}_\alpha Q_j(\hat{x}_j^{\alpha}).
\end{equation}
Noticing that $U_{j-\frac{1}{2}}^{+}=Q_j(\hat{x}_j^1)$ and $U_{j+\frac{1}{2}}^{-}=Q_j(\hat{x}_j^M)$, $\forall j$, the scheme (\ref{1DEulerFV-c}) can be rearranged as follows
 \begin{eqnarray*}
 \bar{U}_j^{n+1}&=&
\sum_{\alpha=1}^{M}\hat{\omega}_\alpha Q_j(\hat{x}_j^{\alpha})-
\lambda(\hat{F}(U_{j+\frac{1}{2}}^{-},U_{j+\frac{1}{2}}^{+})-\hat{F}(U_{j-\frac{1}{2}}^{+},U_{j+\frac{1}{2}}^{-})
+\hat{F}(U_{j-\frac{1}{2}}^{+},U_{j+\frac{1}{2}}^{-})-\hat{F}(U_{j-\frac{1}{2}}^{-},U_{j-\frac{1}{2}}^{+}))\\
&=&\sum_{\alpha=2}^{M-1}\hat{\omega}_\alpha Q_j(\hat{x}_j^{\alpha})+
\hat{\omega}_1\left(U_{j-\frac{1}{2}}^{+}-\frac{\lambda}{\hat{\omega}_1}
[\hat{F}(U_{j-\frac{1}{2}}^{+},U_{j+\frac{1}{2}}^{-})-
\hat{F}(U_{j-\frac{1}{2}}^{-},U_{j-\frac{1}{2}}^{+})]\right)\\
&&+
\hat{\omega}_M\left(U_{j+\frac{1}{2}}^{-}-\frac{\lambda}{\hat{\omega}_M}
[\hat{F}(U_{j+\frac{1}{2}}^{-},U_{j+\frac{1}{2}}^{+})-
\hat{F}(U_{j-\frac{1}{2}}^{+},U_{j+\frac{1}{2}}^{-})]\right)\\
&=&\sum_{\alpha=2}^{M-1}\hat{\omega}_\alpha Q_j(\hat{x}_j^{\alpha})+
\hat{\omega}_1H_1+\hat{\omega}_MH_M,
 \end{eqnarray*}
  where
 \begin{eqnarray*}
 H_1&=&U_{j-\frac{1}{2}}^{+}-\frac{\lambda}{\hat{\omega}_1}
[\hat{F}(U_{j-\frac{1}{2}}^{+},U_{j+\frac{1}{2}}^{-})-
\hat{F}(U_{j-\frac{1}{2}}^{-},U_{j-\frac{1}{2}}^{+})],\\
H_M&=&U_{j+\frac{1}{2}}^{-}-\frac{\lambda}{\hat{\omega}_M}
[\hat{F}(U_{j+\frac{1}{2}}^{-},U_{j+\frac{1}{2}}^{+})-
\hat{F}(U_{j-\frac{1}{2}}^{+},U_{j+\frac{1}{2}}^{-})].
 \end{eqnarray*}
The above two equations are both of the form
(\ref{1DEulerFV-First}), therefore $H_1$ and $H_M$ are in the set $G$ due to $U_{j+\frac{1}{2}}^{\pm}\in G, \forall j$ and the CFL condition (\ref{CFL-CD}) with the HLLC flux (\ref{HLLC}) and the acoustic wavespeeds (\ref{acoustic ws}). Now $\bar{U}_j^{n+1}\in G$ is proved since it is a convex combination of $H_1$, $H_M$ and $Q_j(\hat{x}_j^{\alpha})$ for $2\le \alpha \le M-1$, which are all in $G$.
\end{proof}

Similar to the approach in \cite{zhang2011maximum,zhang2010positivity}, the positivity-preserving limiter for the present scheme in the one-dimensional space will be constructed. The easy-implementation algorithm of WENO schemes in \cite{zhang2011maximum} will be adopted:
\begin{enumerate}
\item Set up a small positive parameter $\varepsilon=\min_j\{10^{-13}, \bar{\rho}_j^n\}$.
\item Compute the limiter
\begin{equation}
\theta_1=\min\left\{\frac{\bar{\rho}_j^n-\varepsilon}{\bar{\rho}_j^n-\rho_{min}},1\right\},
\end{equation}
where $\rho_{min}=\{\rho_{j+\frac{1}{2}}^{-},
\rho_{j-\frac{1}{2}}^{+}, \rho_{j}(x_j^{1*})\}$ and
\begin{equation}
\rho_j(x_j^{1*})=\frac{\bar{\rho}_j^n-\hat{\omega}_1\rho_{j-\frac{1}{2}}^{+}-
\hat{\omega}_M\rho_{j+\frac{1}{2}}^{-}}{1-2\hat{\omega}_1}.
\end{equation}
\item Modify the density by letting
\begin{equation}\label{limiter-rho}
\hat{\rho}_j(x)=\theta_1(\rho_j(x)-\bar{\rho}_j^n) +\bar{\rho}_j^n.
\end{equation}
Get $\hat{\rho}_{j+\frac{1}{2}}^{-}$ and
$\hat{\rho}_{j-\frac{1}{2}}^{+}$ from
\begin{eqnarray*}
&\hat{\rho}_{j+\frac{1}{2}}^{-}=\theta_1(\rho_{j+\frac{1}{2}}^{-}-\bar{\rho}_j^n)
+\bar{\rho}_j^n,\\
&\hat{\rho}_{j-\frac{1}{2}}^{+}=\theta_1(\rho_{j-\frac{1}{2}}^{+}-\bar{\rho}_j^n)
+\bar{\rho}_j^n.
\end{eqnarray*}
Denote
\begin{eqnarray*}
\hat{W}_j^1=\hat{U}_{j+\frac{1}{2}}^{-},\quad
\hat{W}_j^2=\hat{U}_{j-\frac{1}{2}}^{+}, \quad
\hat{W}_j^3=\frac{\bar{U}_j^n-\hat{\omega}_1\hat{U}_{j-\frac{1}{2}}^{+}-
\hat{\omega}_M\hat{U}_{j+\frac{1}{2}}^{-}}{1-2\hat{\omega}_1}.
\end{eqnarray*}
\item
Get $\theta_2=\min_{\alpha=1,2,3}t_{\varepsilon}^{\alpha}$ from modifying the internal energy:

For $\alpha=1, 2, 3$:
\begin{itemize}
\item if $e(\hat{W}_j^{\alpha}) < \varepsilon$, solve the following quadratic equations for $t_{\varepsilon}^{\alpha}$ as in \cite{zhang2010positivity}
    \begin{equation}
    e[(1-t_{\varepsilon}^{\alpha})\bar{U}_j^n+t_{\varepsilon}^{\alpha}\hat{W}_j^{\alpha}]=\varepsilon
    \end{equation}
\item If $e(\hat{W}_j^{\alpha})\geq \varepsilon$, let
$t_{\varepsilon}^{\alpha}=1$.
\end{itemize}

Denote
\begin{eqnarray*}
\tilde{U}_{j+\frac{1}{2}}^{-}=\theta_2(\hat{U}_{j+\frac{1}{2}}^{-}-\bar{U}_j^n)+\bar{U}_j^n,\quad
\tilde{U}_{j-\frac{1}{2}}^{+}=\theta_2(\hat{U}_{j-\frac{1}{2}}^{+}-\bar{U}_j^n)+\bar{U}_j^n.
\end{eqnarray*}
\item The scheme (\ref{1DEulerFV-c}) with the positivity-preserving limiter would be
\begin{equation}\label{1DEuler-FVCW}
\bar{U}^{n+1}_j=\bar{U}^{n}_j-\lambda(\hat{F}(\tilde{U}^{-}_{j+\frac{1}{2}},\tilde{U}^{+}_{j+\frac{1}{2}})-
\hat{F}(\tilde{U}^{-}_{j-\frac{1}{2}},\tilde{U}^{+}_{j-\frac{1}{2}}).
\end{equation}
\end{enumerate}

\begin{remark}
To prove that the limiter will not destroy the high order accuracy of density for
smooth solutions, for a fifth order scheme, we need to show
$\hat{\rho}_j(x)-\rho_j(x)=O(\Delta x^5)$ in (\ref{limiter-rho}).
In the present compact scheme, although $\rho_{j+\frac{1}{2}}^{-}$
and $\rho_{j-\frac{1}{2}}^{+}$ are obtained globally, which are different
from those in \cite{zhang2011maximum,zhang2010positivity}, the constructed
polynomial $\rho_j(x)$ from (\ref{polyQ}) can be seen locally. Thus, the proof of preserving
high order accuracy of density is similar to that
in \cite{zhang2011maximum,zhang2010positivity}. Similar arguments hold for the internal
energy. So the scheme (\ref{1DEuler-FVCW}) is conservative, high order accurate
and positivity preserving.
\end{remark}

\subsection{Temporal discretization}
Strong stability preserving (SSP) high order Runge-Kutta time discretization \cite{gottlieb2009high} will be used to improve the temporal accuracy for the scheme (\ref{1DEuler-FVCW}). The third-order SSP Runge-Kutta method is
\begin{equation}\label{RK3}
\begin{aligned}
&U^{(1)}=U^n+\Delta tL(U^n), & \\
&U^{(2)}=\frac{3}{4}U^n+\frac{1}{4}U^{(1)}+\frac{1}{4}\Delta tL(U^{(1)}), & \\
&U^{n+1}=\frac{1}{3}U^n+\frac{2}{3}U^{(2)}+\frac{2}{3}\Delta
tL(U^{(2)}),
\end{aligned}
\end{equation}
where $L(U)$ is the spatial operator. Similar to
\cite{zhang2010positivity}, for SSP high order time discretizations,
the limiter will be used at each stage on each time step.

\section{Numerical examples}
\label{sec3}
\setcounter{equation}{0}
\setcounter{figure}{0}
\setcounter{table}{0}

In this section, we will investigate the numerical performance of
the present positivity-preserving fifth-order finite volume
compact-WENO (FVCW) scheme. The fifth-order WENO scheme \cite{castro2011high} will be denoted as ``WENO-Z''
and the original fifth order WENO scheme of Jiang and Shu \cite{jiang1995efficient} is denoted as ``WENO-JS''.  We will compare the FVCW scheme to WENO-JS and WENO-Z schemes. For all the numerical tests, the third-order SSP Runge-Kutta method (\ref{RK3}) is used under the CFL condition (\ref{CFL-CD}) unless
otherwise specified. The numerical solutions are computed with $N$ grid nodes and up to time $t$.

\begin{example}\label{Eg:dpertu}
Advection of density perturbation. The initial conditions for
density, velocity and pressure are specified, respectively, as
\begin{equation*}
\rho(x,0) =1+0.2sin(\pi x), \quad u(x,0)=1, \quad p(x,0)=1.
 \end{equation*}
The exact solution of density is $\rho(x,t) =1+0.2sin(\pi (x-t))$.

The computational domain is $[0,2]$ and the boundary condition
is periodic. The $L_1$, $L_2$ and $L_{\infty}$ errors and
orders at $t=2$ for the present finite volume
compact-WENO scheme are shown in Table \ref{Tab:1}. Here the time step
is taken to be $\Delta t=\frac{\omega_1}{\||u|+a\|}h^{5/3}$. We can clearly
observe fifth-order accuracy for this problem.

In this example with smooth exact solutions, we also compare the computational cost between the
FVCW scheme and the WENO-JS scheme. As we know, the FVCW scheme has high resolutions, however, 
a $3\times3$ block tri-diagonal system (\ref{1DEuler-FVC}) needs to be solved at each grid node $x_{j+\frac12}$ and at each stage of each time step. This might be computationally expensive. 
However, we will demonstrate by this example that the FVCW scheme would still be more efficient.
Two kinds of reconstructions for systems are considered. One is based on a characteristic
variable reconstruction, the other is directly reconstructing on the conservative variables. We take relatively coarser grids and choose the time step to satisfy $\lambda \||u|+a\| = 0.16$, so that the spatial error would always dominate.
In Table \ref{Tab:2}, we show the computational cost between the FVCW scheme and the WENO-JS scheme
for the conservative variable reconstruction case. For this case, without characteristic decomposition, only tri-diagonal (not block tri-diagonal) systems need to be solved along each component, less CPU cost would be needed. We can see at a comparable $L_1$ error level, the computational cost for the FVCW scheme is much less than the WENO-JS scheme especially when the error is small, which can also be seen from Fig.\ref{Fig:1} (left), where the comparison of the CPU cost versus the $L_1$ errors is displayed. Similarly in Table \ref{Tab:3} and Fig. \ref{Fig:1} (right) for the characteristic variable case, we can also observe less computational cost for the FVCW scheme when it has comparable error to the WENO-JS scheme. Similar discussions can be found in \cite{ghosh2012compact}. We note that the FVCW scheme with conservative variable reconstruction is more efficient than the characteristic variable reconstruction for smooth solutions. However for discontinuous solutions, the characteristic variable reconstruction would perform better to control
spurious numerical oscillations. In this paper, for the following examples, we will mainly adopt the characteristic variable reconstruction.


\begin{table}
\caption{Numerical errors and orders for Example \ref{Eg:dpertu}.}\label{Tab:1}
\begin{tabular}{lllllll}
\hline \noalign{\smallskip}
    N  &  $L_1$ error &$L_1$ Order&$L_{\infty}$ error& $L_{\infty}$ Order &$L_{2}$ error & $L_{2}$ Order \\
\hline \noalign{\smallskip}
    10 &    7.802E-04 & &    6.506E-04 & &    5.874E-04 & \\
    20 &    1.493E-05 &     5.71 &    1.716E-05 &     5.24 &    1.263E-05 &     5.54 \\
    40 &    3.260E-07 &     5.52 &    2.942E-07 &     5.87 &    2.625E-07 &     5.59 \\
    80 &    9.107E-09 &     5.16 &    9.117E-09 &     5.01 &    7.162E-09 &     5.20 \\
   160 &    2.695E-10 &     5.08 &    2.903E-10 &     4.97 &    2.113E-10 &     5.08 \\
   320 &    8.169E-12 &     5.04 &    9.202E-12 &     4.98 &    6.413E-12 &     5.04 \\
\noalign{\smallskip} \hline
 \end{tabular}
\end{table}

\begin{table}
\caption{Numerical errors and computational cost for WENO-JS and
FVCW schemes for Example \ref{Eg:dpertu}. Conservative variable reconstruction.}\label{Tab:2}
\begin{center}
\begin{tabular}{lllll|lllll}\hline
       \multicolumn{5}{c|}{FVCW} &\multicolumn{5}{c}{WENO-JS}\\\hline
   N  &  $L_1$ error &$L_{\infty}$ error& $L_{2}$ error & CPU cost ($s$) &
   N  &  $L_1$ error &$L_{\infty}$ error& $L_{2}$ error & CPU cost ($s$)            \\\hline
   7  &     3.780E-03 & 2.796E-03          &   2.939E-03      &1.56E-002 &
   15 &    2.236E-03 & 1.899E-03          &   1.792E-03       & 3.13E-02     \\
   14 &    7.819E-05 & 8.125E-05          &    6.366E-05      & 3.12E-02  &
   30 &    7.510E-05 & 7.187E-05          &    6.295E-05      & 0.11           \\
   28 &    2.065E-06 & 1.537E-06          &    1.579E-06      & 0.14         &
   60 &    2.352E-06 & 2.353E-06          &    1.919E-06      & 0.47          \\
   56 &    5.879E-08 & 4.699E-08          &    4.511E-08      & 0.58         &
   120 &  7.336E-08 & 7.082E-08          &    5.878E-08      & 1.88          \\
   112 &  1.945E-09 & 1.482E-09          &    1.506E-09      & 2.22         &
   240 &  2.280E-09 & 2.022E-09          &    1.824E-09      & 7.55          \\
   224 &  8.882E-11 & 6.897E-11          &    6.926E-11      & 8.86         &
   480 &  6.977E-11 & 5.909E-11          &    5.541E-11      & 29.95         \\\hline
 \end{tabular}
 \end{center}
\end{table}

\begin{table}
\caption{Numerical errors and computational cost for WENO-JS and
FVCW schemes for Example \ref{Eg:dpertu}. Characteristic variable reconstruction.}\label{Tab:3}
\begin{center}
\begin{tabular}{lllll|lllll}\hline
       \multicolumn{5}{c|}{FVCW} &\multicolumn{5}{c}{WENO-JS}\\\hline
   N  &  $L_1$ error &$L_{\infty}$ error& $L_{2}$ error & CPU cost ($s$) &
   N  &  $L_1$ error &$L_{\infty}$ error& $L_{2}$ error & CPU cost ($s$)             \\\hline
   7   &    3.780E-03   &   2.796E-03         &  2.939E-03      & 3.13E-02    &
   15 &    2.236E-03   &    1.899E-03         &  1.792E-03      & 4.69E-02     \\
   14 &    7.819E-05   &    8.125E-05         &  6.366E-05      & 0.11            &
   30 &    7.509E-05   &    7.183E-05         &  6.293E-05      & 0.12            \\
   28 &    2.065E-06   &    1.537E-06         &  1.579E-06      & 0.47            &
   60 &    2.351E-06   &    2.346E-06         &  1.917E-06      & 0.69            \\
   56 &    5.879E-08   &    4.699E-08         &  4.511E-08      & 1.86            &
   120 &   7.318E-08  &    6.969E-08         &  5.859E-08      & 2.72            \\
   112 &   1.945E-09  &    1.482E-09         &  1.506E-09      & 7.34            &
   240 &   2.259E-09  &    1.943E-09         &  1.802E-09      & 10.84          \\
   224 &   8.882E-11  &    6.898E-11         &  6.926E-11      & 29.22          &
   480 &   6.800E-11  &    5.616E-11         & 5.377E-11       & 43.27          \\\hline
 \end{tabular}
 \end{center}
\end{table}

\begin{figure}[htp]
\subfigure[Conservative variable reconstruction]{\includegraphics[width=0.5\textwidth]{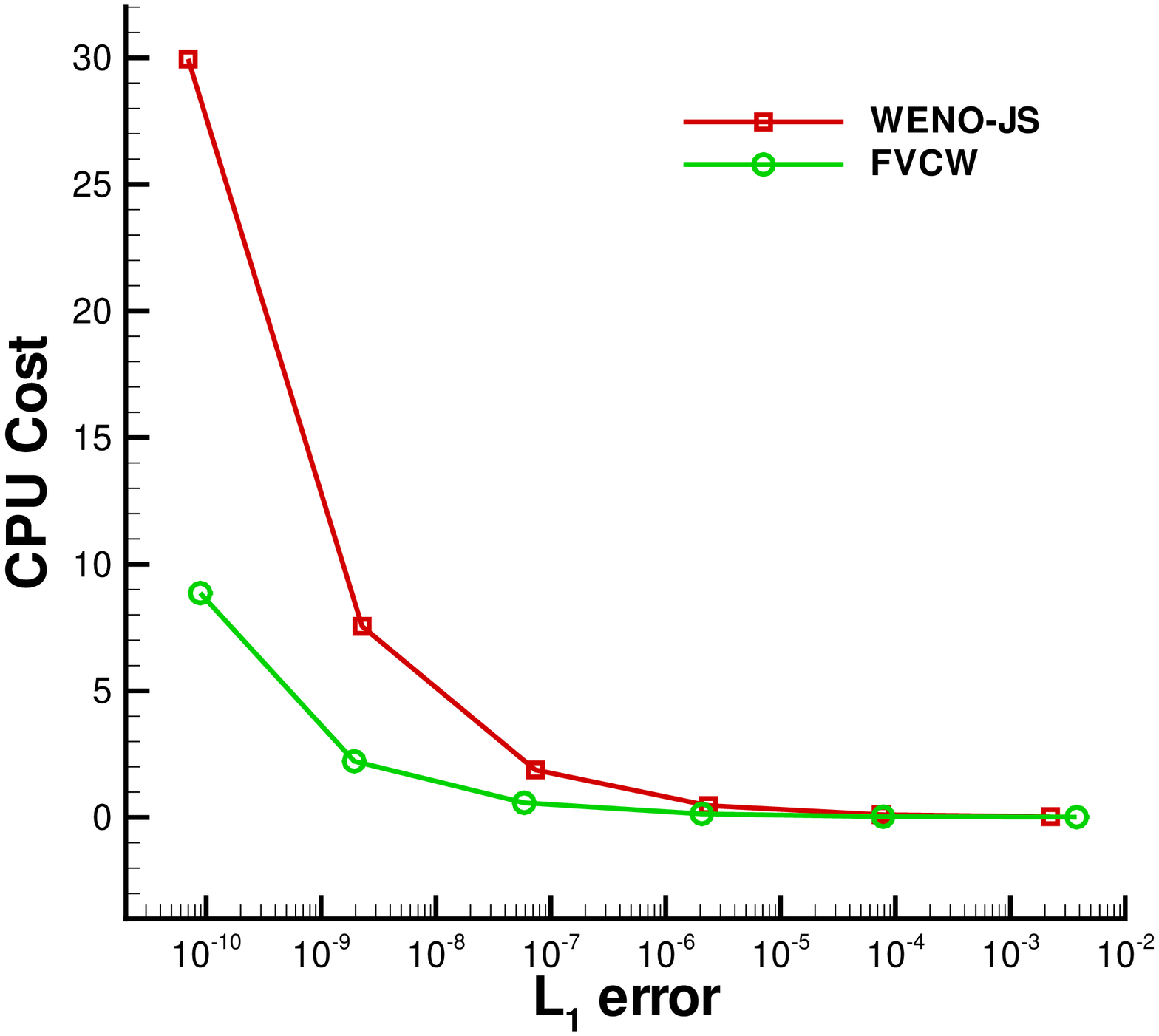}}
\subfigure[Characteristic variable reconstruction]{\includegraphics[width=0.5\textwidth]{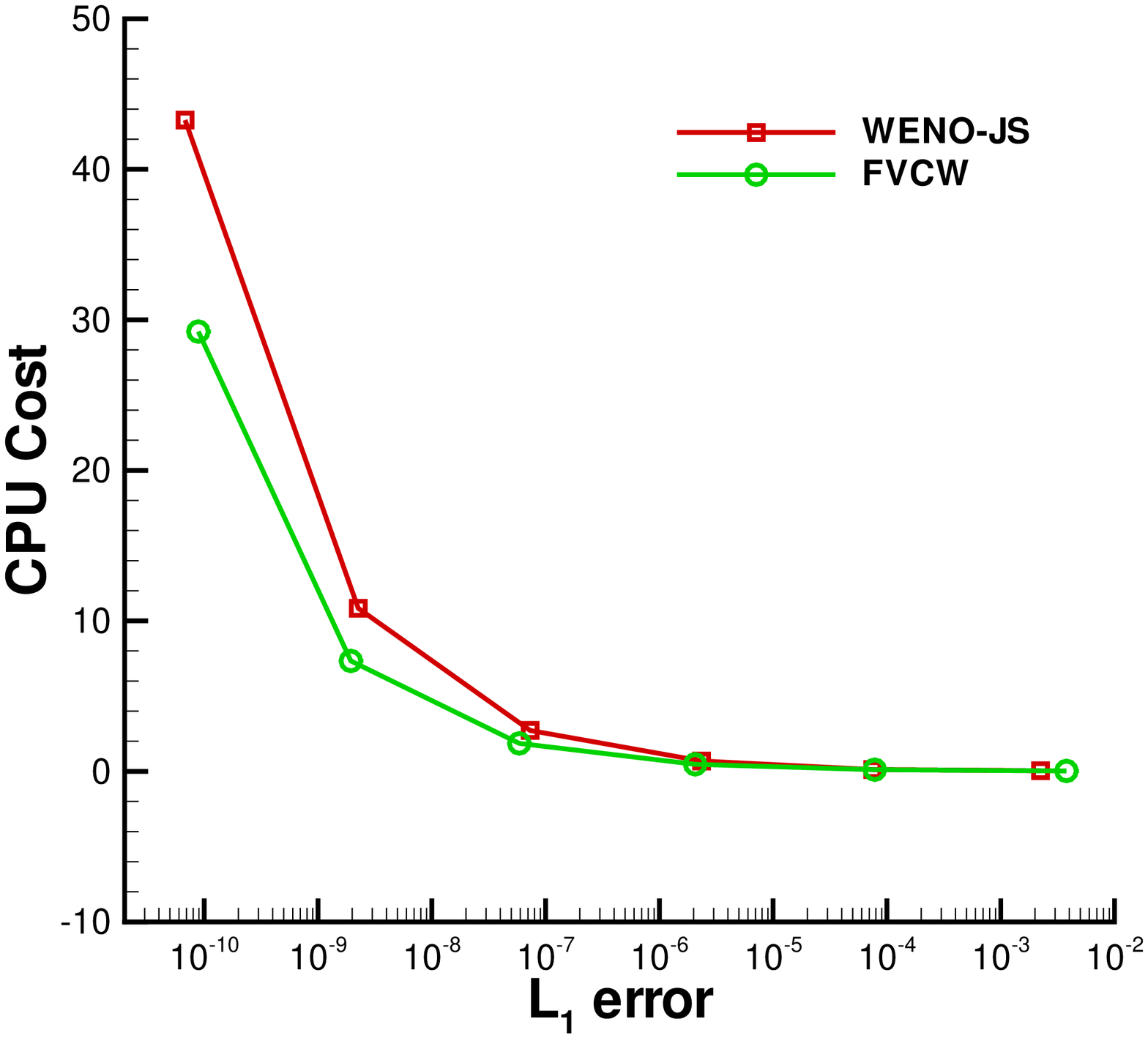}}
\caption{Comparison of CPU cost versus $L_1$ errors for the WENO-JS and FVCW schemes.
Left: conservative variable reconstruction in Table \ref{Tab:2}; Right: characteristic variable reconstruction in Table \ref{Tab:3}.}
\label{Fig:1}
\end{figure}
\end{example}

\begin{example}\label{Eg:Lax}
This example is the one-dimensional Lax shock tube problem
\cite{lax1954weak} with the following Riemann initial conditions
\begin{equation}
(\rho,u,p)=\left\{\begin{array}{ll}
(0.445, 0.698, 3.528), & \textrm{$ -5 \leq x < 0$},\\
(0.5, 0, 0.571), & \textrm{$ 0 \leq x < 5$},
\end{array}\right.
\label{lax}
\end{equation}
and the final time is $t=1.4$.

The exact solutions of a Lax problem contain a strong shock, 
a contact discontinuity and a rarefaction wave. We
compute the solutions on the domain $[-5,5]$ with Neumann boundary
conditions. The density and pressure on a grid
of $200$ points for the WENO-JS, WENO-Z and the present FVCW schemes are shown in Fig. \ref{Fig:2}. For this
test problem, we observe the FVCW scheme is sharper than the
WENO-JS and WENO-Z schemes, as it is less dissipative.

For this problem with discontinuous solutions, we also compare the WENO-JS
scheme with the FVCW scheme at different grid nodes in Fig. \ref{Fig:21}.
As we can see, the FVCW scheme with $N=60$ and CPU cost $0.57s$ can match the result of the WENO-JS
scheme with $N=100$ and CPU cost $0.55s$, both are better than the WENO-JS scheme with $N=60$.
It shows the compact scheme has better resolutions than the non-compact scheme.
At the same resolution, the compact scheme can take much coarser grids while with comparable
computational cost as the non-compact scheme.

\begin{figure}[htp]
\subfigure[Density]{\includegraphics[width=0.5\textwidth]{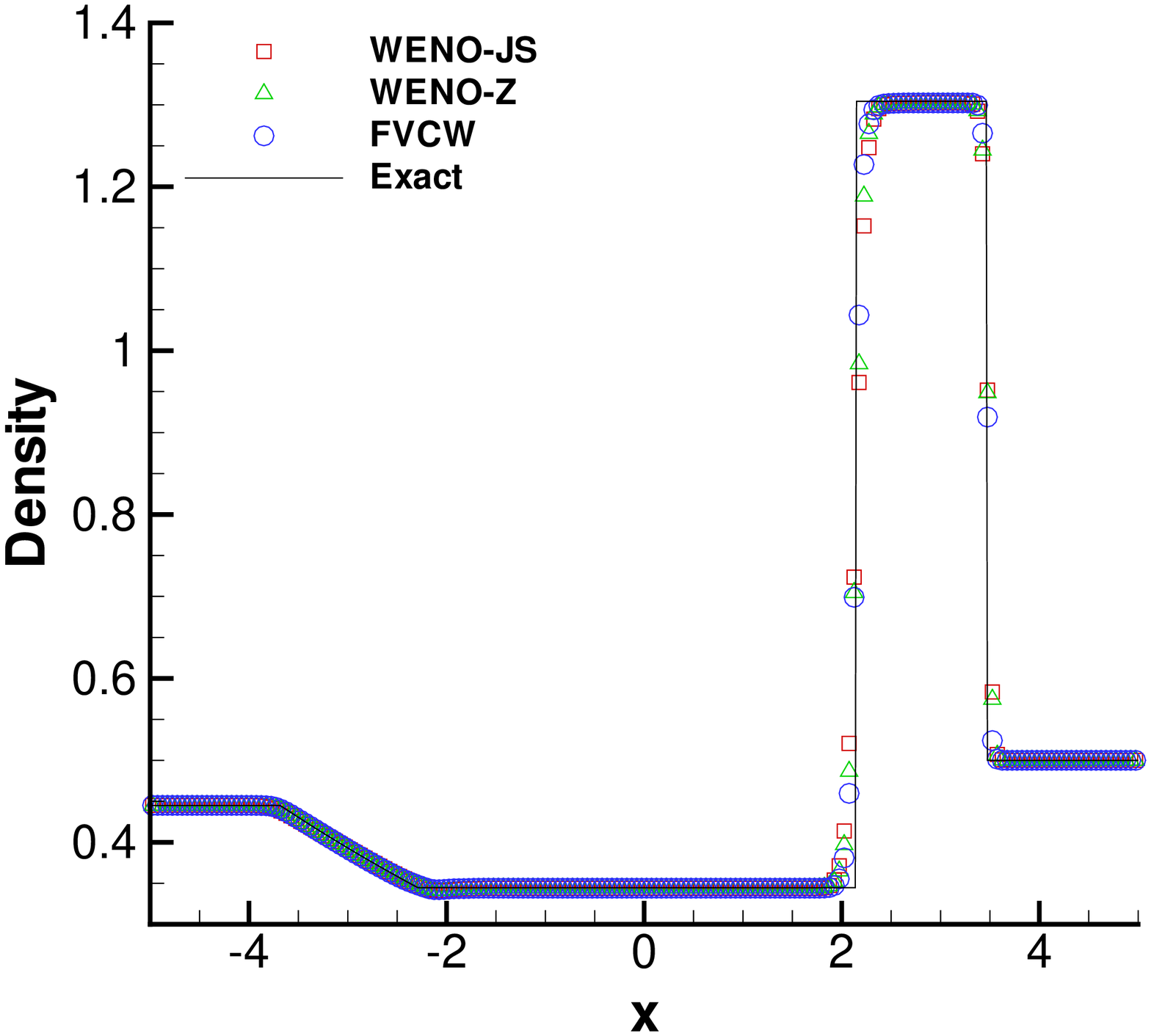}}
\subfigure[Pressure]{\includegraphics[width=0.5\textwidth]{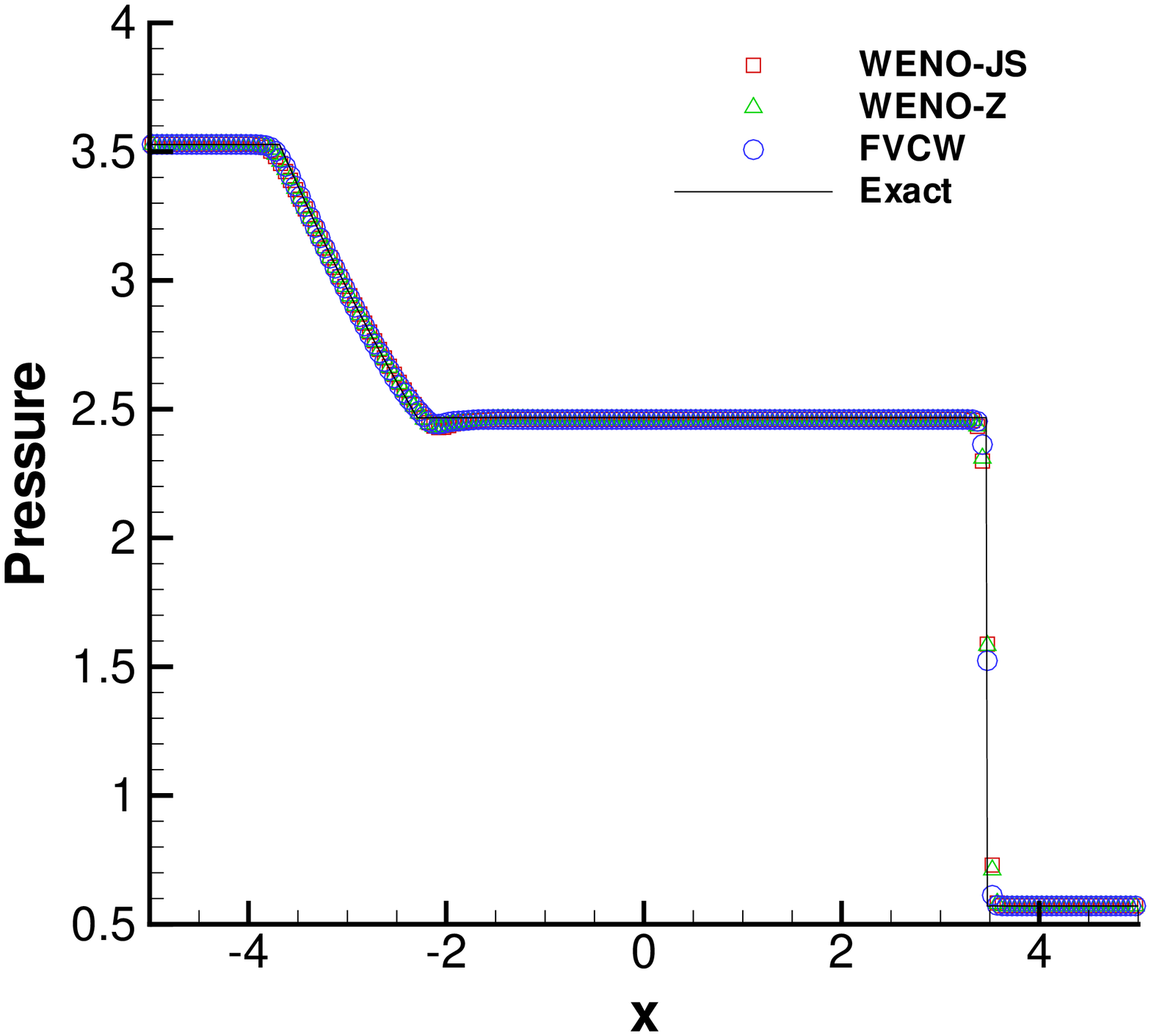}}
\subfigure[Zoom-in of (a) near
shock]{\includegraphics[width=0.5\textwidth]{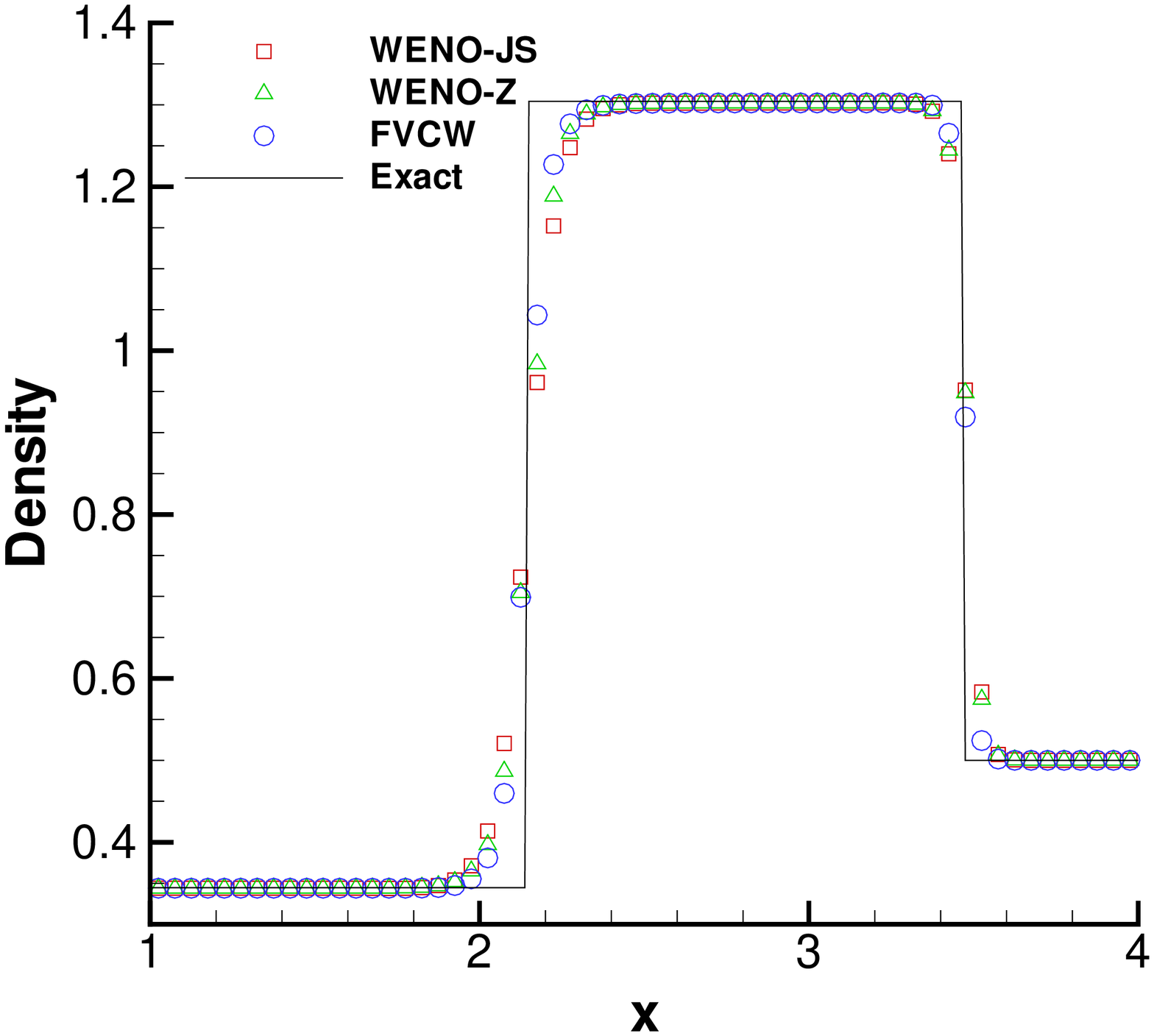}}
\subfigure[Zoom-in of (b) near
shock]{\includegraphics[width=0.5\textwidth]{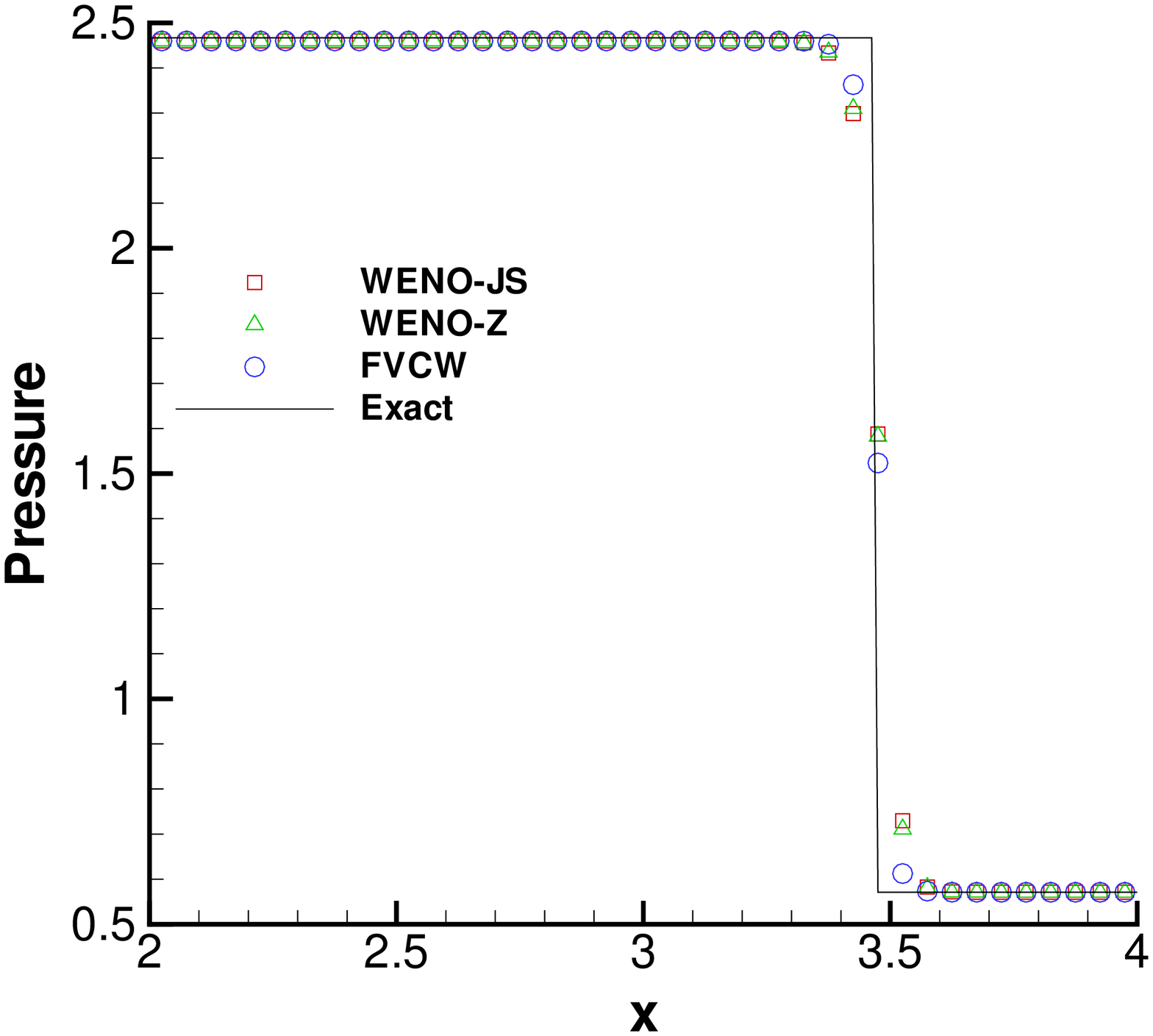}}
\subfigure[Zoom-in of (a) near rarefaction
wave]{\includegraphics[width=0.5\textwidth]{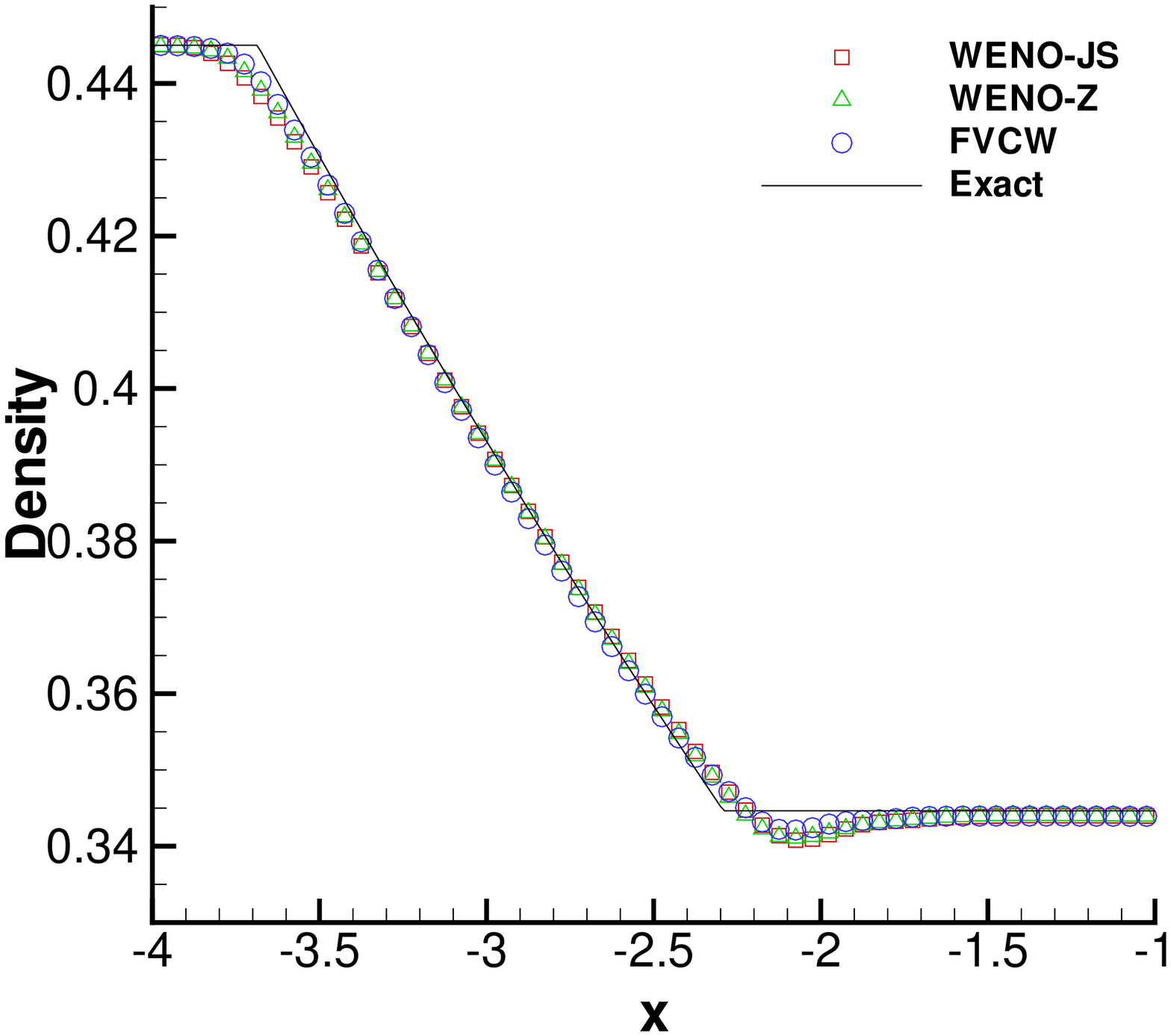}}
\subfigure[Zoom-in of (b) near rarefaction
wave]{\includegraphics[width=0.5\textwidth]{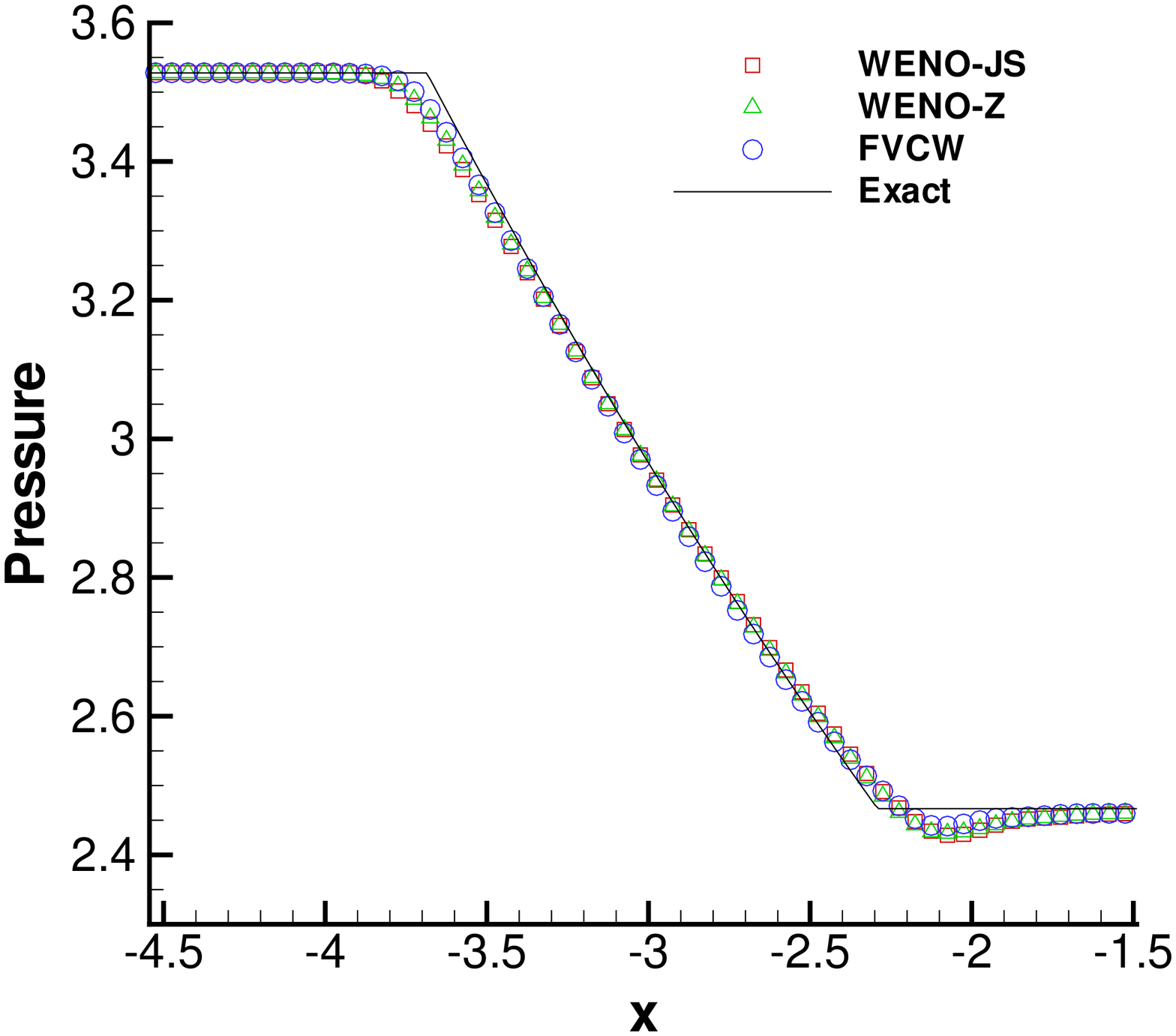}}
\caption{The density (left) and pressure (right) profiles of the Lax
problem (\ref{lax}) at $t=1.4$.}\label{Fig:2}
\end{figure}

\begin{figure}[htp]
\subfigure[Density]{\includegraphics[width=0.5\textwidth]{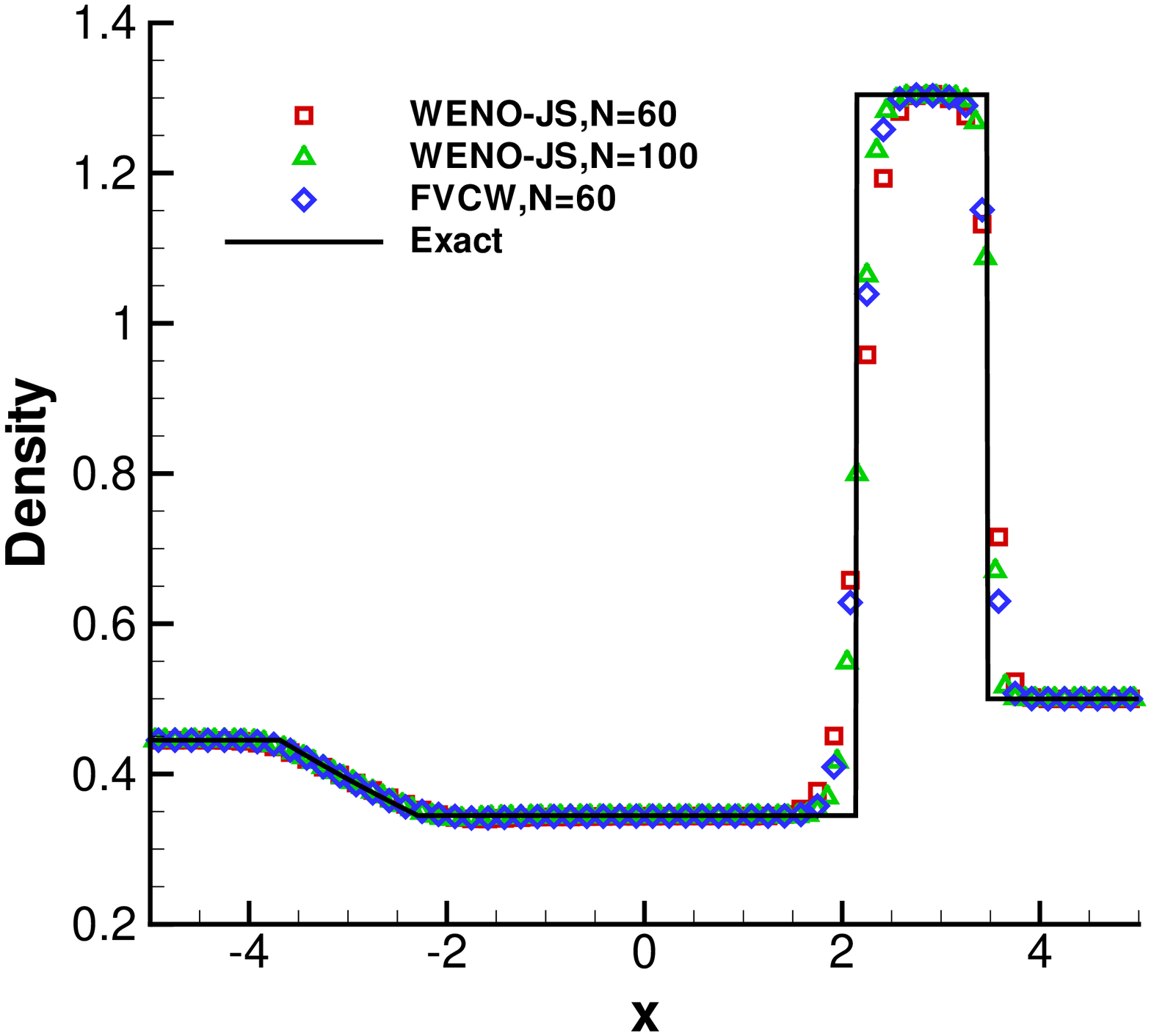}}
\subfigure[Zoom-in of (a) near
shock]{\includegraphics[width=0.5\textwidth]{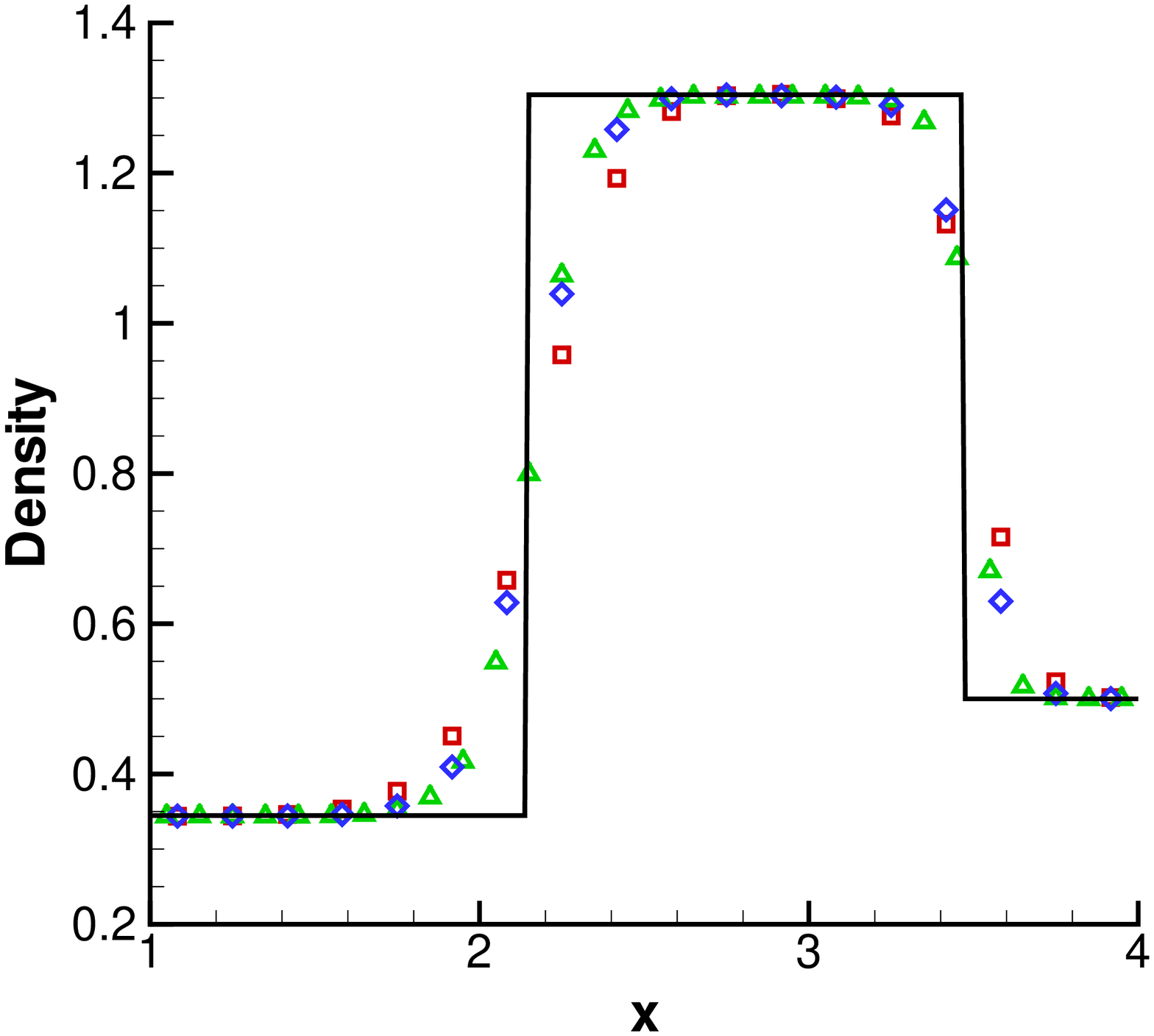}}
\caption{The comparison of density for the Lax
problem (\ref{lax}) with the WENO-JS scheme and the FVCW scheme at $t=1.4$.}\label{Fig:21}
\end{figure}
\end{example}

\begin{example}\label{Eg:Sod}
This example is the one-dimensional Sod shock tube problem 
\cite{sod1978survey} with the following Riemann initial conditions
\begin{equation}
(\rho,u,p)=\left\{\begin{array}{ll}
(0.125, 0, 1), & \textrm{$ -5 \leq x < 0$},\\
(1, 0, 1), & \textrm{$ 0 \leq x < 5$},
\end{array}\right.
\label{sod}
\end{equation}
and the final time is $t=2.0$.

The exact solution contains a left-running rarefaction wave and
a right-running contact discontinuity and a shock wave. The spatial
domain $[-5, 5]$ is discretized with $100$ grid points and the results are shown
in Fig.\ref{Fig:3}. We compare our numerical results with those
obtained by WENO-JS and WENO-Z schemes. The present scheme can
capture the shock front and the contact discontinuity with correct
locations and satisfactory sharpness. From Fig.\ref{Fig:3}(b,c), we
can observe that the numerical results obtained by the present FVCW
scheme shows significant lower smearing across the
discontinuities.

\begin{figure}
\begin{center}
\subfigure[Density]{\includegraphics[width=0.5\textwidth]{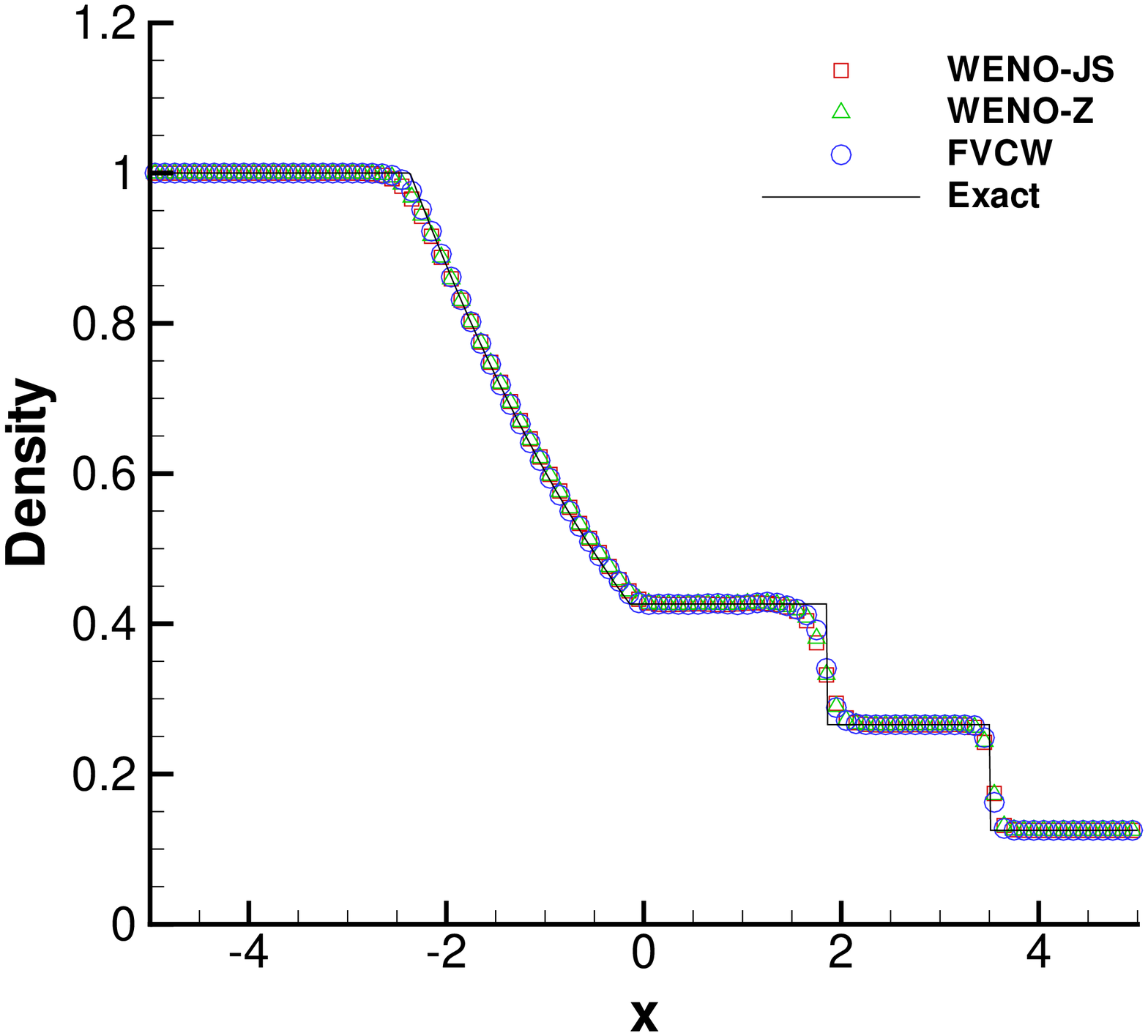}}
\end{center}
 \subfigure[Zoom-in of (a) near
shock]{\includegraphics[width=0.5\textwidth]{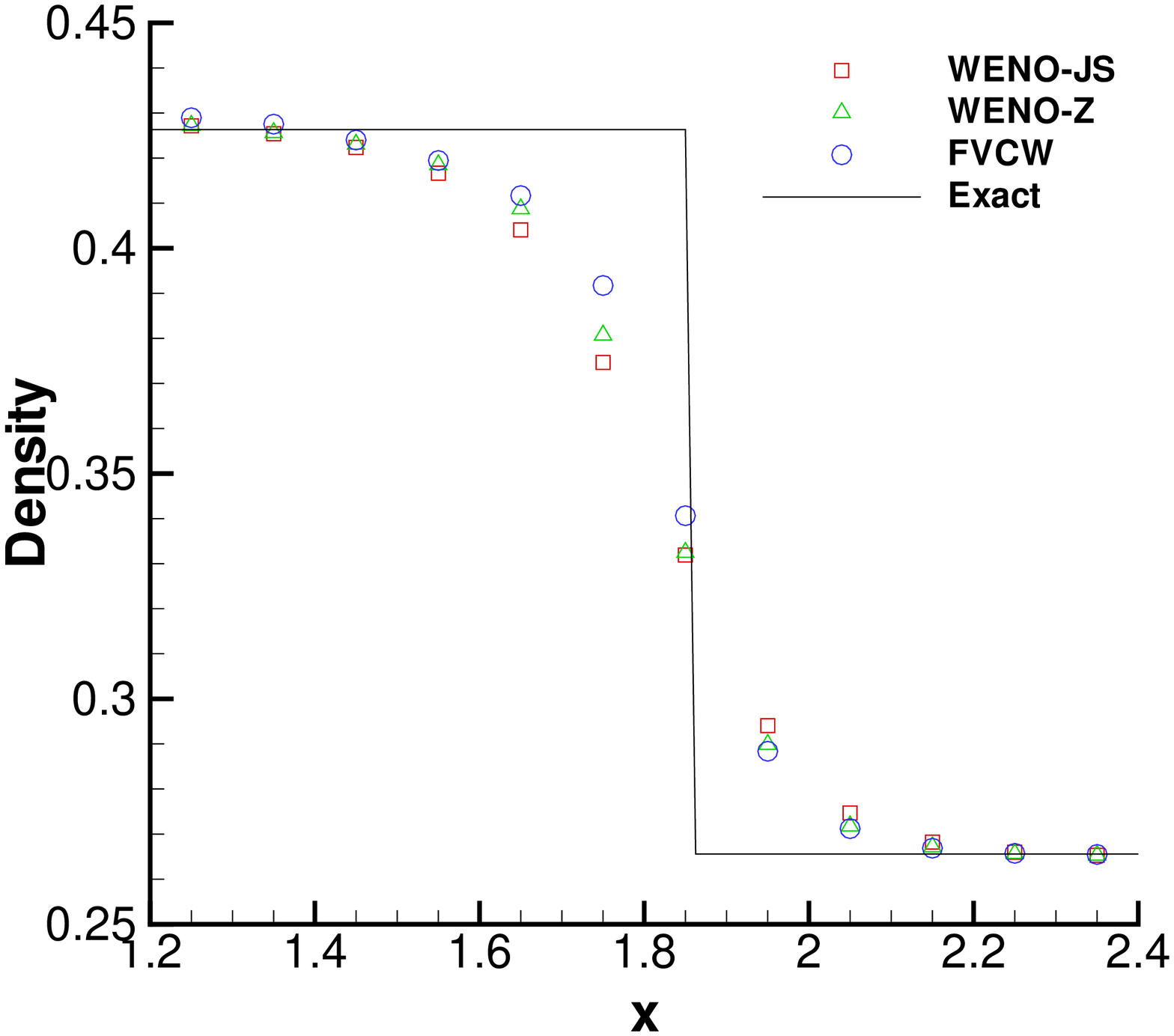}}
\subfigure[Zoom-in of (a) near rarefaction
wave]{\includegraphics[width=0.5\textwidth]{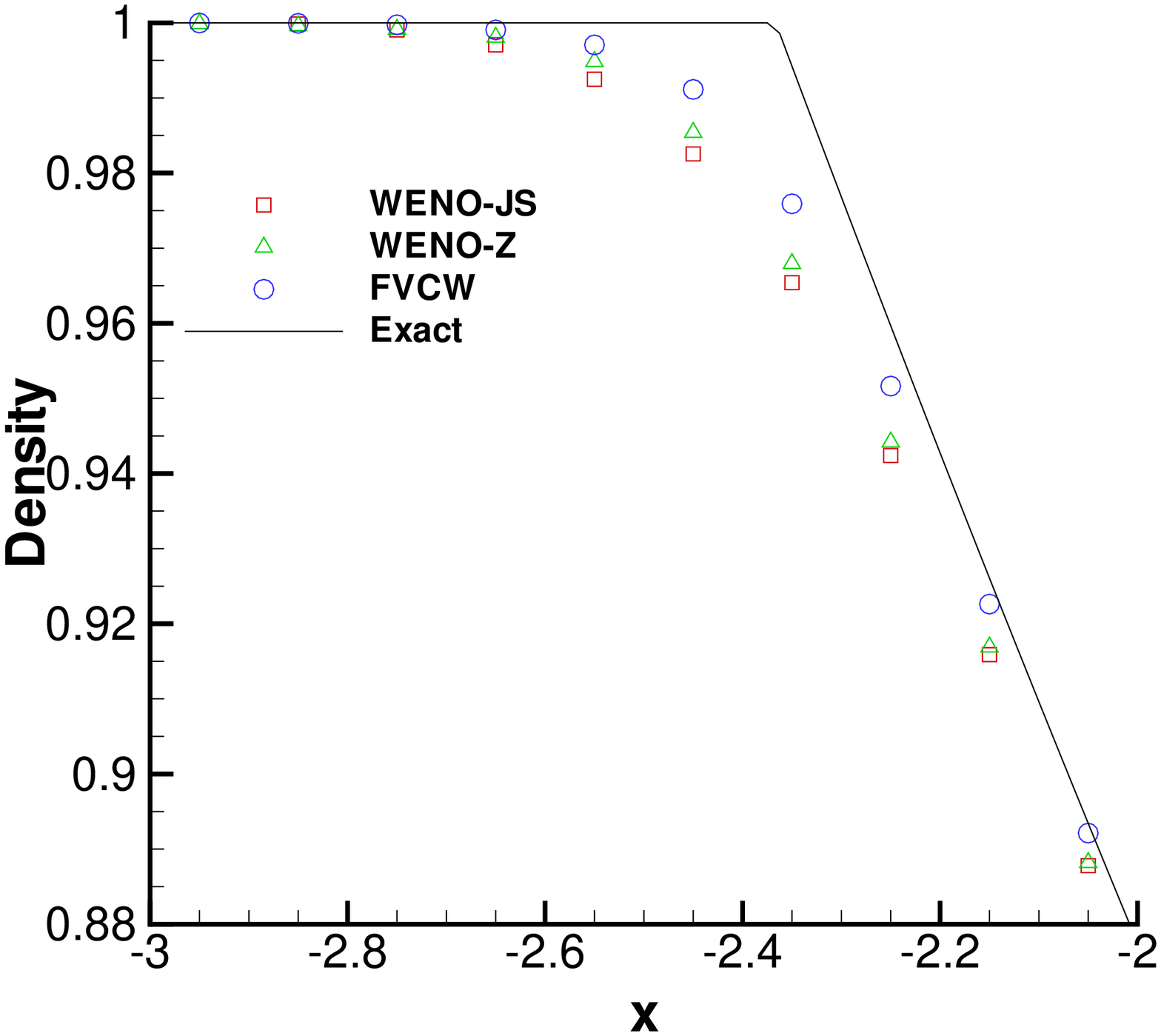}}
\caption{The density profiles of the Sod problem (\ref{sod}) at $t=2.0$.}
\label{Fig:3}
\end{figure}

\end{example}

\begin{example}\label{Eg:Sti}
In this example, the one dimensional Mach 3 shock-turbulence wave
interaction problem \cite{shu1989efficient} is tested with the following
initial conditions
\begin{equation}
(\rho,u,p)=\left\{\begin{array}{ll}
(3.857143, 2.629369, 10.33333), & \textrm{$ -5 \leq x < -4$},\\
(1+0.2\sin{5x}, 0, 1), & \textrm{$ -4 \leq x < 5$},
\end{array}\right.
\label{OS}
\end{equation}
and the final time is $t=1.8$. The solution of this problem consists
of the interaction of a stationary shock and fine scale structures
which are located behind a right-going main shock. As the density
perturbation passes through the shock, it produces perturbations
developing into the shock with smaller amplitude. Fig.\ref{Fig:4} shows
the density on a grid of $200$ points for the WENO-JS, WENO-Z and
FVCW schemes. The ``exact solution'' is a reference solution computed
by the WENO-JS scheme with $3200$ grid points. It is observed that the
present finite volume compact scheme captures the fine scale
structures of the solution at the high-frequency waves behind the
shock better than WENO-JS and WENO-Z, while also maintaining non-oscillatory
behavior across the shock wave. The numerical solution is greatly
improved with $N=400$ and the numerical results are shown in Fig.
\ref{Fig:5}.

\begin{figure}
\subfigure[Density:
$N=200$]{\includegraphics[width=0.5\textwidth]{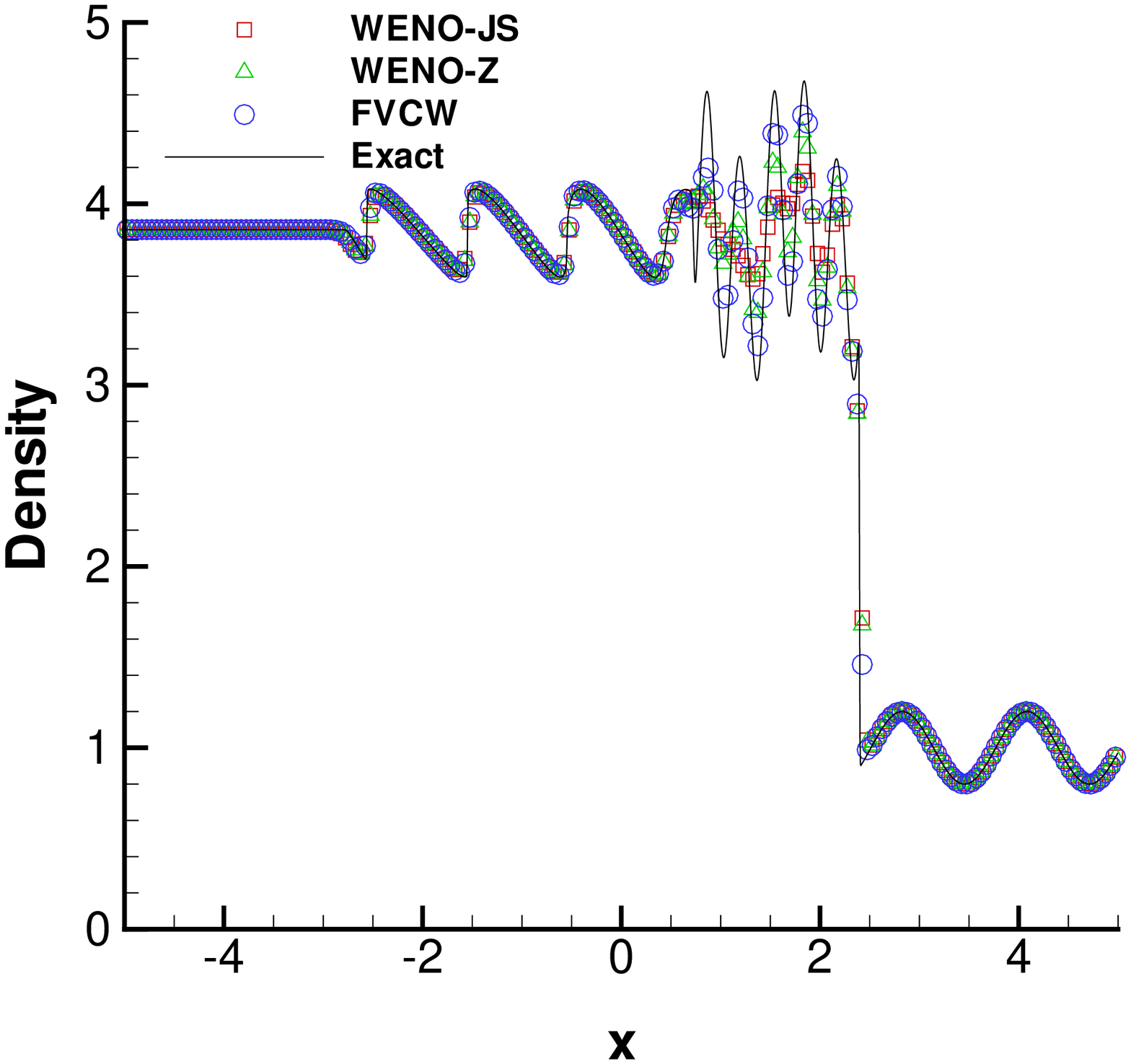}}
\subfigure[Zoom-in of (a) near shock-turbulence
wave]{\includegraphics[width=0.5\textwidth]{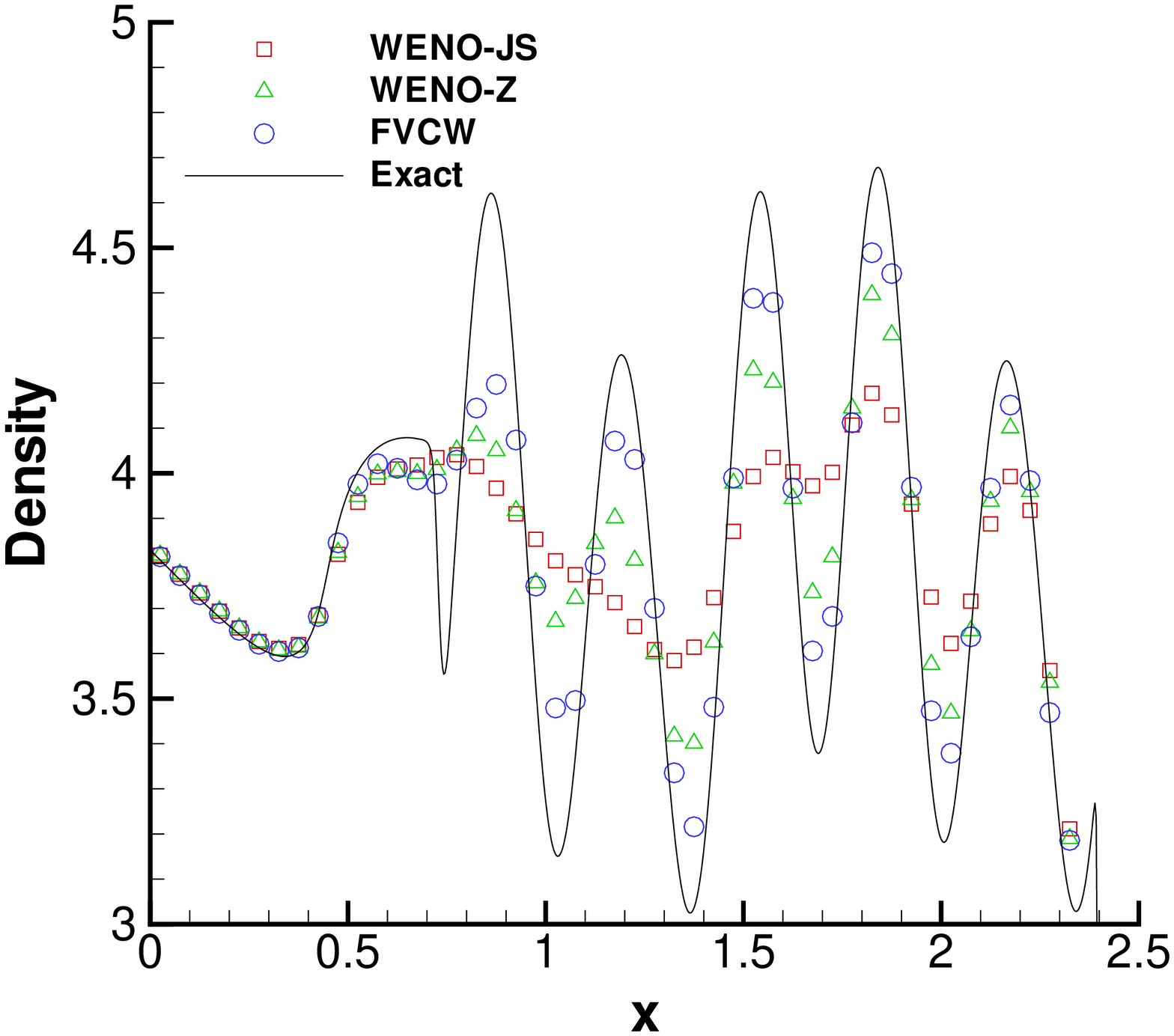}}
\caption{Shock-turbulence interaction (\ref{OS}) with $N=200$ at $t=1.8$.}
\label{Fig:4}
\end{figure}

\begin{figure}
\subfigure[Density:
$N=400$]{\includegraphics[width=0.5\textwidth]{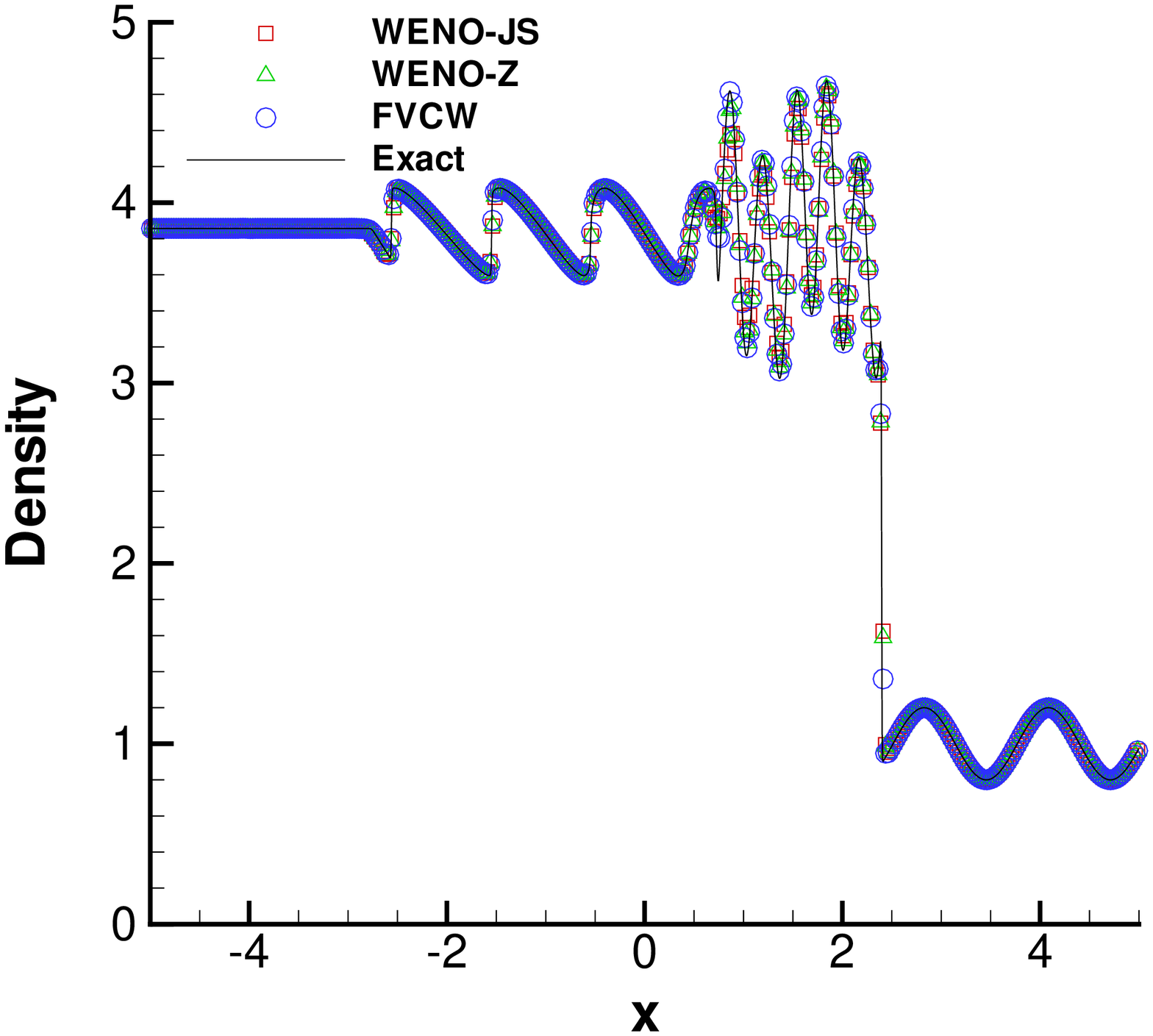}}
\subfigure[Zoom-in of (a) near shock-turbulence
wave.]{\includegraphics[width=0.5\textwidth]{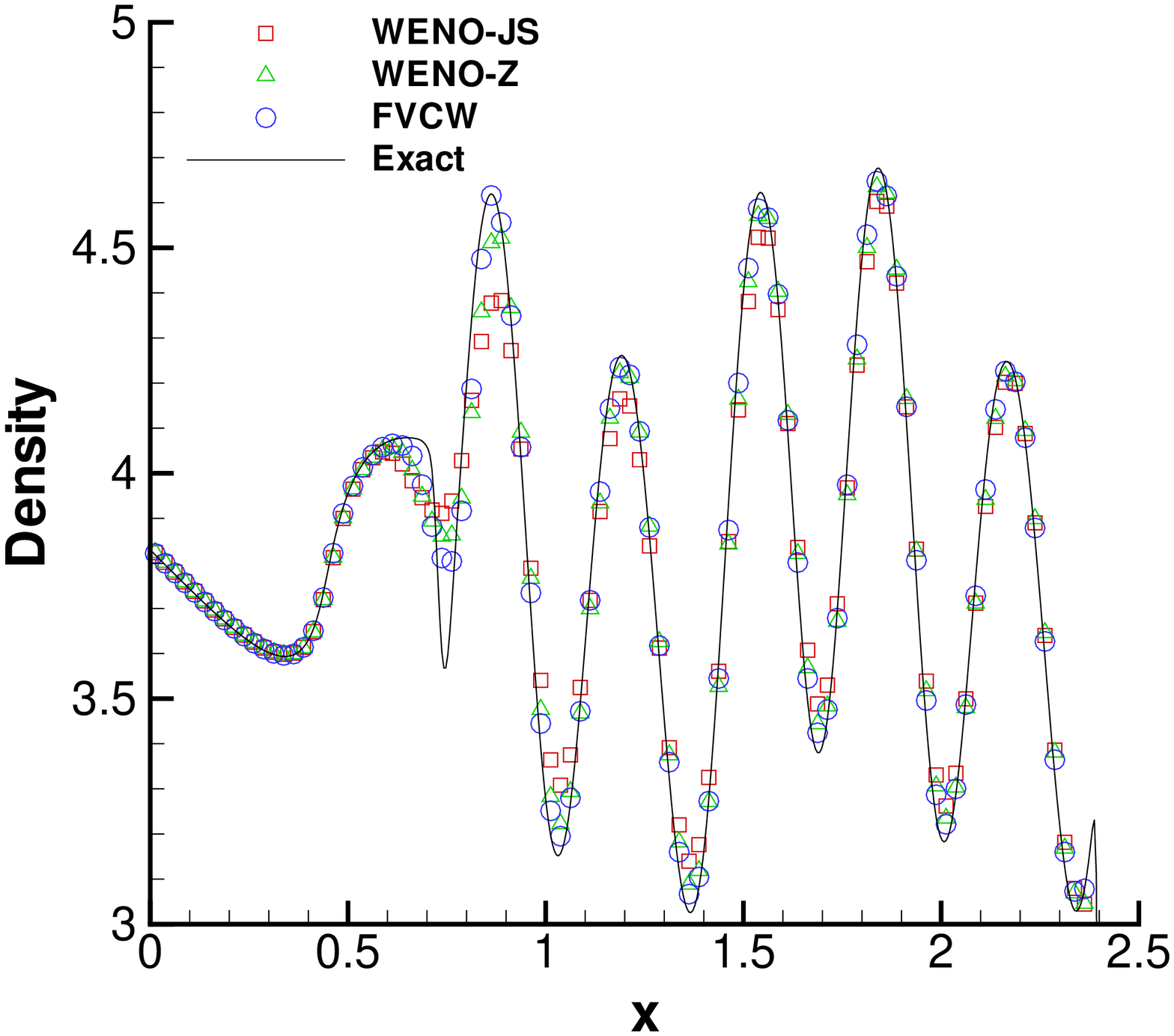}}
\caption{Shock-turbulence interaction (\ref{OS}) with $N=400$ at $t=1.8$.}
\label{Fig:5}
\end{figure}
\end{example}

\begin{example}\label{Eg:bwi}
The one dimensional blastwave interaction problem of Woodward and
Collela \cite{woodward1984numerical} has the following initial
conditions
\begin{equation}
(\rho,u,p)=\left\{\begin{array}{ll}
(1,0,1000), & \textrm{$ 0 \leq x < 0.1$},\\
(1,0,0.01), & \textrm{$ 0.1 \leq x < 0.9$},\\
(1,0,100), & \textrm{$ 0.9 \leq x \leq 1.0$},
\end{array}\right.
\label{blast}
\end{equation}
and reflective boundary conditions. The final time is $t=0.038$.
The initial pressure gradients generate two density shock waves which
collide and interact at later time. The solution of this problem
contains rarefactions, interaction of shock waves and the collision
of strong shock waves. The ``exact solution'' of this test problem is
a reference solution computed by the WENO-JS scheme with $3200$ grid points.
The density obtained with WENO-JS, WENO-Z and the present FVCW
schemes at $t=0.038$ with $200$ cells are shown in Fig. \ref{Fig:6}.
The zoomed regions of the density profile Fig. \ref{Fig:6} (b) show
that the present FVCW scheme gives better resolution than the other two
schemes. The numerical solution is also greatly improved with $N=400$
and the numerical results are shown in Fig. \ref{Fig:7}.

\begin{figure}
\subfigure[Density:
$N=200$]{\includegraphics[width=0.5\textwidth]{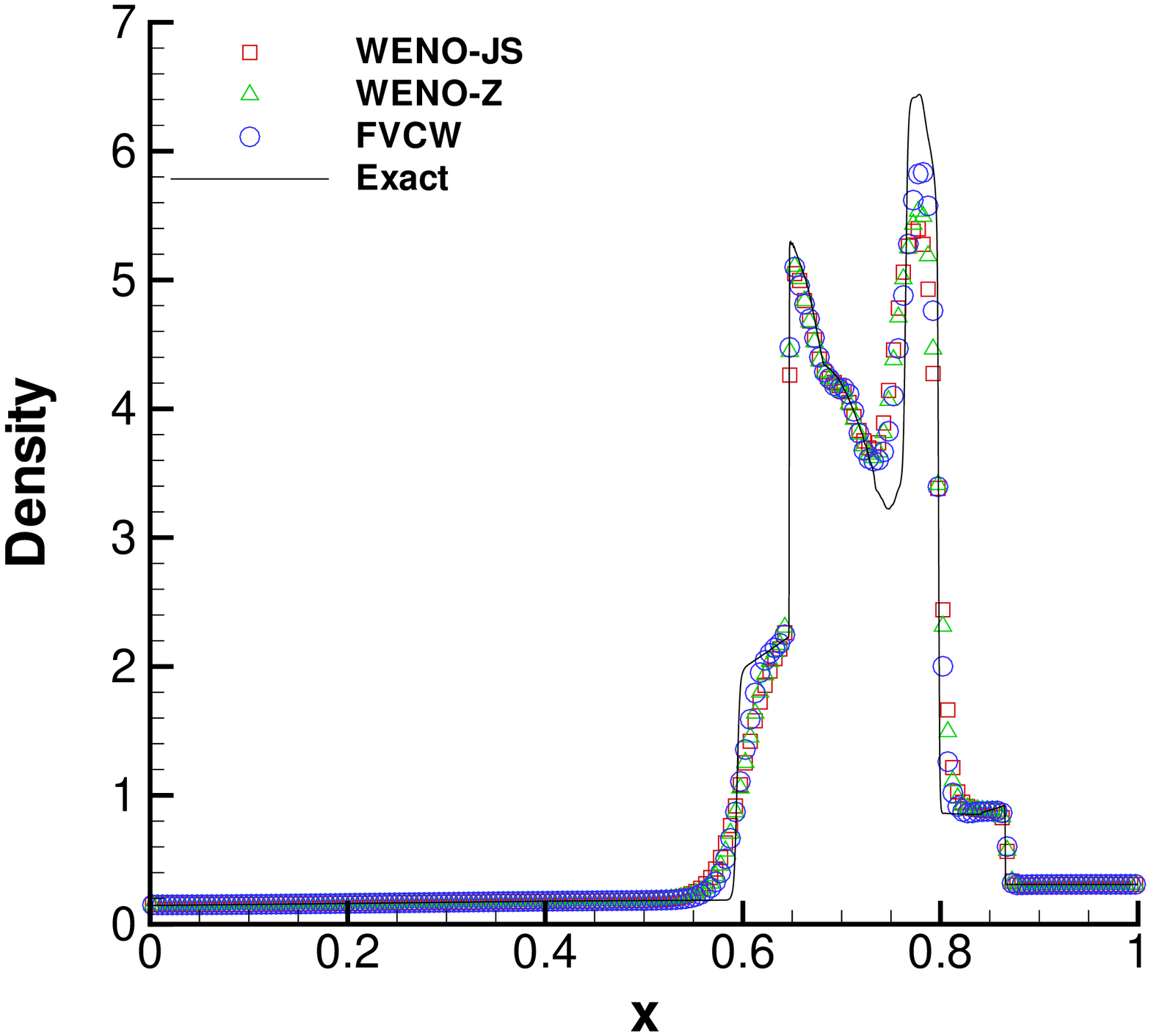}}
\subfigure[Zoom-in of (a) at the complex smooth
region.]{\includegraphics[width=0.5\textwidth]{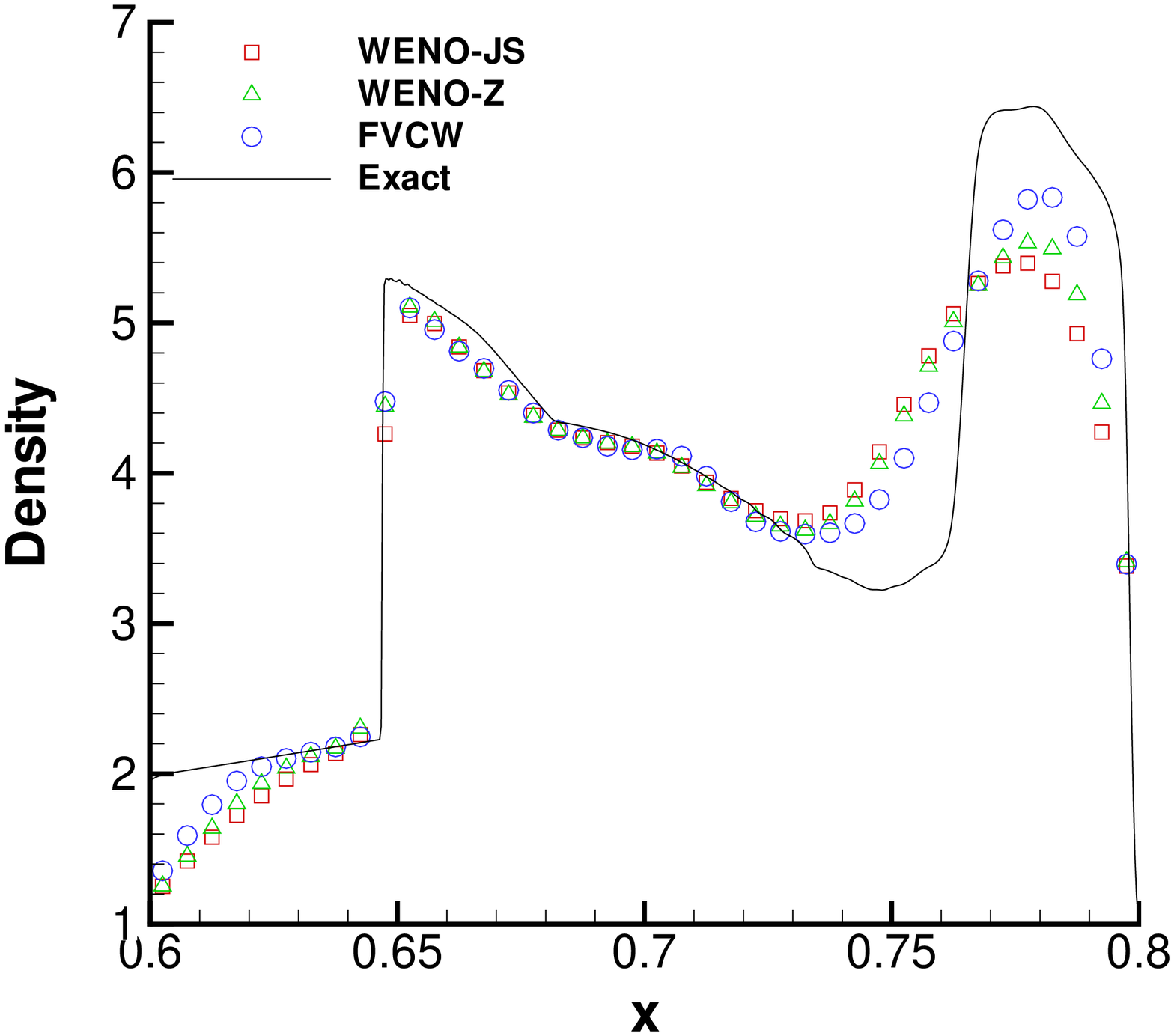}}
\caption{Blastwave interaction problem (\ref{blast}) with $N=200$ at $t=0.038$.}
\label{Fig:6}
\end{figure}
\begin{figure}
\subfigure[Density:
$N=400$]{\includegraphics[width=0.5\textwidth]{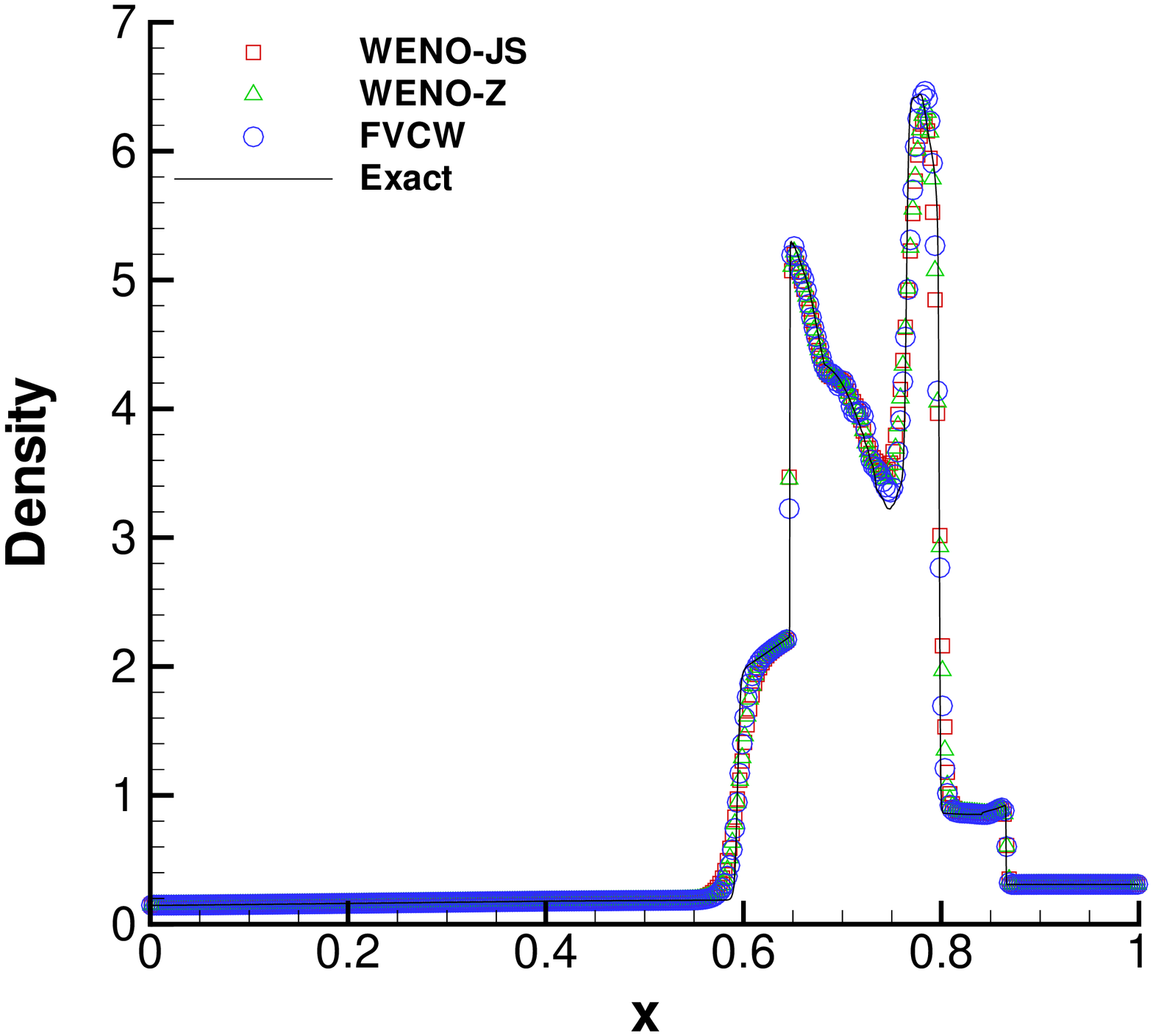}}
\subfigure[Zoom-in of (a) at the complex smooth
region.]{\includegraphics[width=0.5\textwidth]{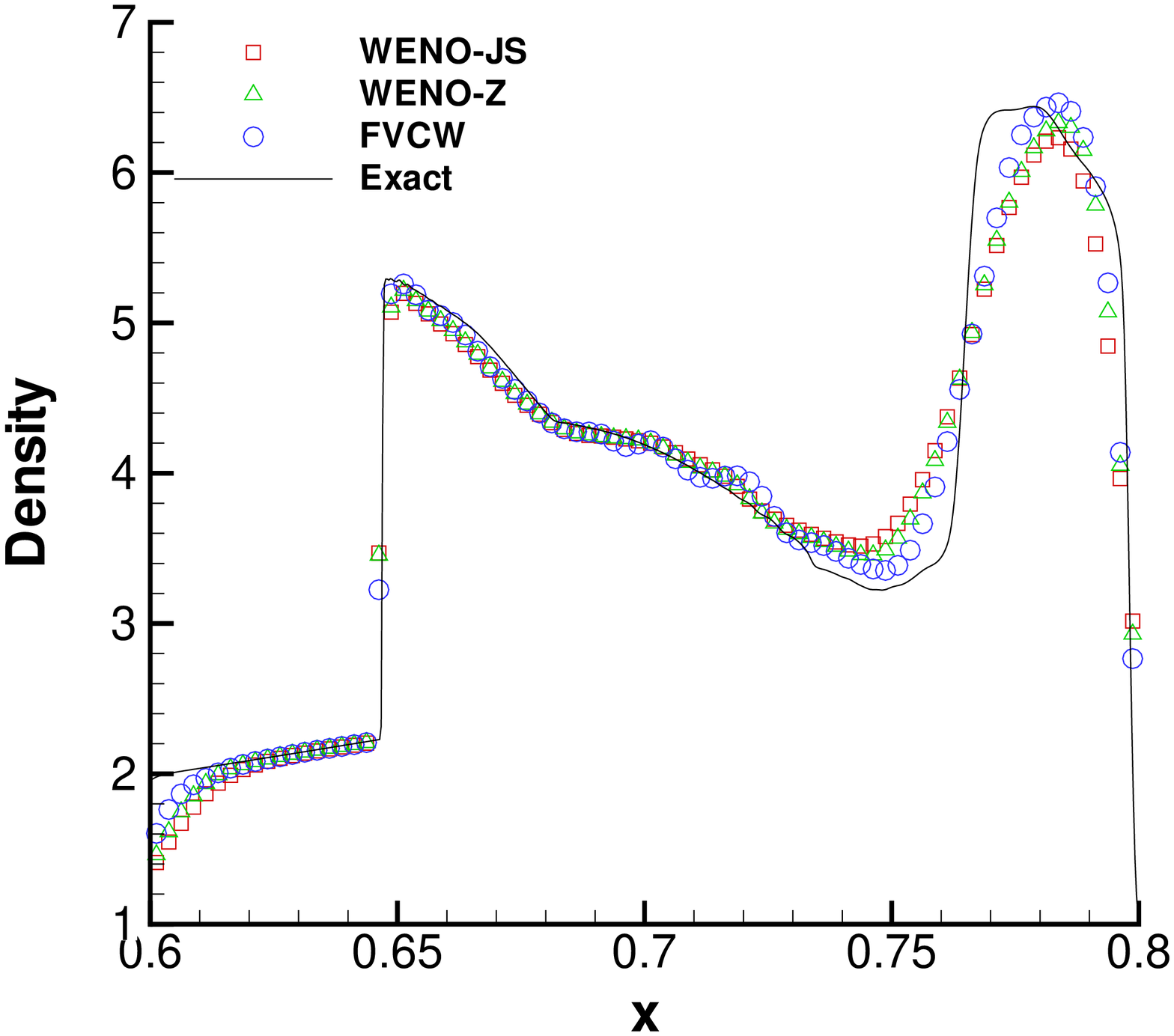}}
\caption{Blastwave interaction problem (\ref{blast}) with $N=400$ at $t=0.038$.}
\label{Fig:7}
\end{figure}
\end{example}

\begin{example}\label{Eg:123}
In this test, we consider a one-dimensional low density and low internal
energy Riemann problem with the following initial conditions
\begin{equation}
(\rho,u,p)=\left\{\begin{array}{ll}
(1,-2,0.4), & \textrm{$ 0 \leq x < 0.5$},\\
(1,2,0.4), & \textrm{$ 0.5 \leq x < 1$}.
\end{array}\right.
\label{lowrp}
\end{equation}
We take $h=0.0025$ and the final time $t=0.1$. The exact solution of
this test consists of a trivial contact wave and two symmetric
rarefaction waves. The results of the present positivity-preserving
FVCW scheme with $400$ cells compared with the exact solution are
shown in Fig. \ref{Fig:8}. The minimum numerical values of the
density and the internal energy are $1.835E-02$ and $3.158E-01$
respectively. For this problem, we can observe some oscillations in the 
central region, especially for the velocity and the internal energy. 
This might be due to the small density around that area and we have used a less dissipative HLLC flux. Slight oscillations on the density would cause very large oscillations
on the velocity and the internal energy. The Lax-Friedrichs flux can be used to
control the oscillations, we omit the results here to save space.
\begin{figure}
\subfigure[Density ]{\includegraphics[width=0.5\textwidth]{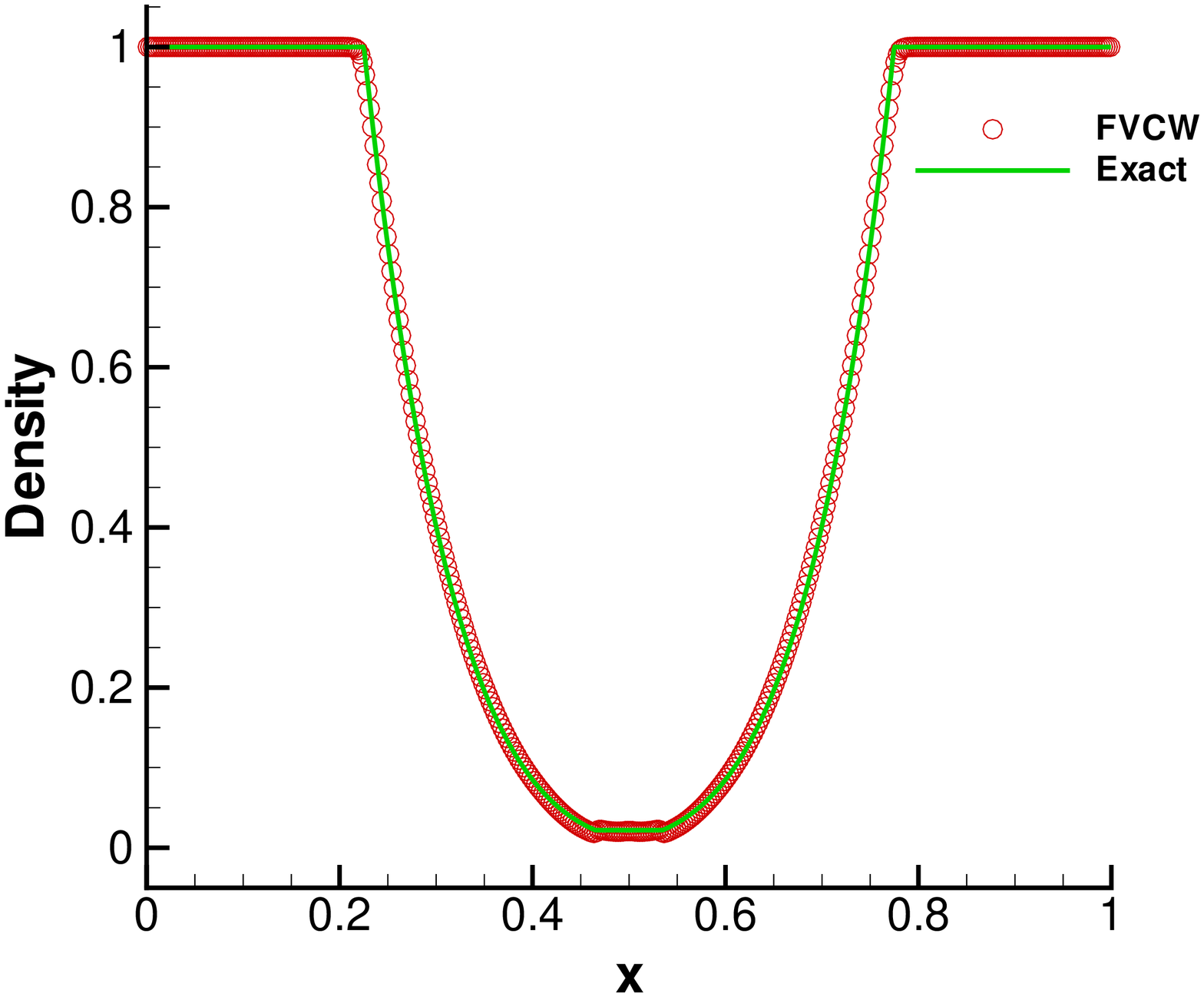}}
\subfigure[Velocity]{\includegraphics[width=0.5\textwidth]{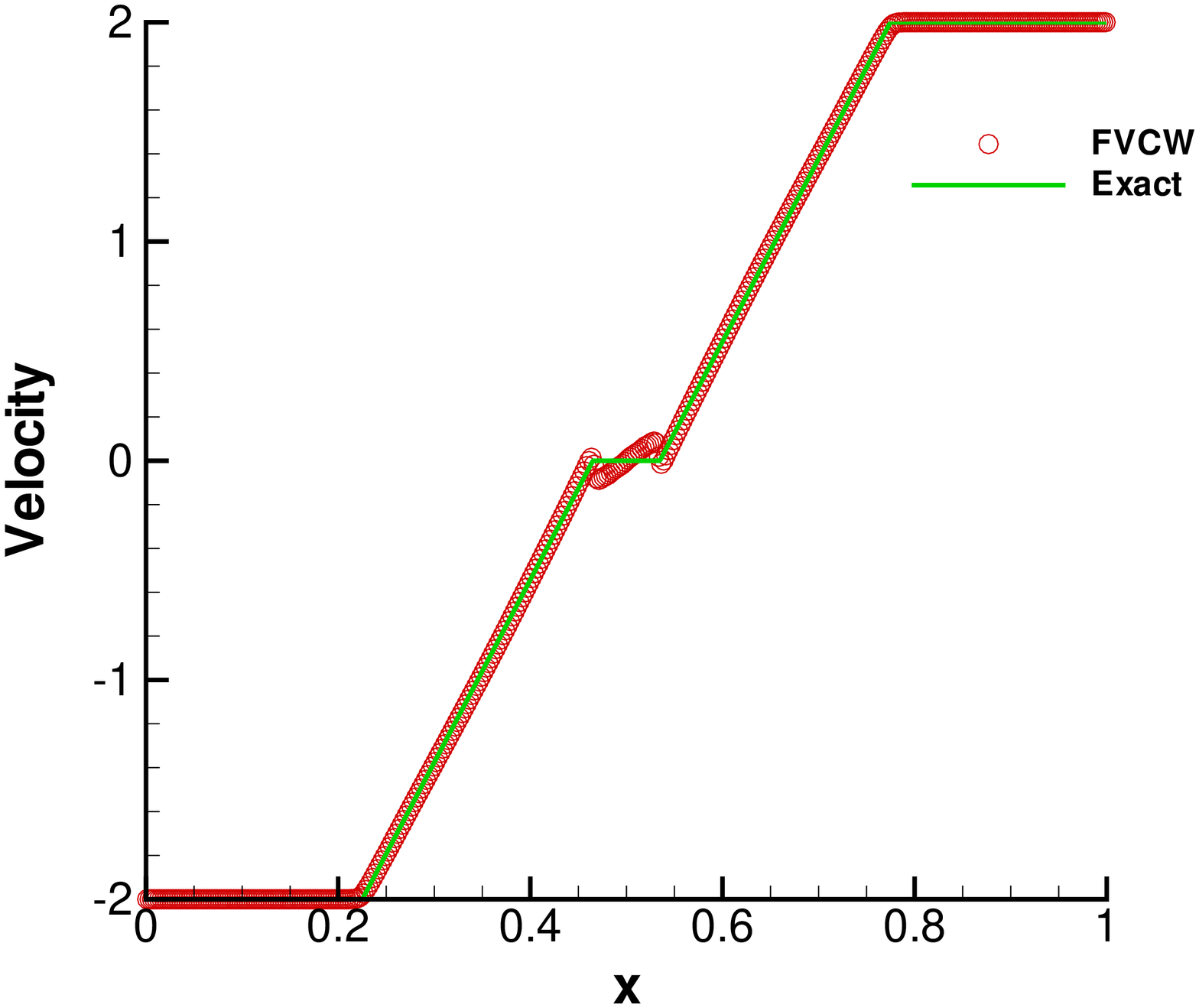}}
\subfigure[Pressure]{\includegraphics[width=0.5\textwidth]{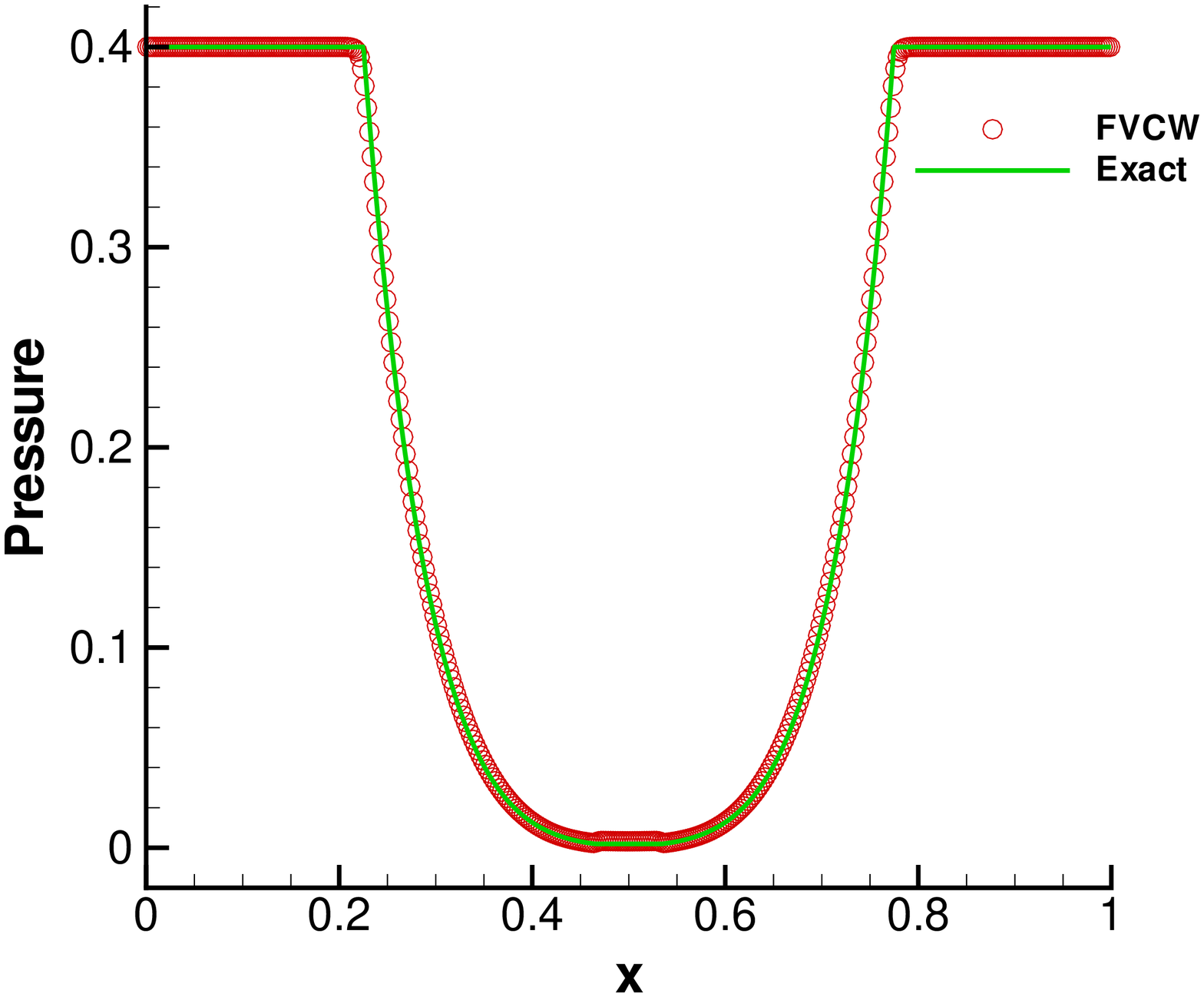}}
\subfigure[Internal
energy]{\includegraphics[width=0.5\textwidth]{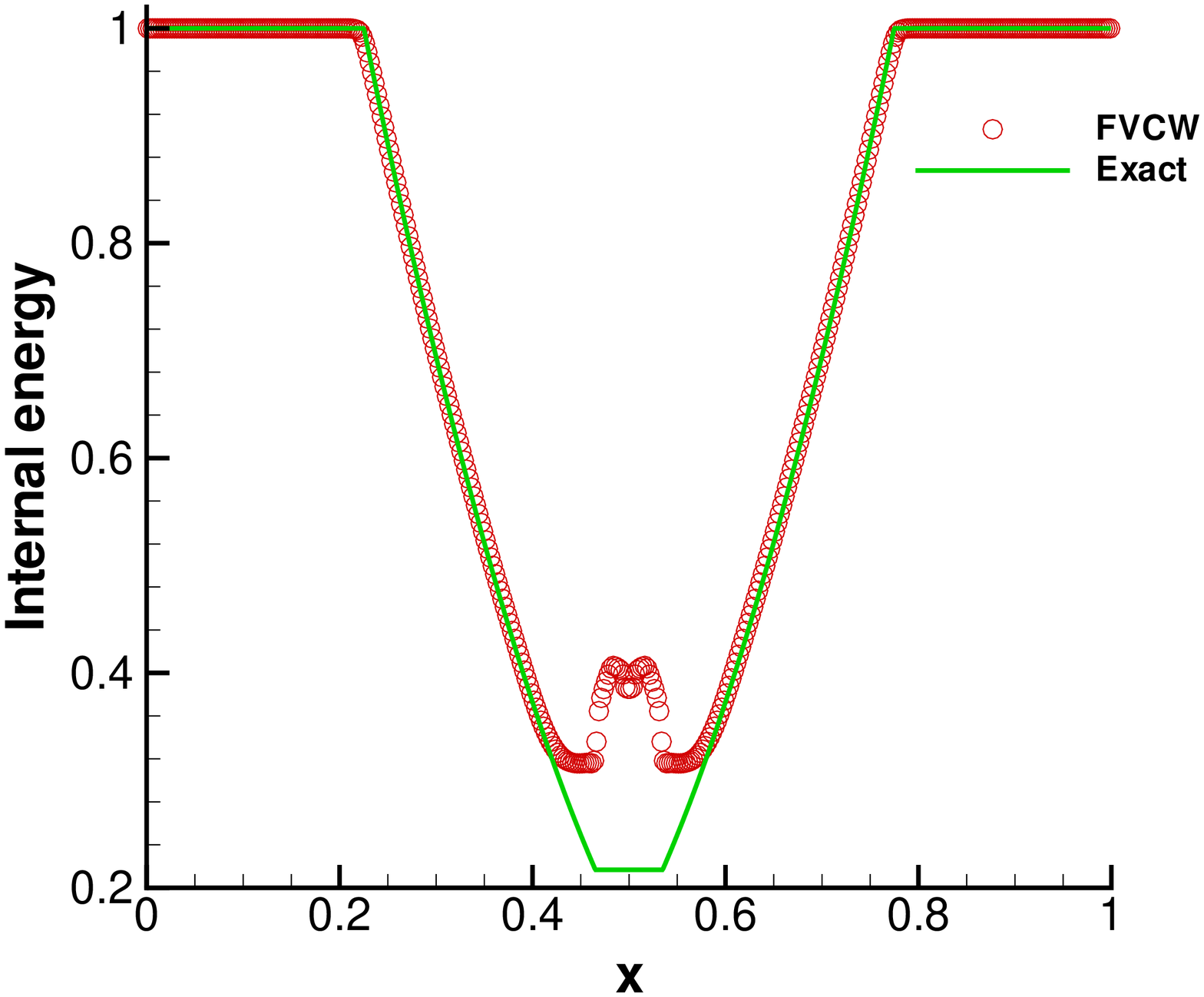}}
\caption{The results of the low density and low internal energy problem (\ref{lowrp})
with $N=400$ at $t=0.1$.}\label{Fig:8}
\end{figure}
\end{example}

\begin{example}\label{Eg:ess}
In this test, a strong shock wave is generated by an extremely high
pressure in the initial conditions
\begin{equation}
(\rho,u,p)=\left\{\begin{array}{ll}
(1,0,10^{10}), & \textrm{$ 0 \leq x \leq 0.5$},\\
(0.125,0,0.1), & \textrm{$ 0.5 \leq x < 1$},
\end{array}\right.
\label{hshock}
\end{equation}
with the final time $t=2.5\times 10^{-6}$. The results of the
present positivity-preserving FVCW scheme with 200 cells compared
with the exact solution are shown in Fig.\ref{Fig:9}. The numerical
solutions are very satisfactory in regard to numerical diffusion
and spurious oscillations. The minimum numerical values of the density
and the internal energy for this problem are $1.250E-01$ and $2.000E+00$
respectively. Both are positive.
\begin{figure}
\subfigure[Density]{\includegraphics[width=0.5\textwidth]{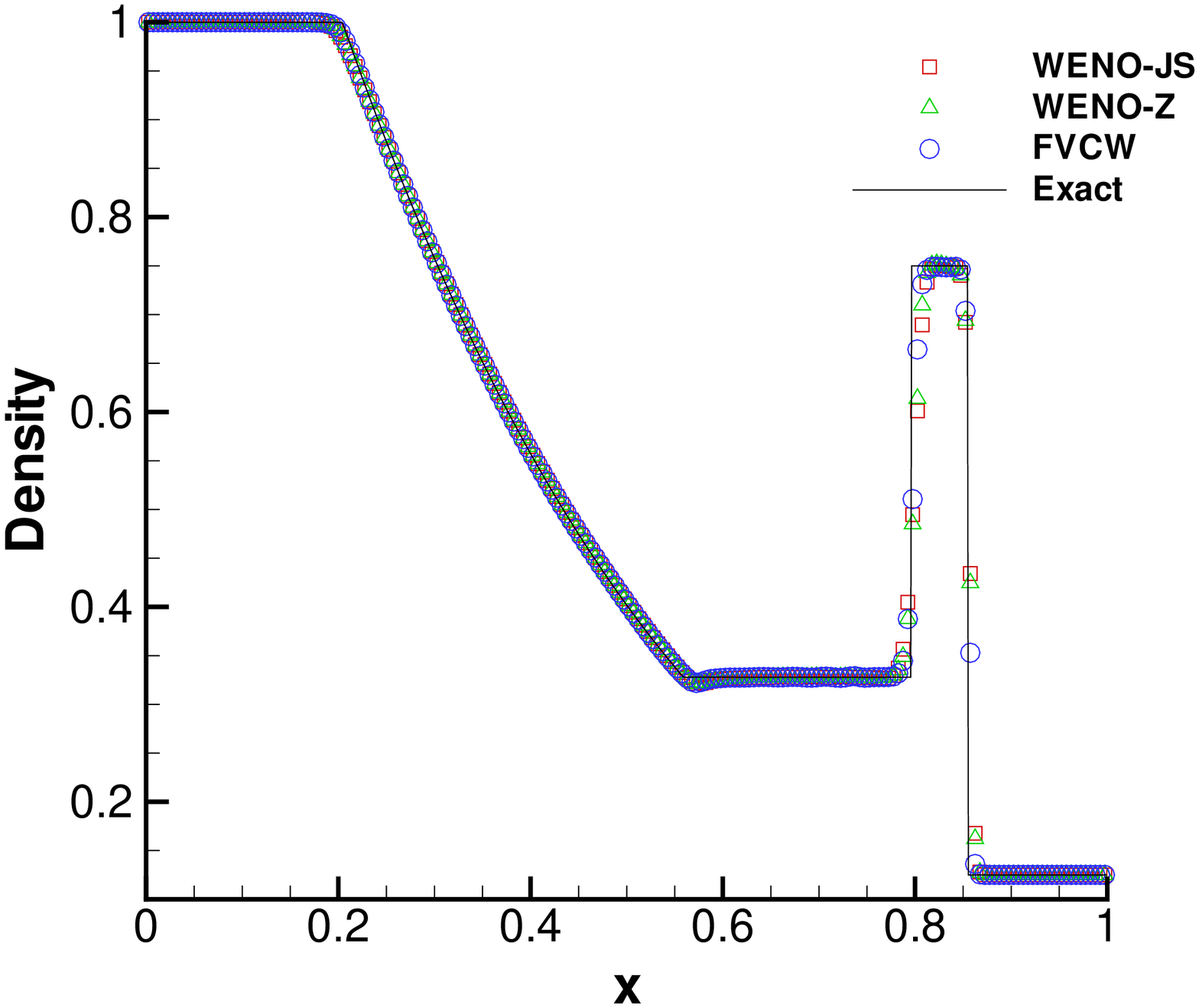}}
\subfigure[Velocity]{\includegraphics[width=0.5\textwidth]{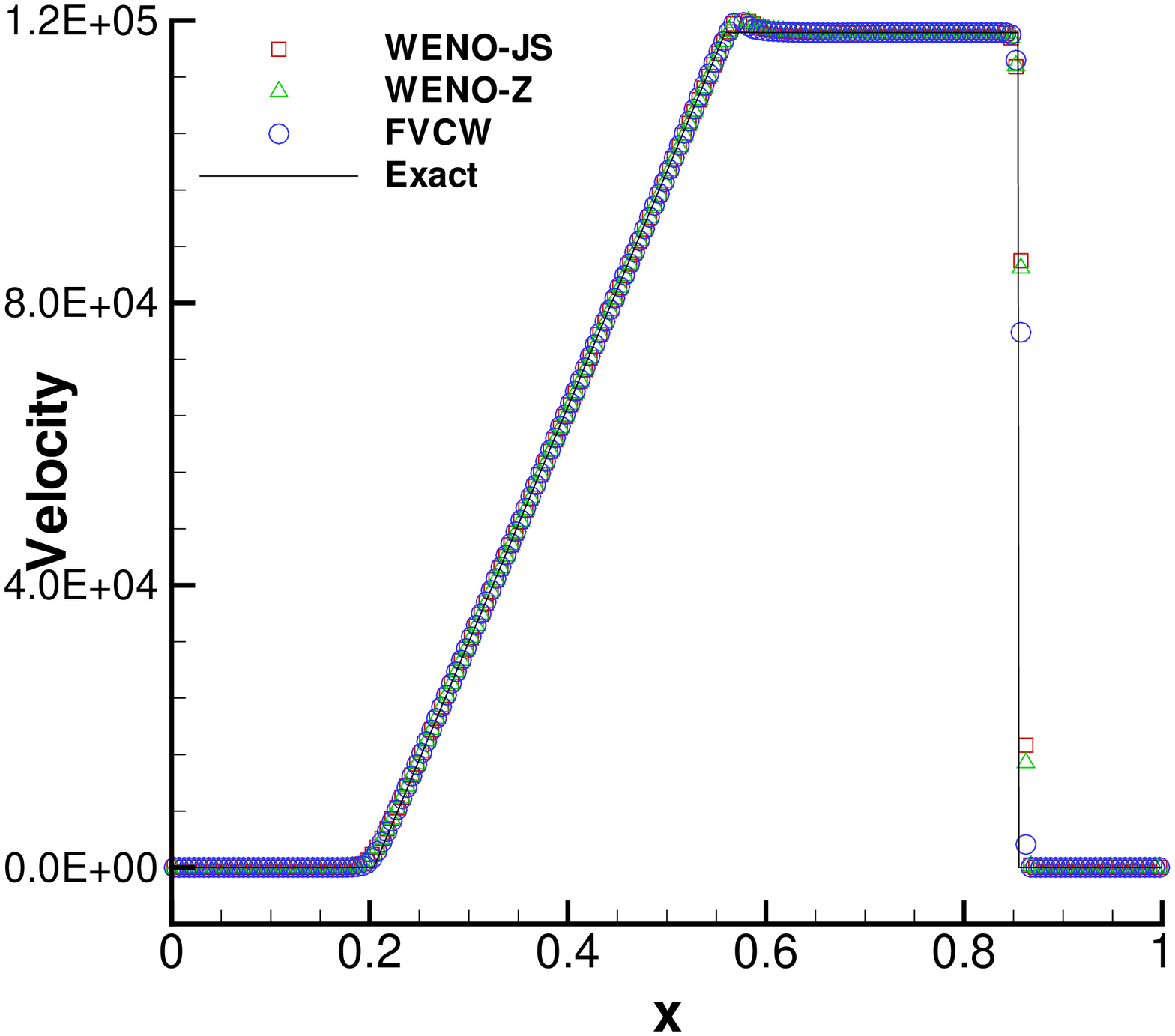}}
\subfigure[Pressure]{\includegraphics[width=0.5\textwidth]{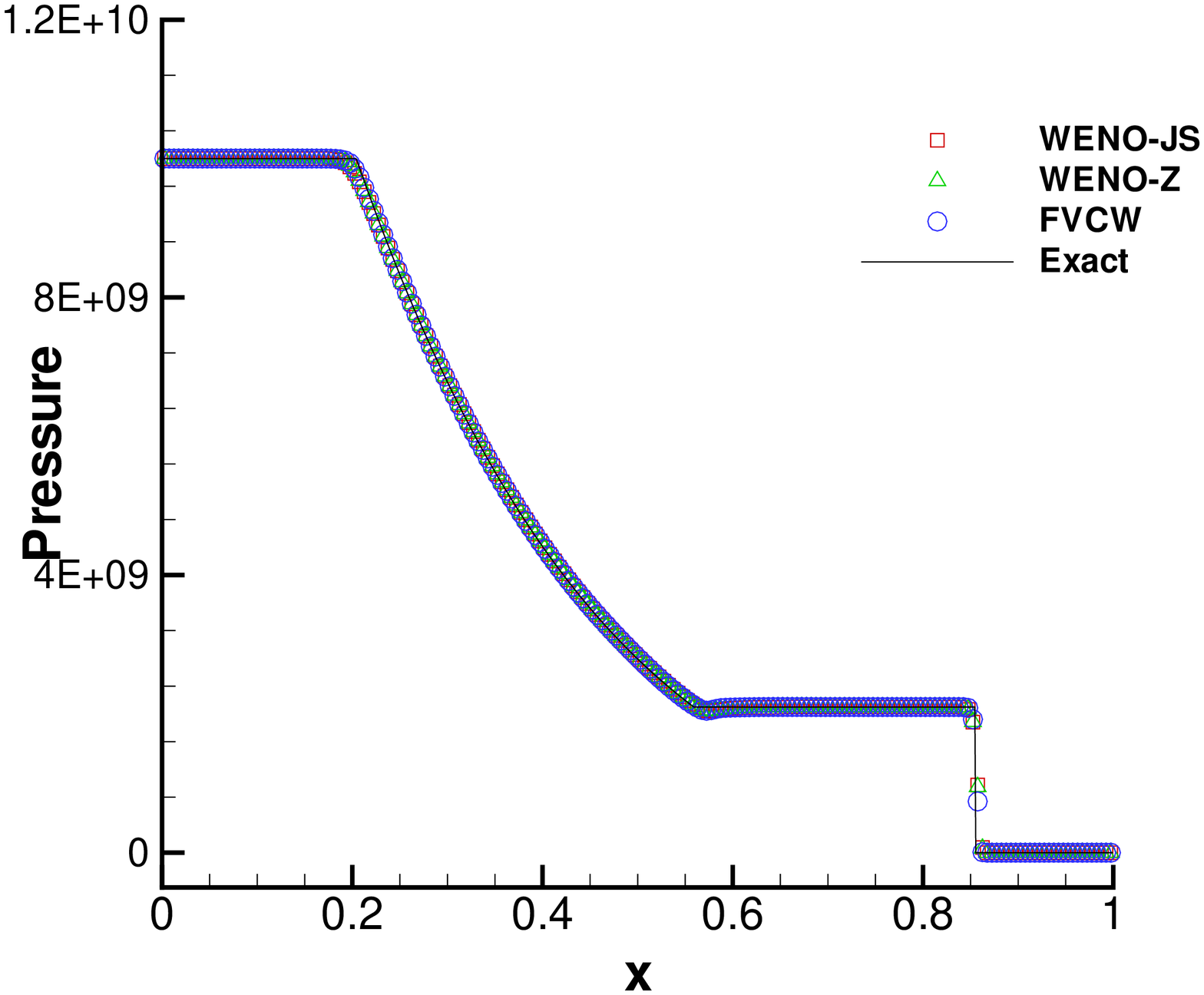}}
\subfigure[Internal
energy]{\includegraphics[width=0.5\textwidth]{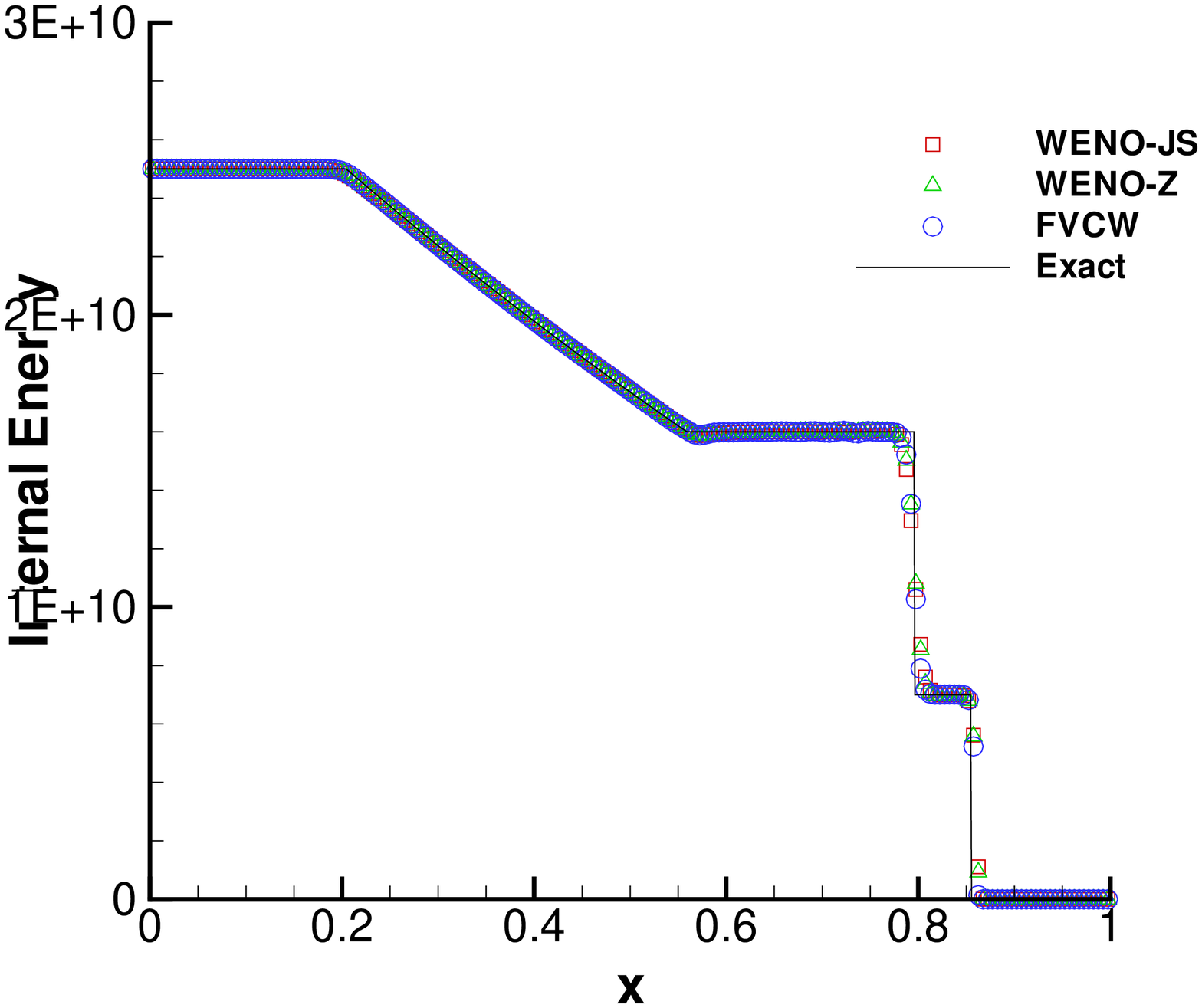}}
\caption{The results of the strong shock wave problem (\ref{hshock}) with
$N=200$ at $t = 2.5\times10^{-6}$.}
\label{Fig:9}
\end{figure}
\end{example}

\begin{example}\label{Eg:vac}
This one-dimensional test problem involves vacuum or near-vacuum solutions
with the following initial conditions
\begin{equation}
(\rho,u,p)=\left\{\begin{array}{ll}
(7,-1,0.2), & \textrm{$ -1 \leq x < 0$},\\
(7,-1,0.2), & \textrm{$ 0 \leq x \leq 1$},
\end{array}\right.
\label{doubler}
\end{equation}
with $h=0.005$ and the final time is $t=0.6$. The computed pressure,
density and velocity distributions are show in Fig. \ref{Fig:10}
(left). For this double rarefaction problem, the present FVCW scheme
with the HLLC flux has comparable results as those in Zhang and Shu
\cite{zhang2012positivity} (see their Fig. 5.1 (left)).  
The minimum numerical values of the density and the pressure are small positive values of
$2.120E-04$ and $2.201E-04$ respectively. For this problem with
vacuum or near-vacuum solutions, some oscillations can also be observed which
might be due to the same reason as described in Example \ref{Eg:123}.

\begin{figure}
\subfigure[Density]{\includegraphics[width=0.5\textwidth]{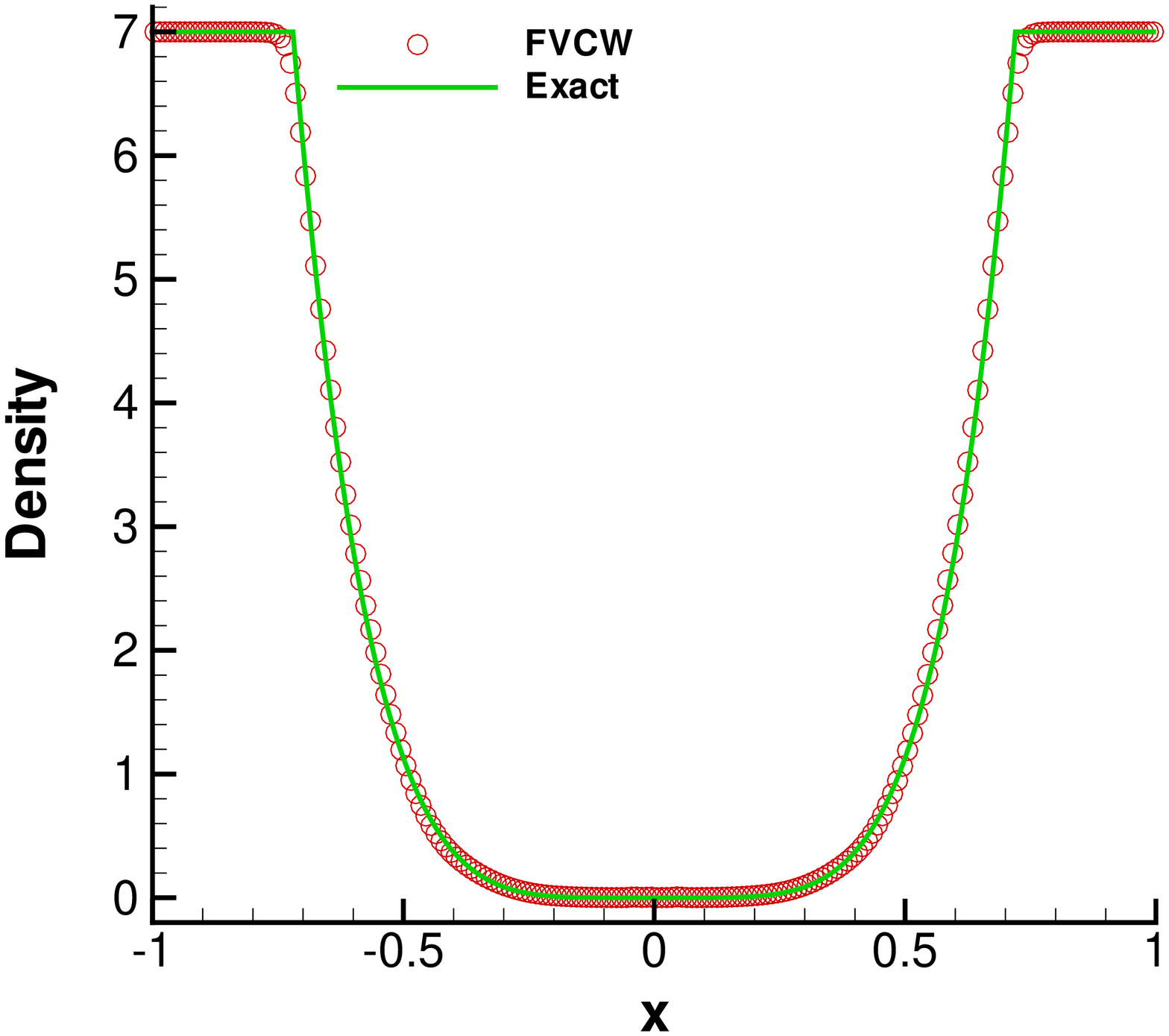}}
\subfigure[Density]{\includegraphics[width=0.5\textwidth]{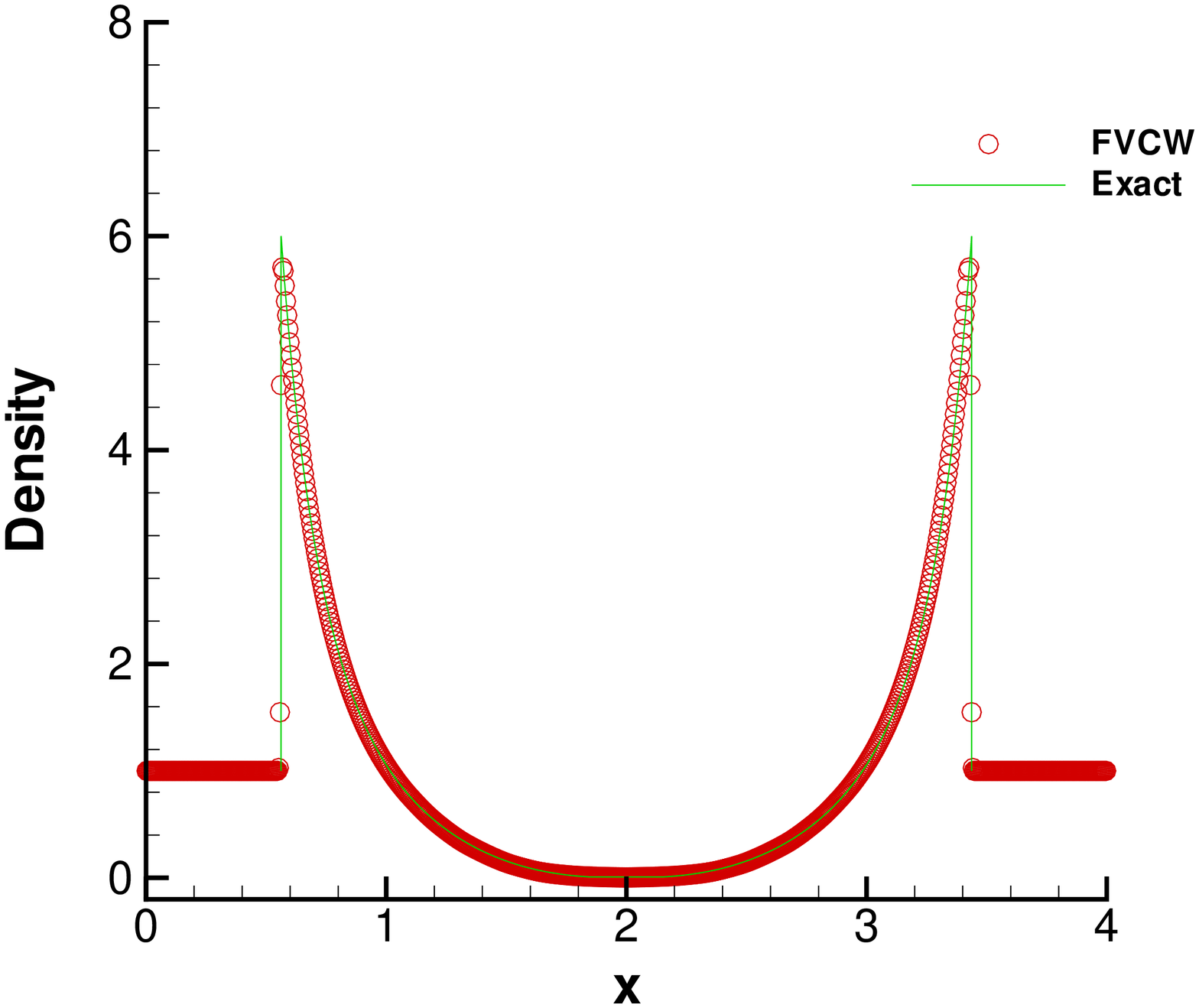}}
\subfigure[Pressure]{\includegraphics[width=0.5\textwidth]{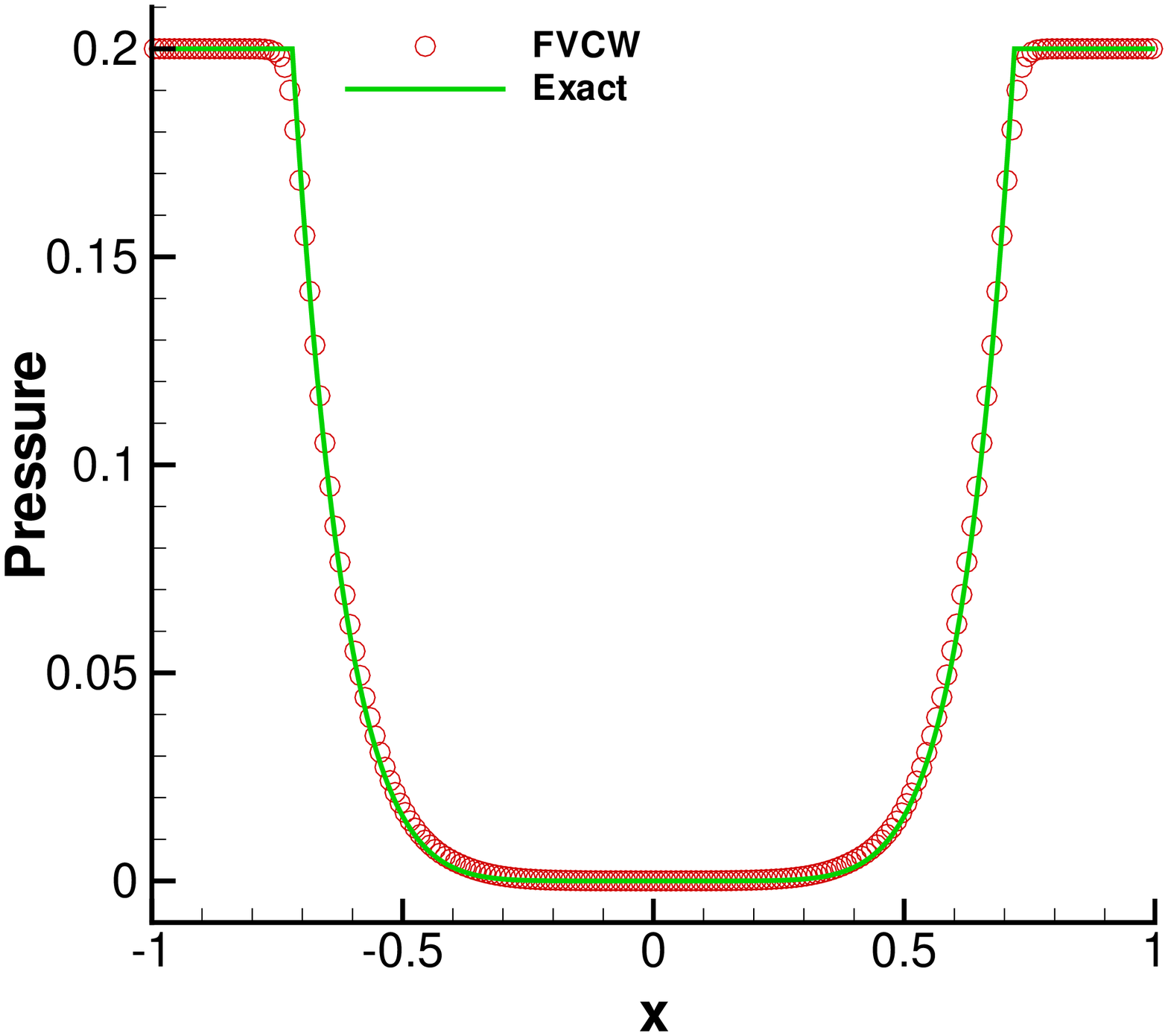}}
\subfigure[Pressure]{\includegraphics[width=0.5\textwidth]{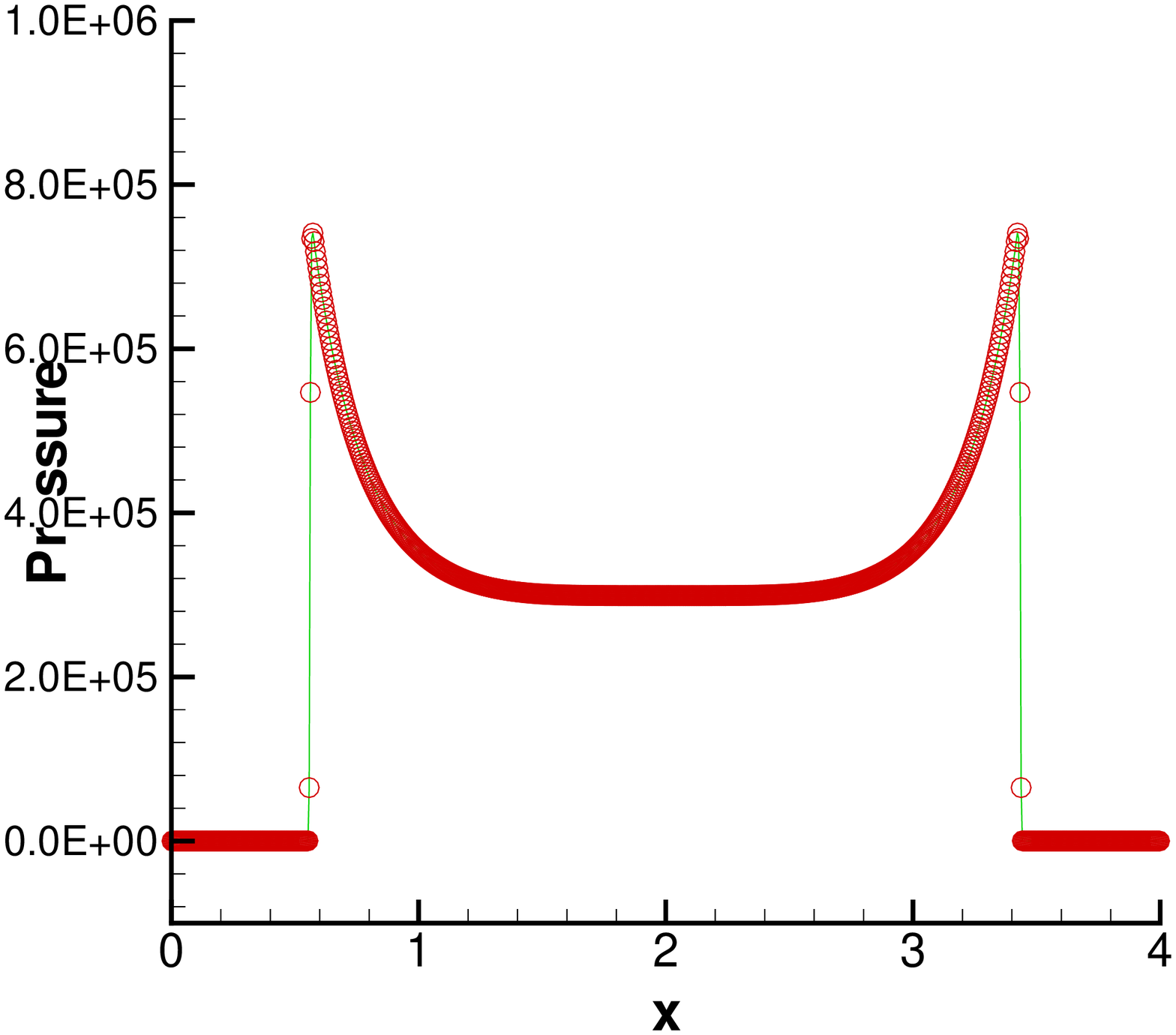}}
\subfigure[Velocity]{\includegraphics[width=0.5\textwidth]{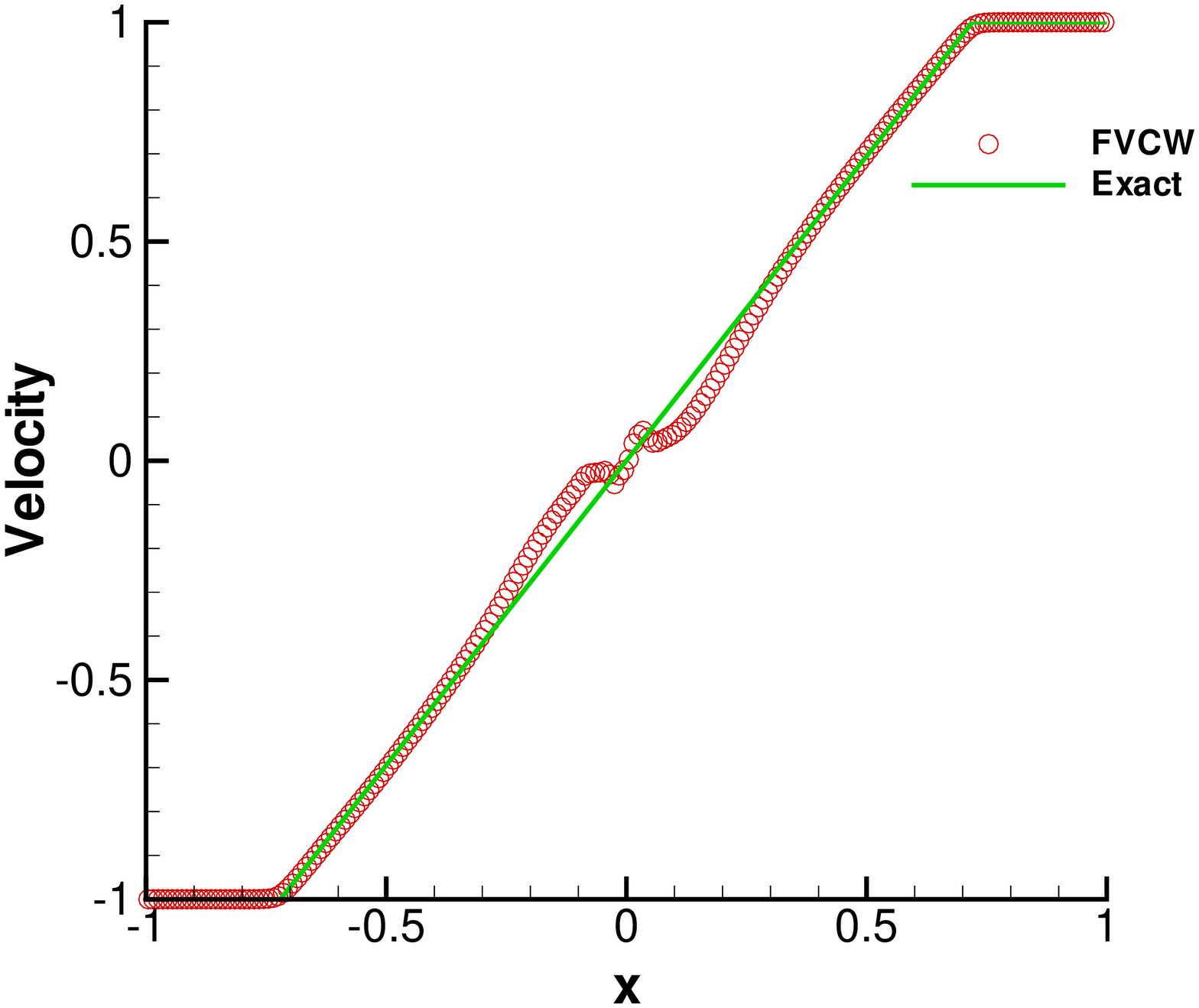}}
\subfigure[Velocity]{\includegraphics[width=0.5\textwidth]{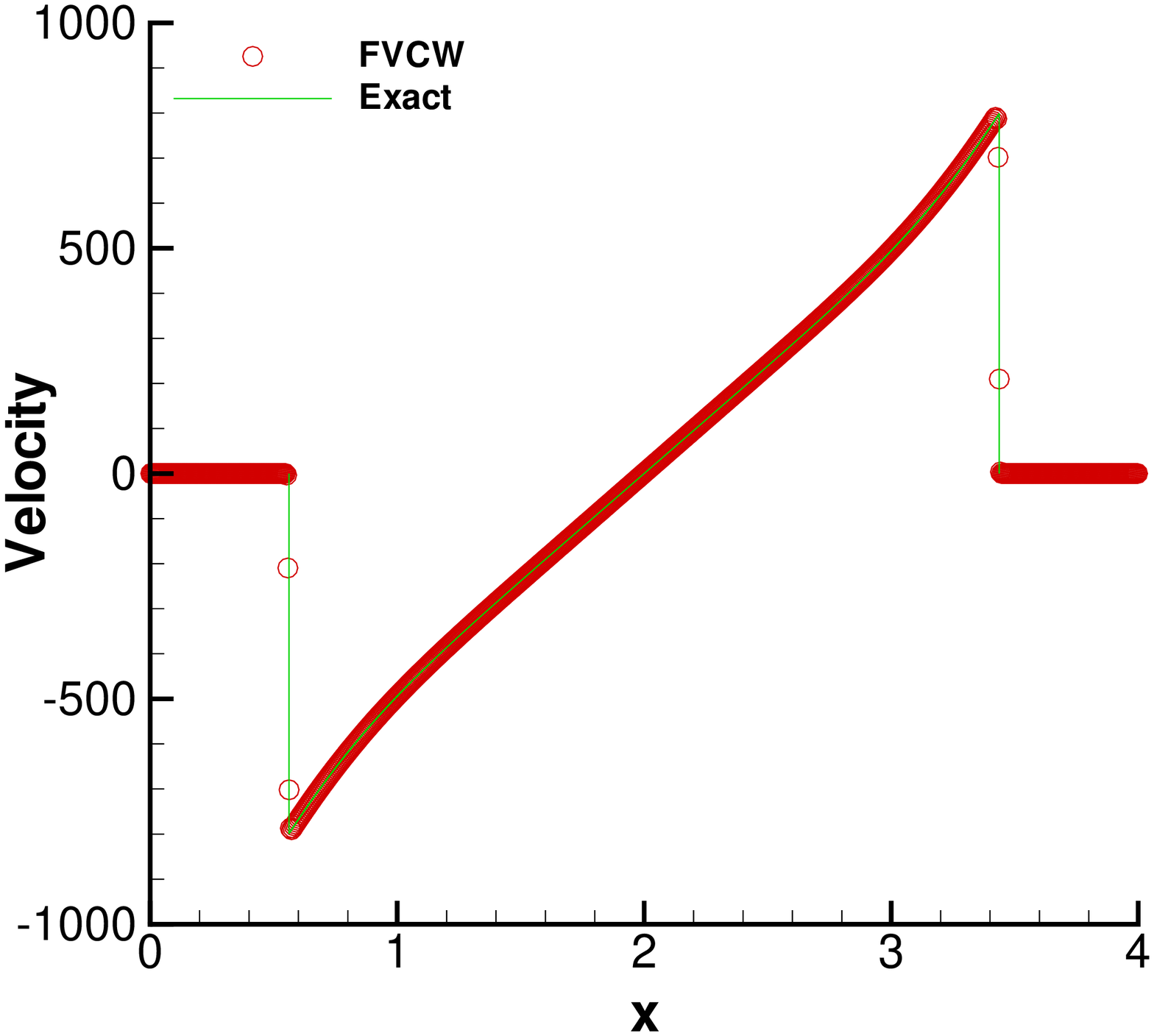}}
\caption{One-dimensional problems involving vacuum or near vacuum, $h=0.005$:
 (left) double rarefaction problem (\ref{doubler}) at $t=0.6$; (right) planar Sedov blast-wave problem (\ref{sedov}) at $t=0.001$.}\label{Fig:10}
\end{figure}
\end{example}

\begin{example}\label{Eg:sbv}
This one-dimensional test problem is the planar Sedov blast-wave
problem with the following initial conditions
\begin{equation}
(\rho,u,p)=\left\{\begin{array}{ll}
(1,0,4\times 10^{-13}),  & \textrm{$ 0 < x \leq 2-0.5 h$, $ 2+0.5 h < x < 4$},\\
(1,0,2.56\times 10^{8}), & \textrm{$ 2-0.5 h< x \leq 2+0.5 h$},
\end{array}\right.
\label{sedov}
\end{equation}
with $h=0.005$ and the final time is $t=0.001$. The numerical
results of the present positivity-preserving fifth order finite
volume compact-WENO scheme are shown in Fig.\ref{Fig:10} (right).
By comparing with Zhang and Shu \cite{zhang2012positivity} (see
their Fig. 5.1 (right)) for the planar Sedov blast-wave problem, we
can observe that a slightly sharper blast wave is obtained by using
the present FVCW scheme. The minimum numerical values of the density
and the internal energy are also small positive values of $4.731E-03$ and $1.000E-12$
respectively.
\end{example}

\begin{example}\label{Eg:lsb}
LeBlanc shock tube problem. In this extreme shock tube problem, the
computational domain is [0,9] filled with a perfect gas with
$\gamma=5/3$. The initial conditions are with high ratio of jumps
for the internal energy and density. The jump for the internal energy is $10^6$ and
the jump for the density is $10^3$. The initial conditions are given by
\begin{equation}
(\rho,u,e)=\left\{\begin{array}{ll}
(1,0,0.1), & \textrm{$ 0\leq x < 3$},\\
(0.001,0,10^{-7}), & \textrm{$3< x\leq 9$}.
\end{array}\right.
\label{leBlanc}
\end{equation}

The solution consists of a strong rarefaction wave moving to the left, a
contact discontinuity and a shock moving to the right. The
difficulty for numerical simulations of this
problem can be found in
\cite{liu2009high,Cheng2014positivity,loubere2005subcell}. Numerical
results obtained with the present FVCW schemes at $t=6.0$ with $400$
and $1000$ cells
are shown in Fig.\ref{Fig:11}. By comparing with the exact
solutions, we can observe that the present FVCW scheme preserves
positive density and internal energy, and the minimum numerical
values for density and pressure are $1.000E-03$ and $1.000E-07$
respectively. An overshoot is produced, especially for the internal
energy, however similar results are obtained in
\cite{liu2009high,Cheng2014positivity,loubere2005subcell}.
Fig.\ref{Fig:11} shows that the numerical solution is greatly
improved as the mesh is refined.
\begin{figure}
\subfigure[Velocity]{\includegraphics[width=0.5\textwidth]{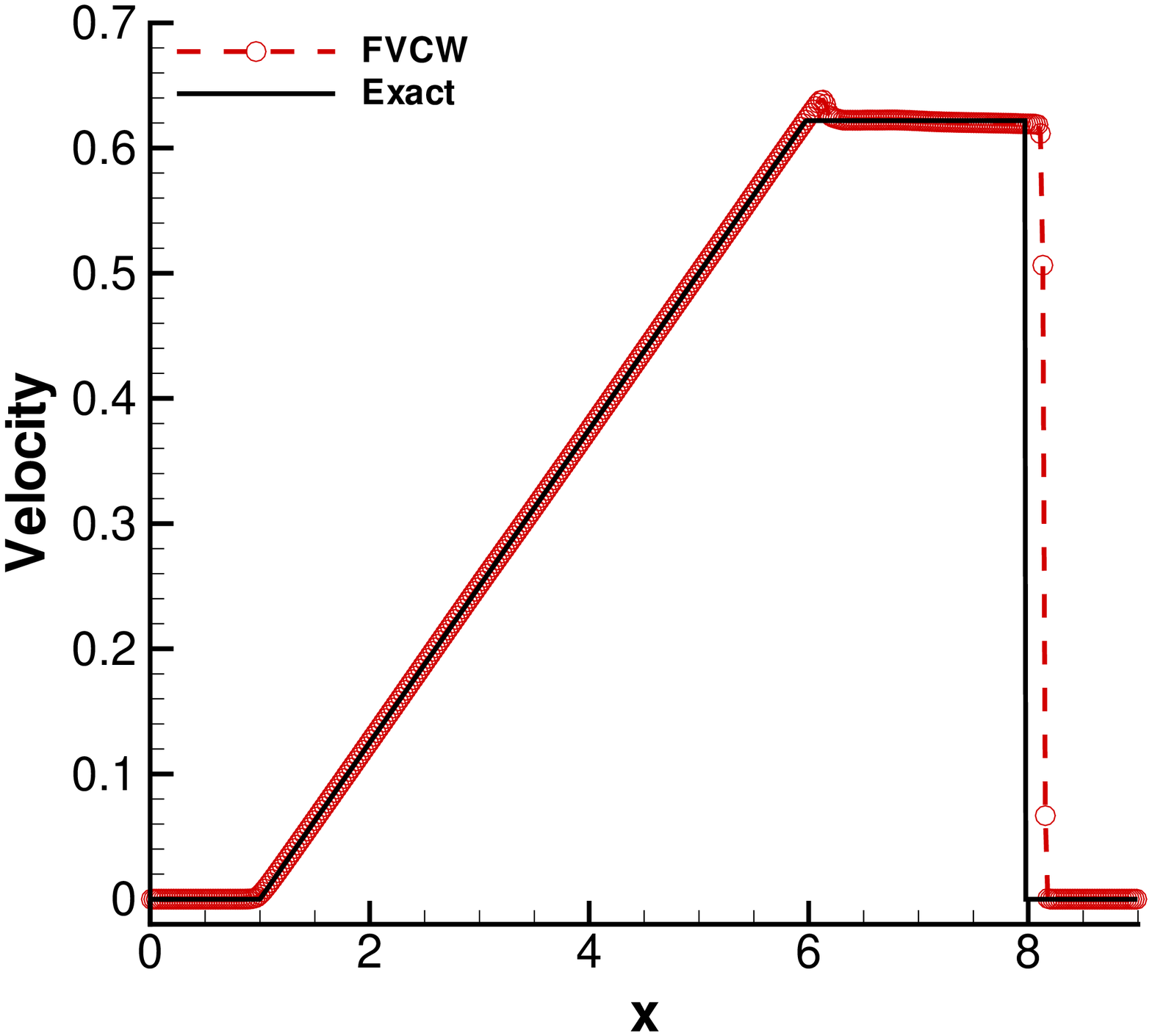}}
\subfigure[Velocity]{\includegraphics[width=0.5\textwidth]{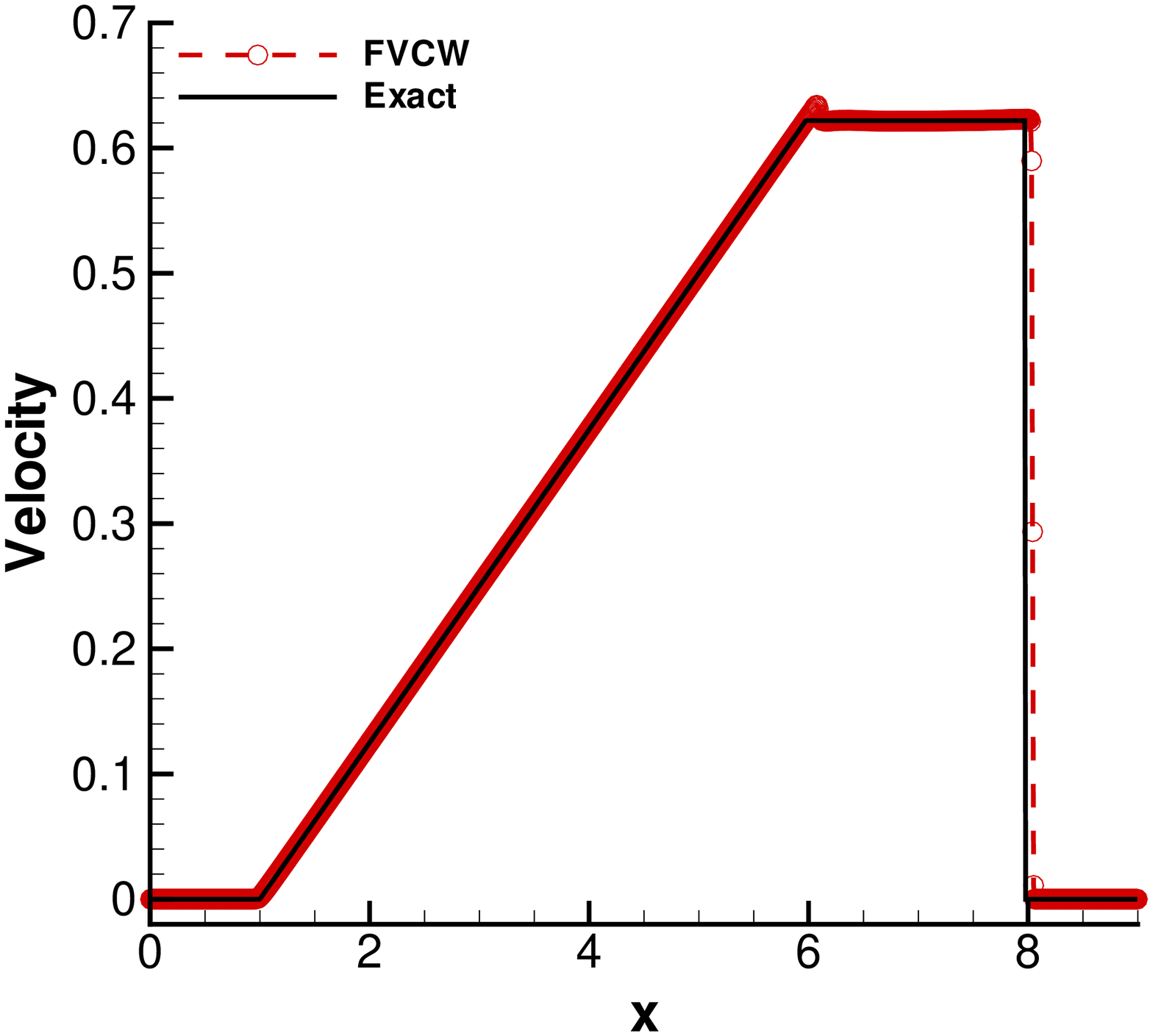}}
\subfigure[Internal
energy]{\includegraphics[width=0.5\textwidth]{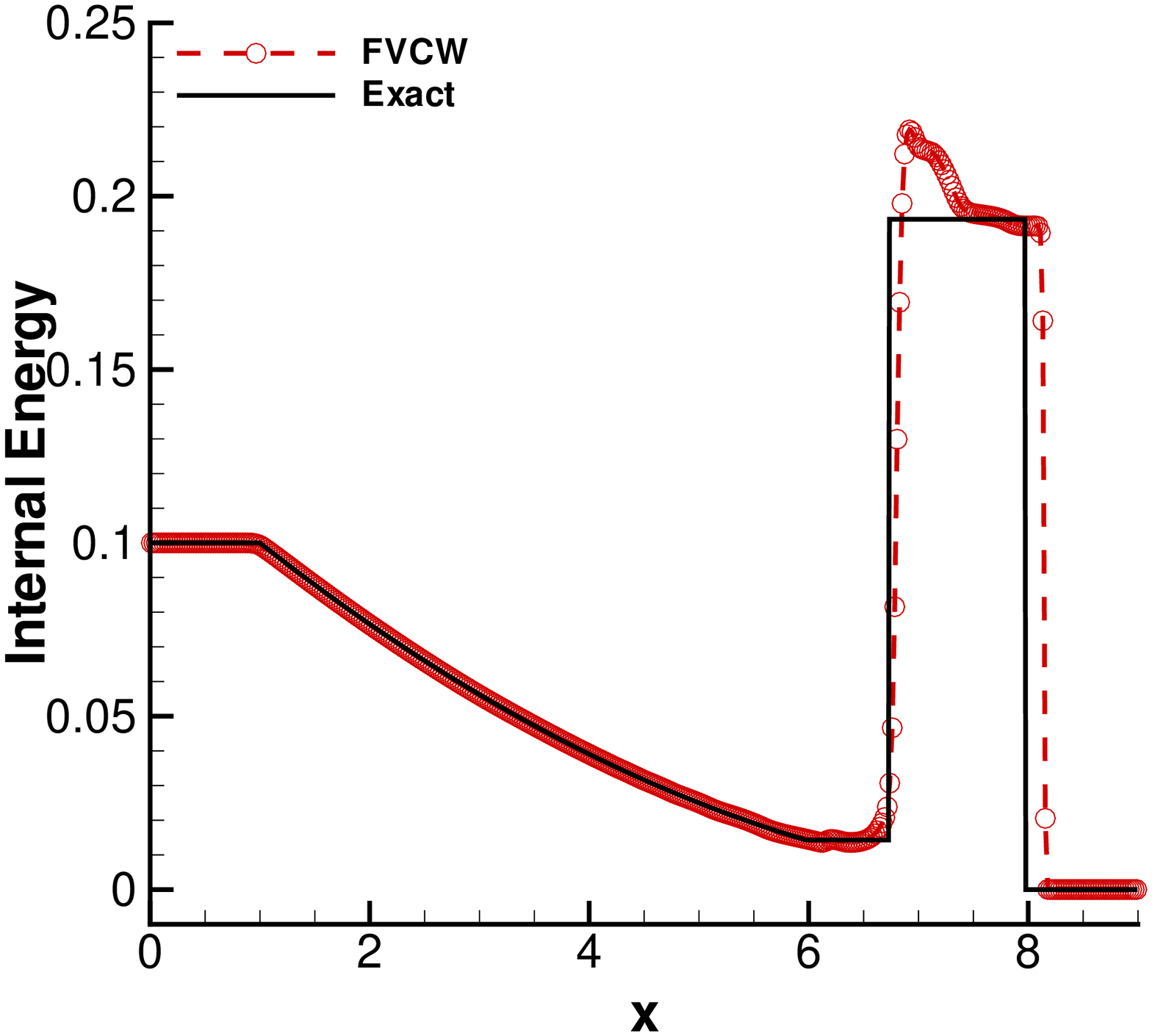}}
\subfigure[Internal
energy]{\includegraphics[width=0.5\textwidth]{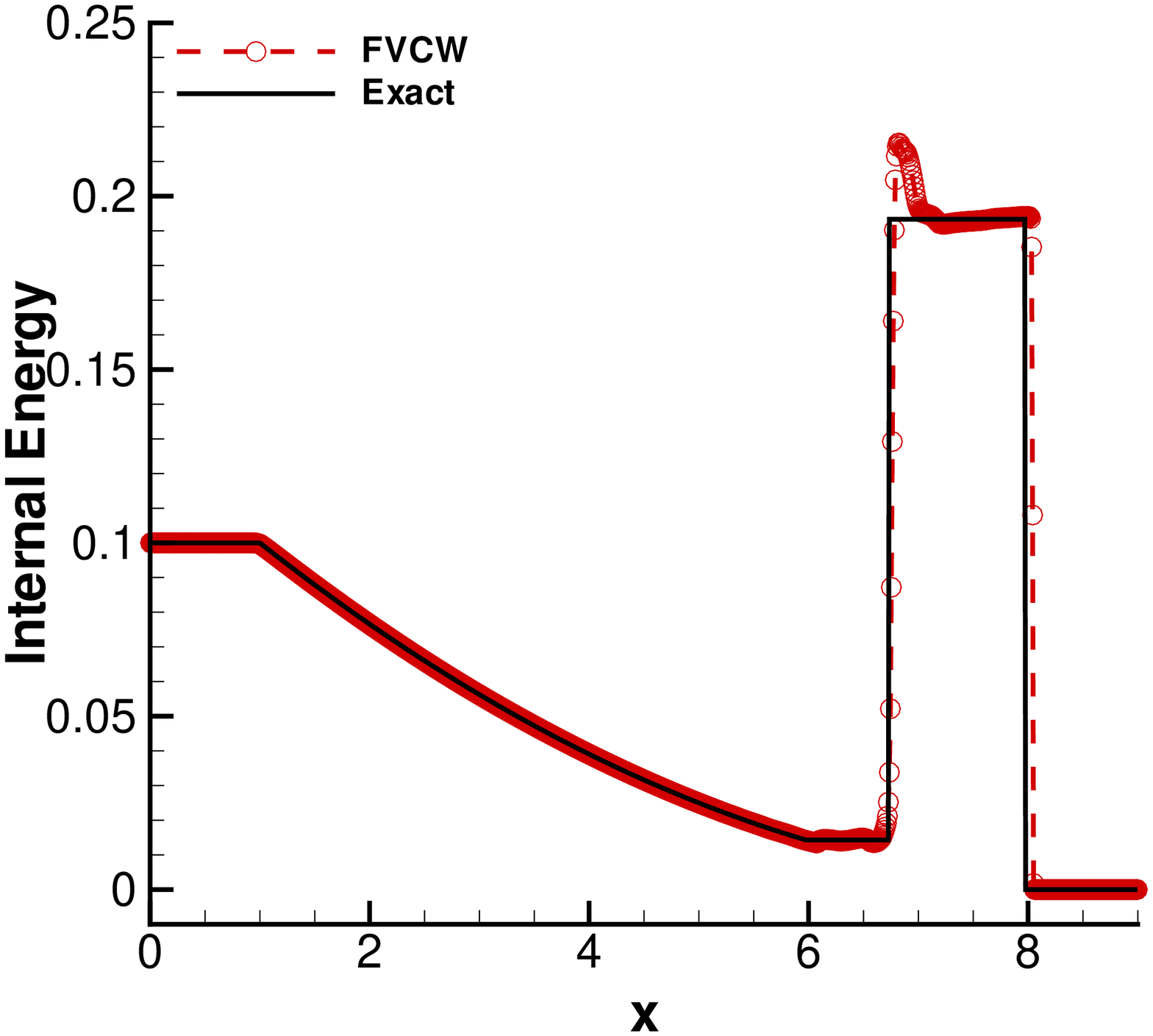}}
\caption{The results of the Leblanc problem (\ref{leBlanc}) at $t=6.0$. $N=400$
(left), $N=1000$ (right).}\label{Fig:11}
\end{figure}
\end{example}

\section{Conclusions}
\setcounter{equation}{0}
\setcounter{figure}{0}
\setcounter{table}{0}

In this paper, we have developed a positivity-preserving fifth-order finite
volume compact-WENO scheme for compressible Euler equations in one dimension. 
Compared to finite difference compact-reconstruction WENO schemes proposed by Ghosh et. al. \cite{ghosh2012compact}, the positivity-preserving limiter is used to
preserve positive density and internal energy under a finite volume framework. An approximate HLLC
Riemann solver is used due to its less dissipation and robustness. The present scheme increases spectral properties of the classical WENO schemes. Compared to classical fifth order finite volume
compact schemes, the present scheme keeps the essentially non-oscillatory properties
for capturing discontinuities. Numerical results have shown that the present scheme is positivity preserving, high order accurate, and can produce superior resolutions 
compared to the classical WENO schemes. Extension the FVCW scheme to multi-dimensional problems
contributes our future work.

\begin{acknowledgements}
The work was partly supported by the Fundamental Research Funds for
the Central Universities (2010QNA39, 2010LKSX02). The third author
acknowledges the funding support of this research by the Fundamental
Research Funds for the Central Universities(2012QNB07).
\end{acknowledgements}


\bibliographystyle{siam}
\bibliography{paper-euler}
%
%

\end{document}